\documentclass{amsart}
\usepackage{mathptmx, amsthm, bbm, amscd, tikz-cd, amssymb, enumerate, colonequals, xspace}

\usepackage{color}
\definecolor{chianti}{rgb}{0.6,0,0}
\definecolor{meretale}{rgb}{0,0,.6}
\definecolor{leaf}{rgb}{0,.35,0}
\usepackage[colorlinks=true, pagebackref, hyperindex, citecolor=meretale, urlcolor=leaf, linkcolor=chianti]{hyperref}

\newtheorem{theorem}{Theorem}[section]
\newtheorem{lemma}[theorem]{Lemma}
\newtheorem{corollary}[theorem]{Corollary}
\newtheorem{proposition}[theorem]{Proposition}

\theoremstyle{definition}
\newtheorem{definition}[theorem]{Definition}
\numberwithin{equation}{theorem}

\newtheorem{remarkx}[theorem]{Remark}
\newenvironment{remark}{\pushQED{\qed}\remarkx}{\popQED\endremarkx}

\def\ara{\operatorname{ara}}
\def\cptdim{\operatorname{cptdim}}
\def\height{\operatorname{ht}}
\def\lcd{\operatorname{lcd}}
\def\rad{\operatorname{rad}}

\def\ker{\operatorname{ker}}
\def\image{\operatorname{image}}
\def\Pf{\operatorname{Pf}}
\def\rank{\operatorname{rank}}
\def\sgn{\operatorname{sgn}}
\def\tr{{\operatorname{{tr}}}}
\def\Gr{\operatorname{Gr}}

\def\GL{\operatorname{GL}}
\def\SL{\operatorname{SL}}
\def\Sp{\operatorname{Sp}}
\def\Ort{\operatorname{O}}

\def\PosSymR{\operatorname{PosSymR}}

\def\Hom{\operatorname{Hom}}
\def\Spec{\operatorname{Spec}}
\def\Supp{\operatorname{Supp}}
\def\Alt{\operatorname{Alt}}
\def\pairs{\operatorname{P}}
\def\Sym{\operatorname{Sym}}
\def\US{{\operatorname{US}}}
\def\Var{\operatorname{Var}}

\def\fraka{\mathfrak{a}}
\def\frakb{\mathfrak{b}}
\def\frakm{\mathfrak{m}}

\def\frakp{\mathfrak{p}}

\def\frakA{\mathfrak{A}}
\def\frakB{\mathfrak{B}}
\def\frakC{\mathfrak{C}}
\def\frakP{\mathfrak{P}}
\def\frakQ{\mathfrak{Q}}
\def\frakS{\mathfrak{S}}

\def\AA{\mathbb{A}}
\def\CC{\mathbb{C}}
\def\FF{\mathbb{F}}
\def\FFp{\mathbb{F}_{\!p}}
\def\KK{\mathbb{K}}
\def\LL{\mathbb{L}}

\def\QQ{\mathbb{Q}}
\def\RR{\mathbb{R}}
\def\SS{\mathbb{S}}

\def\ZZ{\mathbb{Z}}
\def\one{\mathbbm{1}}

\def\calF{\mathcal{F}}
\def\calG{\mathcal{G}}
\def\calL{\mathcal{L}}

\def\calP{\mathcal{P}}
\def\calS{\mathcal{S}}

\def\eps{\varepsilon}
\def\ge{\geqslant}
\def\le{\leqslant}
\def\phi{\varphi}
\def\bar{\overline}
\def\tilde{\widetilde}
\def\to{\longrightarrow}
\def\mapsto{\longmapsto}
\def\into{\lhook\joinrel\longrightarrow}
\def\mapsfrom{\mathrel{\reflectbox{\ensuremath{\mapsto}}}}
\def\comp{\mathrm{c}}

\newcommand\et{\textrm{{\rm \'et}}}
\newcommand\sing{{\textrm{{\rm sing}}}}

\newcommand{\AN}{{\rm (AN)}\xspace}
\newcommand{\ET}{{\rm (ET)}\xspace}

\begin{document}
\title[The arithmetic rank of determinantal nullcones]{The arithmetic rank of determinantal nullcones}

\author{Jack Jeffries}
\address{Department of Mathematics, University of Nebraska, 203 Avery Hall, Lincoln, NE-68588, USA}
\email{jack.jeffries@unl.edu}

\author{Vaibhav Pandey}
\address{Department of Mathematics, Purdue University, 150 N University St., West Lafayette, IN~47907, USA}
\email{pandey94@purdue.edu}

\author{Anurag K. Singh}
\address{Department of Mathematics, University of Utah, 155 South 1400 East, Salt Lake City, UT~84112, USA}
\email{singh@math.utah.edu}

\author{Uli Walther}
\address{Department of Mathematics, Purdue University, 150 N University St., West Lafayette, IN~47907, USA}
\email{walther@math.purdue.edu}

\thanks{J.J. was supported by NSF CAREER award DMS 2044833, V.P. by AMS-Simons travel grant ASTG-23-284908, A.K.S. by NSF grants DMS~2101671 and DMS~2349623, and U.W. by NSF grant DMS~2100288 and Simons grant~SFI-MPS-TSM-00012928. This material is based upon work supported by the National Science Foundation under grant DMS~1928930 and by the Alfred P. Sloan Foundation under grant G-2021-16778, while the authors were in residence at the Simons Laufer Mathematical Sciences Institute (formerly MSRI) in Berkeley, California, during the Spring 2024 semester.}


\begin{abstract}
We compute the arithmetic rank as well as the local/\'etale cohomological dimension of nullcone ideals arising from the classical actions of the symplectic group, the general linear group, and the orthogonal group. We use these calculations to establish striking vanishing results for local cohomology modules supported at these nullcone ideals; this is achieved via a careful analysis of the critical local cohomology modules. The vanishing theorems that we prove are sharp in various respects.
\end{abstract}
\maketitle

\setcounter{tocdepth}{1}
\tableofcontents

\section{Introduction}

Consider a polynomial ring $S$ over a field $\KK$, and a group $G$ acting on $S$ via degree-preserving $\KK$-algebra automorphisms. By the \emph{nullcone ideal} of the action, we mean the expansion of the homogeneous maximal ideal of the invariant ring $S^G$ to the polynomial ring~$S$. The notion arises at least as far back as Hilbert's proof of the finite generation of invariant rings~\cite{Hilbert}, and has been studied extensively e.g., \cite{Hesselink, HJPS, Kraft-Schwarz, Kraft-Wallach, Lorincz:symplectic, PTW, Schwarz}. For classical invariant rings of characteristic zero, work of Kraft and Schwarz records precisely when the nullcone ideal is radical or prime~\cite[Theorem~9.1]{Kraft-Schwarz}; the positive characteristic case is settled in~\cite{HJPS}, where it is also determined precisely when the nullcone ideal is perfect, i.e., when it defines a Cohen--Macaulay ring --- for each of the classical group actions, independent of the characteristic, it turns out that the minimal primes of the nullcone ideal are perfect. The $F$-regularity property is investigated in~\cite{PTW} and~\cite{Lorincz:symplectic}.

Motivated by the Nullstellensatz, the \emph{arithmetic rank} of an ideal $\fraka$ is the least number of elements required to generate $\fraka$ up to taking radicals. This is often a notoriously difficult invariant to compute, with some innocuous looking examples remaining a challenge for over sixty years, e.g.,~\cite{Hartshorne:stci}. Our paper begins with the observation that for the action of a linearly reductive group, the arithmetic rank of the nullcone ideal is readily determined:

\begin{theorem}
\label{theorem:ara:intro}
Let $S$ be a polynomial ring over a field $\KK$, and let $G$ be a linearly reductive group acting on $S$ by degree-preserving $\KK$-algebra automorphisms. Let $S^G$ denote the ring of invariants, and $\frakm_{S^G}$ the homogeneous maximal ideal of $S^G$. Then the nullcone ideal~$\frakm_{S^G}S$ has arithmetic rank $\dim S^G$.
\end{theorem}

The proof is so elementary that we present it right away, though some notation and background is provided later in this section.

\begin{proof}
Set $R\colonequals S^G$ and $d\colonequals\dim R$. The homogeneous maximal ideal $\frakm_R$ of $R$ may be generated up to radical by $d$ elements, namely by a homogeneous system of parameters for~$R$. This gives us an upper bound for the arithmetic rank of the nullcone,
\[
\ara(\frakm_RS)\ \le\ d.
\]
Since $G$ is assumed to be linearly reductive, the inclusion $R\into S$ is pure, and so
\begin{equation}
\label{equation:theorem:ara:intro}
H^d_{\frakm_R}(R)\otimes_RS\ =\ H^d_{\frakm_R}(S)\ =\ H^d_{\frakm_RS}(S)
\end{equation}
is nonzero. But then $\ara(\frakm_RS)\ge d$, see~\cite[Proposition~9.12]{24hours}, or the discussion later in this section.
\end{proof}

Theorem~\ref{theorem:ara:intro} applies in the case of classical invariant rings of characteristic zero, i.e., when $G$ is the general linear group, the symplectic group, the orthogonal group, or the special linear group, over a field of characteristic zero, and the action is as in Weyl's book: for the general linear group, consider a direct sum of copies of the standard representation and copies of the dual; in the other cases take copies of the standard representation. The invariant rings, respectively, are determinantal rings, rings defined by Pfaffians of alternating matrices, symmetric determinantal rings, and the Pl\"ucker coordinate rings of Grassmannians. It is the nullcones of these actions that have been studied extensively in~\cite{Kraft-Schwarz, HJPS, Lorincz:symplectic, PTW}. One of the main goals of the present paper is to determine the arithmetic rank of the corresponding nullcone ideals in the case of positive characteristic; the issue is that the classical groups are typically \emph{not} linearly reductive in positive characteristic, and the inclusion $S^G\into S$ is typically no longer pure,~\cite[Theorem~1.1]{HJPS}. Indeed, the local cohomology obstruction~\eqref{equation:theorem:ara:intro} vanishes, and the lower bound on arithmetic rank is instead obtained using \'etale cohomology.

The case of the special linear group---with the invariant rings being the Pl\"ucker coordinate rings of Grassmannians---is subsumed by the work of Bruns and Schw\"anzl: Let $S\colonequals\KK[Y]$ be a polynomial ring, where $Y$ is a $t\times n$ matrix of indeterminates. Consider the action of $G\colonequals\SL_t(\KK)$ on $S$, where
\[
M\colon Y\mapsto MY\quad\text{ for }\ M\in\SL_{t}(\KK).
\]
When $\KK$ is infinite, the invariant ring $R\colonequals S^G$ is the $\KK$-algebra generated by the size $t$ minors of $Y$, so the nullcone ideal $\frakm_RS$ is the determinantal ideal $I_t(Y)S$. By \cite[Theorem~2]{Bruns-Schwanzl}, if $t\le n$, the arithmetic rank of this ideal is
\[
\ara(\frakm_RS)\ =\ nt-t^2+1\ =\ \dim R.
\]
The following theorem summarizes our results on the arithmetic rank of the nullcone ideals~$\frakm_RS\subseteq S$ for the other classical invariant rings; the rings $R$ in cases (a), (b), (c), are, respectively, Pfaffian determinantal rings, determinantal rings, and symmetric determinantal rings. When the field $\KK$ is infinite, these arise as invariant rings for the actions of the symplectic group, the general linear group, and the orthogonal group, described earlier.

\begin{theorem}
Let $\KK$ be a field of characteristic other than two, and let $R\subseteq S$ denote one of the following inclusions:
\begin{enumerate}[\quad\rm(a)]
\item $\KK[Y^\tr \Omega Y] \subseteq \KK[Y]$, where $Y$ is a $2t\times n$ matrix of indeterminates, and $\Omega$ is as in~\eqref{equation:omega};

\item $\KK[YZ] \subseteq \KK[Y,Z]$, where $Y$ and $Z$ are $m\times t$ and $t\times n$ matrices of indeterminates;

\item $\KK[Y^\tr Y] \subseteq \KK[Y]$, where $Y$ is a $t\times n$ matrix of indeterminates.
\end{enumerate}
Let $\frakm_R$ denote the homogeneous maximal ideal of $R$. Then, in each of the cases above, the nullcone ideal~$\frakm_RS$ has arithmetic rank $\dim R$.

Let $\frakm_S$ denote the homogeneous maximal ideal of the polynomial ring $S$. If $\KK$ has characteristic zero, then there exists a degree-preserving $S$-module isomorphism
\[
H^{\dim R}_{\frakm_RS}(S)\ \cong\ H^{\dim S}_{\frakm_S}(S),
\]
provided that $2t+1\le n$ in case \textrm{(a)}, $1<t<\min\{m,n\}$ in case \textrm{(b)}, or $3\le t\le n$ in case \textrm{(c)}.
\end{theorem}

While the assumption that the characteristic of $\KK$ differs from two is primarily to allow for calculations of \'etale cohomology with coefficients in $\ZZ/2$, the arithmetic rank of $\frakm_RS$ may indeed differ in characteristic two, see Remark~\ref{remark:characteristic:two}. Sections~\ref{section:local:pfaffian},~\ref{section:local:determinantal},~and~\ref{section:local:symmetric} summarize our results for the respective nullcones. In each case, we obtain the arithmetic rank of the nullcone ideal, and also study the critical local cohomology module. We prove vanishing theorems for local cohomology modules supported at nullcone ideals that mirror corresponding results for determinantal ideals obtained in \cite[Theorem~1.1]{LSW}. For example, in the Pfaffian case, this involves working with inclusions of the form $\ZZ[Y^\tr \Omega Y] \subseteq \ZZ[Y]$, with $\ZZ$ the ring of integers, where we prove that each local cohomology module of the form 
\begin{equation}
\label{equation:intro:lc}
H^k_{I_1(Y^\tr \Omega Y)}(\ZZ[Y])
\end{equation}
is a torsionfree $\ZZ$-module, and that it is a $\QQ$-vectorspace precisely when $k$ differs from the height of $I_1(Y^\tr \Omega Y)$; here $I_1(Y^\tr \Omega Y)$ denotes the ideal generated by the entries of the matrix $Y^\tr \Omega Y$. Moreover, if $Y$ is a~$2t\times n$ matrix of indeterminates and $n\ge 2t+1$, we prove that there exists a degree-preserving isomorphism
\[
H^c_{I_1(Y^\tr \Omega Y)}(\ZZ[Y])\ \cong\ H^{2tn}_\frakm(\QQ[Y]),
\]
where $\frakm$ denotes the homogeneous maximal ideal of $\QQ[Y]$ under the standard grading, and~$c\colonequals\binom{n}{2} - \binom{n-2t}{2}$, which is the cohomological dimension of $I_1(Y^\tr \Omega Y)$.

The reason for investigating local cohomology over the integers is that it enables base change to arbitrary commutative Noetherian rings, which then yields vanishing results such as Theorems~\ref{theorem:vanish:pfaffian},~\ref{theorem:vanish:determinantal}, and~\ref{theorem:vanish:symmetric}; the ring of integers has a canonical homomorphism to any commutative ring. In Section~\ref{section:p:torsion} we prove a powerful new result for showing that various local cohomology modules of interest are torsionfree $\ZZ$-modules or $\QQ$-vectorspaces:

\begin{theorem}
\label{theorem:vector:space}
Let $S$ be a polynomial ring in finitely many indeterminates over $\ZZ$, and~$\frakP$ a prime ideal of $S$ such that $\frakP\cap\ZZ=0$, and, for each positive prime integer $p$, the ring $S/(\frakP + pS)$ is $F$-rational. Then $H^k_\frakP(S)$ is a torsionfree $\ZZ$-module for each $k$; it is a \mbox{$\QQ$-vectorspace} whenever $k$ differs from $\height\frakP$.
\end{theorem}

This greatly simplifies the proofs of earlier results from \cite{LSW} and \cite{Pandey:veronese}; for example, the theorem applies when $S\colonequals\ZZ[X]$, for $X$ an $m\times n$ matrix of indeterminates, and $\frakP$ is the ideal generated by the minors of $X$ of a fixed size $t$; this recovers~\cite[Theorem~1.2]{LSW}. Similarly, the theorem applies when $\frakP$ is the ideal of size~$t$ minors of a symmetric matrix of indeterminates, or the ideal generated by the size~$2t$ Pfaffians of an alternating matrix of indeterminates; these results were originally proven as~\cite[Theorem~7.1 (1), (2)]{LSW} and~\cite[Theorem~6.2 (1), (2)]{LSW} respectively. Another consequence, Theorem~\ref{theorem:toric}, is new to the best of our knowledge: for $S$ a polynomial ring over the integers, and~$\frakP$ an ideal defining a normal semigroup ring, each local cohomology module of the form $H^k_\frakP(S)$ is a torsionfree $\ZZ$-module; it is a $\QQ$-vectorspace if $k$ differs from $\height\frakP$.

The singular and \'etale cohomology calculations required for the arithmetic rank results are performed in Sections~\ref{section:topology:pfaffian},~\ref{section:topology:determinantal}, and~\ref{section:topology:symmetric}, for the respective cases of Pfaffian nullcones, determinantal nullcones, and symmetric determinantal nullcones. The locally trivial fiber bundles used in these sections, along with the necessary results from linear algebra, are recorded in Appendix~\ref{appendix}. Preliminary remarks on singular and \'etale cohomology may be found in Section~\ref{section:prelim:cohomology}.

\subsection*{Definitions and notation}

The \emph{local cohomological dimension} of an ideal $\fraka$ in a Noetherian ring $R$ is
\[
\lcd\fraka\colonequals\sup\{k\in\ZZ \mid H^k_\fraka(R)\neq0\}.
\]
For $i>\lcd\fraka$, it turns out that $H^i_\fraka(M)$ vanishes for each $R$-module $M$, \cite[Theorem~9.6]{24hours}. The \emph{arithmetic rank} of $\fraka$, denoted $\ara\fraka$, is the least integer $k$ such that
\[
\rad\fraka \ =\ \rad (f_1,\dots,f_k)R
\]
for elements $f_i\in R$. Since $H^\bullet_\fraka(R)$ may be computed using a \v Cech complex on $f_1,\dots,f_k$, it follows that $H^i_\fraka(R)=0$ for $i>\ara\fraka$. Hence $\ara\fraka\ge\lcd\fraka$, and indeed this is the local cohomology obstruction used in the proof of Theorem~\ref{theorem:ara:intro}. The corresponding lower bounds for $\ara\fraka$ from singular and \'etale cohomology are recorded in the lemma below. Let $G$ be an Abelian group and~$X$ a quasiprojective variety over a field $\KK$. When $\KK$ is the complex numbers, $H^i_\sing(X,G)$ denotes the singular cohomology of~$X$ in the Euclidean topology, with coefficients in $G$. When $\KK$ is an arbitrary algebraically closed field, $H^i_\et(X,G)$ denotes the \'etale cohomology of $X$ with coefficients in $G$; see~Subsections~\ref{ssec:singular} and~\ref{ssec:etale} for more.

\begin{lemma}
\label{lemma:ara}
Let $\KK$ be a field, and let $V\colonequals\Var(f_1,\dots,f_k)$ be an algebraic set in $\KK^d$. 
\begin{enumerate}[\quad\rm(1)]
\item When $\KK$ equals $\CC$, we have
\[
H^i_\sing(\CC^d\smallsetminus V,\, \QQ)=0\quad\text{ for each }\ i>d+k-1.
\]
\item For $\KK$ an algebraically closed field, we have
\[
H^i_\et(\KK^d\smallsetminus V,\, \ZZ/\ell)=0\quad\text{ for each }\ i>d+k-1,
\]
where $\ell$ is a prime integer that is relatively prime to the characteristic of $\KK$.
\end{enumerate}
\end{lemma}

When working over the complex numbers, singular cohomology also reflects Bass numbers of local cohomology as follows:

\begin{lemma}{\cite[Theorem~3.1]{LSW}, see also~\cite[Theorem~3.7]{BBLSZ}}
\label{lemma:lsw}
Consider the polynomial ring $S\colonequals\CC[x_1,\dots,x_d]$. Let $\fraka$ be an ideal of $S$, and~$\frakm$ a maximal ideal. Suppose $k_0$ is a positive integer such that $\Supp H^k_\fraka(S)\subseteq\{\frakm\}$ for each integer $k\ge k_0$. Then, for each~$k\ge k_0$, one has an isomorphism of $S$-modules
\[
H^k_\fraka(S)\ \cong\ H^n_\frakm(S)^\mu,
\]
where $\mu$ is the rank of $H_\sing^{d+k-1}(\CC^d\smallsetminus\Var(\fraka),\,\QQ)$.
\end{lemma}

We use the convention that the binomial coefficient $\binom{i}{j}$ is zero for integers $i<j$.

\section{Integer torsion in local cohomology}
\label{section:p:torsion}

Huneke~\cite[Problem~4]{Huneke:Sundance} asks whether local cohomology modules of Noetherian rings have finitely many associated prime ideals. The answer is negative in general; counterexamples may be found in \cite{Singh:MRL, Katzman, SS:IMRN, Jeffries:IMRN, JS}. Affirmative answers include the case of regular rings of prime characteristic~\cite{HS:TAMS}, regular local and affine rings of characteristic zero~\cite{Lyubeznik:Invent}, unramified regular local rings of mixed characteristic~\cite{Lyubeznik:Comm}, and smooth $\ZZ$-algebras~\cite{BBLSZ}. While the counterexamples constructed in~\cite{Katzman} and~\cite{SS:IMRN} are for affine algebras over a field, the rest come from integer torsion in local cohomology; for example, in \cite{Singh:MRL}, it is proven that for
\[
S\colonequals\ZZ[u,v,w,x,y,z]/(ux+vy+wz),
\]
the local cohomology module $H^3_{(x,y,z)}(S)$ has a $p$-torsion element for each prime integer~$p$, and therefore has infinitely many associated prime ideals as an $S$-module. Given a \emph{finite} set of prime integers~$\calS$, there exists a polynomial ring $S$ over $\ZZ$ with a local cohomology module $H^k_\fraka(S)$ that has $p$-torsion precisely for $p\in\calS$,~\cite[Example~5.11]{SinghWalther:Bock}; the ideal $\fraka$ may be chosen as a Stanley--Reisner ideal. The following theorem enables us to control $p$-torsion, and will be a key ingredient towards proving that various local cohomology modules of interest in this paper, such as in~\eqref{equation:intro:lc}, are torsionfree $\ZZ$-modules. Since we need the case of the Gaussian integers in Section~\ref{section:local:symmetric}, we record it in the generality below:

\begin{theorem}
\label{theorem:F:rational:torsion}
Let $A$ be the ring of integers in a number field, $\pi\in A$ a nonzero prime element, and $S$ a polynomial ring in finitely many indeterminates over $A$. Let $\frakP$ be a prime ideal of $S$ such that $\pi\notin\frakP$, and the ring $S/\rad(\frakP + \pi S)$ is $F$-rational. Then, for each~$k\ge 0$, multiplication by $\pi$ on $H^k_{\frakP}(S)$ is injective.
\end{theorem}

\begin{proof}
Multiplication by $\pi$ on $S$ induces the local cohomology exact sequence
\[
\CD
H^{k-1}_\frakP(S/\pi S) @>>> H^k_\frakP(S) @>\cdot\pi>> H^k_\frakP(S) @>>> H^k_\frakP(S/\pi S) @>\delta>> H^{k+1}_\frakP(S) @>>>.
\endCD
\]
Since $S/\rad(\frakP+\pi S)$ is $F$-rational, and hence Cohen--Macaulay upon localization at each maximal ideal, \cite[Proposition~III.4.1]{PS} implies that $H^k_\frakP(S/\pi S)=0$ when $k$ differs from $\height\frakP(S/\pi S)$. Note that by our assumptions,
\[
\height\frakP(S/\pi S)\ =\ \height\frakP\ =\ \height(\frakP+\pi S)-1.
\]
It therefore suffices to show that the connecting homomorphism
\[
\CD
H^{\height\frakP}_\frakP(S/\pi S) @>\delta>> H^{\height\frakP+1}_\frakP(S)
\endCD
\]
is zero. The map $\delta$ is one of left $D_{S|A}$-modules, where $D_{S|A}$ denotes the ring of $A$-linear differential operators on $S$. Since $S/\pi S$ is an $F$-finite regular ring, and $S/\rad(\frakP + \pi S)$ is $F$-rational, \cite[Corollary~4.10]{Blickle} implies that $H^{\height\frakP}_\frakP(S/\pi S)$ is \emph{locally} a simple $D_{S/\pi S|A/\pi A}$-module; but, then, $H^{\height\frakP}_\frakP(S/\pi S)$ is a simple $D_{S/\pi S|A/\pi A}$-module. By~\cite[Lemma~2.1]{BBLSZ}, 
\[
D_{S|A}\otimes_A A/\pi A\ \cong\ D_{S/\pi S|A/\pi A}
\]
so $H^{\height\frakP}_\frakP(S/\pi S)$ is a simple $D_{S|A}$-module. It follows that $\delta$ is either zero or injective.

Suppose $\delta$ is injective. Let $\frakQ$ be a minimal prime of $\frakP+\pi S$, in which case
\[
H^{\height\frakP}_\frakP(S/\pi S)_\frakQ
\]
is nonzero, see for example~\cite[Theorem~9.3]{24hours}, so the injectivity of $\delta$ implies that
\[
{H^{\height\frakP+1}_\frakP(S)}_\frakQ\ =\ H^{\height\frakP+1}_\frakP(S_\frakQ)
\]
is nonzero as well. But $\height\frakP+1=\dim S_\frakQ$, so $H^{\height\frakP+1}_\frakP(S_\frakQ)=0$ by the Hartshorne-Lichtenbaum vanishing theorem \cite[Theorem~14.6]{24hours}, a contradiction. Hence $\delta=0$.
\end{proof}

Theorem~\ref{theorem:vector:space} is a straightforward consequence:

\begin{proof}[Proof of Theorem~\ref{theorem:vector:space}]
In view of Theorem~\ref{theorem:F:rational:torsion}, multiplication by each prime integer $p>0$ on~$S$ gives an exact sequence
\[
\CD
0 @>>> H^k_\frakP(S) @>\cdot p>> H^k_\frakP(S) @>>> H^k_\frakP(S/pS) @>>> 0.
\endCD
\]
As $\frakP(S/pS)$ is a perfect ideal, \cite[Proposition~III.4.1]{PS} gives the vanishing of $H^k_\frakP(S/pS)$ when $k$ differs from $\height\frakP(S/pS)=\height\frakP$.
\end{proof}

Another immediate consequence:

\begin{theorem}
\label{theorem:toric}
Let $S\colonequals\ZZ[x_1,\dots,x_n]$ be a polynomial ring, and $\frakP$ an ideal such that $S/\frakP$ is a normal semigroup ring over $\ZZ$. Then each local cohomology module $H^k_\frakP(S)$ is a torsionfree $\ZZ$-module; $H^k_\frakP(S)$ is a $\QQ$-vectorspace if $k$ differs from $\height\frakP$.
\end{theorem}

\begin{proof}
Let $p>0$ be a prime integer. By \cite[Proposition~1]{Hochster:toric}, the ring $S/(\frakP + pS)$ is a direct summand of a polynomial ring over $\FFp$, hence $F$-regular.
\end{proof}

Theorem~\ref{theorem:toric} applies, for example, when $\frakP\subseteq S$ is the defining ideal of a Veronese subring of a polynomial ring over $\ZZ$. This was proven earlier in~\cite{Pandey:veronese}.

\begin{proposition}
\label{prop:MV:torsion}
Let $A$ be the ring of integers in a number field, $\pi\in A$ a prime element such that $A_{(\pi)}$ is an unramified discrete valuation ring, and $S$ a polynomial ring in finitely many indeterminates over $A$. Let $\frakA$ and $\frakB$ be ideals of $S$ such that:
\begin{enumerate}[\quad \rm(1)]
\item the element $\pi$ is a nonzerodivisor on $S/(\frakA+\frakB)$;
\item for each $k$, multiplication by $\pi$ is injective on $H^k_\frakA(S)$ and also on $H^k_\frakB(S)$;
\item the ring $S/\frakQ$ is an $F$-rational domain, where $\frakQ\colonequals \frakA + \frakB + \pi S$;
\item $\dim(S_\frakQ/\frakA S_\frakQ)\ge 2$ and $\dim(S_\frakQ/\frakB S_\frakQ)\ge 2$;
\item The punctured spectra of $(S_\frakQ/\frakA S_\frakQ)^\star$ and $(S_\frakQ/\frakB S_\frakQ)^\star$ are connected, where $(-)^\star$ denotes the completion of the strict henselization of the completion.
\end{enumerate}
Then, for each $k$, multiplication by $\pi$ on $H^k_{\frakA\cap\frakB}(S)$ is injective.
\end{proposition}

\begin{proof}
Since $S$ has no $\pi$-torsion, neither does $H^0_{\frakA\cap\frakB}(S)$. Assume $k\ge 1$. One has a commutative diagram with exact rows and columns
\[
\begin{CD}
@VVV @VVV @VVV @. \\
H^{k-1}_\frakA(S/\pi S)\oplus H^{k-1}_\frakB(S/\pi S) @>\alpha>> 
H^{k-1}_{\frakA\cap\frakB}(S/\pi S) @>\beta>> H^k_\frakQ(S/\pi S) @>>> \\
@V{\delta'}VV @V{\delta}VV @V{\delta''}VV @. \\
H^k_\frakA(S)\oplus H^k_\frakB(S) @>\alpha'>> H^k_{\frakA\cap\frakB}(S) @>\beta'>> H^{k+1}_{\frakA+\frakB}(S) @>>> \\
@V{\cdot\pi}VV @V{\cdot\pi}VV @V{\cdot\pi}VV @. \\
H^k_\frakA(S)\oplus H^k_\frakB(S) @>>> H^k_{\frakA\cap\frakB}(S) @>>> H^{k+1}_{\frakA+\frakB}(S) @>>> \\
@VVV @VVV @VVV @. \\
\end{CD}
\]
where the rows arise as Mayer--Vietoris sequences, and the columns are induced by multiplication by $\pi$ on $S$; note that
\[
H^k_{\frakA+\frakB}(S/\pi S)\ =\ H^k_\frakQ(S/\pi S)
\]
for each $k$. The maps in the diagram are maps of left $D_{S|A}$-modules.

It suffices to show that $\delta$ is zero. The map $\delta'$ is zero by (2), so $\delta$ kills $\image\alpha=\ker\beta$. Hence $\delta$ factors through
\[
\image\beta\ \subseteq\ H^k_\frakQ(S/\pi S).
\]
Since $S/\frakQ$ is Cohen--Macaulay by (3), the module $H^k_\frakQ(S/\pi S)$ is nonzero only when $k$ equals $\height\frakQ(S/\pi S)$ by \cite[Proposition~III.4.1]{PS}. For the rest of the proof, set $k\colonequals\height\frakQ(S/\pi S)$.

Suppose $\delta$ is nonzero. Then $\image\beta$ is a nonzero $D_{S|A}$-submodule of $H^k_\frakQ(S/\pi S)$, but the latter is a simple $D_{S/\pi S|A/\pi A}$-module by \cite[Corollary~4.10]{Blickle} in view of (3), hence also a simple $D_{S|A}$-module. It follows that $\beta$ is surjective, so $\delta$ factors as $\delta=\gamma\beta$, where
\[
\gamma\colon H^k_\frakQ(S/\pi S)\ \to\ H^k_{\frakA\cap\frakB}(S)
\]
is a nonzero map of $D_{S|A}$-modules; note that $\gamma$ must be injective by the simplicity of $H^k_\frakQ(S/\pi S)$. But $\delta''=\beta'\gamma$ is zero by Theorem~\ref{theorem:F:rational:torsion} and~(3), so $\image\gamma$ is contained in $\ker\beta'=\image\alpha'$. The image of $\gamma$ is an isomorphic copy of $H^k_\frakQ(S/\pi S)$, and thus supported at $\frakQ$. It follows that $H^k_\frakA(S)\oplus H^k_\frakB(S)$ is supported at $\frakQ$. But (4) and (5) imply
\[
H^k_\frakA(S_\frakQ)\ =\ 0\ =\ H^k_\frakB(S_\frakQ)
\]
in view of the Second Vanishing Theorem, \cite[Theorem~1.4]{ZhangSVT}, since $S_\frakQ$ is an unramified mixed characteristic regular local ring of dimension $k+1$.
\end{proof}

\begin{remark}
Let $R$ be an excellent local ring, and $\frakA$ an ideal such that for each minimal prime $\frakP$ of $\frakA$, the ring $R/\frakP$ is a normal domain. If the punctured spectrum of $R/\frakA$ is connected, then, in the notation of Proposition~\ref{prop:MV:torsion}, the punctured spectrum of $(R/\frakA)^\star$ is connected as well: the completion of an excellent local normal ring is normal by~\cite[\href{https://stacks.math.columbia.edu/tag/0C23}{Tag 0C23}]{stacks-project}, while its strict henselization is normal by~\cite[\href{https://stacks.math.columbia.edu/tag/06DI}{Tag 06DI}]{stacks-project}; the claim is then a consequence of the flatness of $R/\frakA \to (R/\frakA)^\star$.
\end{remark}

\section{Local cohomology of Pfaffian nullcones}
\label{section:local:pfaffian}

Let $X$ be an $n\times n$ alternating matrix of indeterminates over a field $\KK$, and $\Pf_{2t+2}(X)$ the ideal of $\KK[X]$ generated by the Pfaffians of the size $2t+2$ principal submatrices of $X$; we refer to $\KK[X]/\Pf_{2t+2}(X)$ as a \emph{Pfaffian determinantal ring}. For an equivalent description, consider the $2t\times 2t$ alternating matrix
\begin{equation}
\label{equation:omega}
\Omega \colonequals \begin{bmatrix}
0 & 1 & & & & & \\
-1 & 0 & & & & & \\
& & 0 & 1 & & & \\
& & -1 & 0 & & & \\
& & & & \ddots & &\\
& & & & & 0 & 1 \\
& & & & & -1 & 0
\end{bmatrix},
\end{equation}
where the remaining entries are zero, and let $Y$ be a $2t\times n$ matrix of indeterminates over~$\KK$. In this case, $Y^\tr\Omega Y$ is an $n\times n$ alternating matrix with rank at most $2t$, so the entrywise map provides a surjective ring homomorphism
\[
\CD
\KK[X]/\Pf_{2t+2}(X) @>>> \KK[Y^\tr\Omega Y],
\endCD
\]
that one verifies is an isomorphism via a dimension count. It follows that the subring ${R\colonequals \KK[Y^\tr\Omega Y]}$ of~$S\colonequals \KK[Y]$ is isomorphic to $\KK[X]/\Pf_{2t+2}(X)$. The displayed isomorphism remains valid if the field $\KK$ is replaced by the integers $\ZZ$.

Consider the $\KK$-linear action of the symplectic group~$\Sp_{2t}(\KK)$ on $S$, where
\begin{equation}
\label{equation:action:symplectic}
M\colon Y\mapsto MY\quad\text{ for }\ M\in\Sp_{2t}(\KK).
\end{equation}
When $\KK$ is infinite, the invariant ring is precisely the subring~$R$, see~\cite{Weyl},~\cite[\S6]{DeConcini-Procesi}, or~\cite[Theorem~5.1]{Hashimoto}, with the nullcone ideal being the ideal of $S$ generated by the entries of the matrix $Y^\tr\Omega Y$.

We use $\frakP$ or~$\frakP(Y)$, as needed, to denote the ideal of $S$ generated by the entries of $Y^\tr\Omega Y$. By \cite[Theorem~6.8]{HJPS}, the ideal $\frakP$ is prime, $S/\frakP$ is Cohen--Macaulay, and
\begin{equation}
\label{equation:pfaffian:height}
\height \frakP\ =\ 
\begin{cases}
\displaystyle{\binom{n}{2}}& \text{if }\ n\le t+1,\\
nt-\displaystyle{\binom{t+1}{2}}& \text{if }\ n\ge t.\\
\end{cases}
\end{equation}
Moreover, the ring $S/\frakP$ is $F$-regular if $\KK$ has positive characteristic, and has rational singularities if $\KK$ has characteristic zero; see \cite[Theorem~3.6]{PTW} or \cite[Proposition~4.7]{Lorincz:symplectic}.

Note that $Y^\tr\Omega Y$ is an alternating matrix, so the ideal $\frakP$ has $\binom{n}{2}$ minimal generators. In the case that $n\le t+1$, it follows that $\frakP$ is generated by a regular sequence of length $\binom{n}{2}$, which, of course, is then the arithmetic rank of $\frakP$. More generally:

\begin{theorem}
\label{theorem:ara:pfaffian}
Let $Y$ be a~$2t\times n$ matrix of indeterminates over a field $\KK$ of characteristic other than two. Then the arithmetic rank of the ideal $\frakP\colonequals I_1(Y^\tr\Omega Y)$ in $\KK[Y]$ is
\[
\binom{n}{2} - \binom{n-2t}{2}.
\]
In particular, the following are equivalent:
\begin{enumerate}[\quad\rm(1)]
\item the ideal $\frakP$ is generated by a regular sequence;
\item the ideal $\frakP$ is a set theoretic complete intersection;
\item $n\le t+1$.
\end{enumerate}
\end{theorem}

\begin{proof}
Let $X$ be an $n\times n$ alternating matrix of indeterminates. The subring $R\colonequals \KK[Y^\tr\Omega Y]$ of~$S\colonequals \KK[Y]$ is isomorphic to the Pfaffian determinantal ring $\KK[X]/\Pf_{2t+2}(X)$, that has dimension
\[
c\colonequals\binom{n}{2} - \binom{n-2t}{2}.
\]
If the field $\KK$ has characteristic zero, $c$ equals $\ara\frakP$ by Theorem~\ref{theorem:ara:intro}, using the $\Sp_{2t}(\KK)$ action on $S$ as in~\eqref{equation:action:symplectic}. More generally, as in the proof of Theorem~\ref{theorem:ara:intro}, the homogeneous maximal ideal of $R$ is generated, up to radical, by~$c$ homogeneous elements, namely by a homogeneous system of parameters for~$R$. These~$c$ elements, when viewed as elements of the ring $S$, generate an ideal that has radical $\frakP$. Thus, independent of the characteristic of $\KK$, one has
\[
\ara\frakP\ \le\ c.
\]
While proving the reverse inequality, it suffices to assume that $\KK$ is algebraically closed. If~$n\ge 2t$ and the characteristic of $\KK$ is other than two, Theorem~\ref{theorem:cohomology:open:pfaffian}, in view of Lemma~\ref{lemma:ara}, yields $\ara\frakP\ge c$. Assume next that $n<2t$. Then, following our convention regarding binomial coefficients, $c=\binom{n}{2}$, and it remains to verify that this is a lower bound for $\ara\frakP$. If $n\le t$, then
\[
\ara\frakP\ \ge\ \height\frakP\ =\ \binom{n}{2}
\]
by~\eqref{equation:pfaffian:height}. Suppose for some $n$ with $t<n<2t$, the ideal $\frakP$ is generated up to radical by fewer than $\binom{n}{2}$ elements. Let $2k$ denote the largest even integer with $2k\le n$, and set
\[
Y'\colonequals\begin{bmatrix}
y_{11} & \cdots & y_{1n}\\
\vdots & & \vdots\\
y_{2k,1} & \cdots & y_{2k,n}\\
0 & \cdots & 0\\
\vdots & & \vdots\\
0 & \cdots & 0\\
\end{bmatrix},
\]
i.e., $Y'$ is the specialization of $Y$ obtained by setting the entries of the last $2t-2k$ rows to~$0$. Let $S'$ denote the corresponding specialization of $S$, which we may regard as a polynomial ring in $2k\times n$ indeterminates. Then the ideal $I_1(Y'^\tr\Omega Y')S'$ is generated up to radical by fewer than $\binom{n}{2}$ elements, contradicting what we have verified in the case of a $2k\times n$ matrix of indeterminates.

For the equivalences, note that $(3) \implies (1)$ follows from~\eqref{equation:pfaffian:height}, while $(1) \implies (2)$ is immediate. Lastly, if $n\ge t+2$, we see that $\ara\frakP>\height\frakP$, i.e., that
\[
\binom{n}{2}-\binom{n-2t}{2}\ >\ nt-\displaystyle{\binom{t+1}{2}}
\]
since
\[
\binom{n}{2}-nt+\displaystyle{\binom{t+1}{2}}\ =\ \binom{n-t}{2}\ >\ \binom{n-2t}{2}.
\qedhere
\]
\end{proof}

\begin{remark}
\label{remark:pfaffian:integers}
Working over the integers, one continues to have an isomorphism
\[
\ZZ[X]/\Pf_{2t+2}(X)\ \cong\ \ZZ[Y^\tr\Omega Y],
\]
as in the proof of Theorem~\ref{theorem:ara:pfaffian}, given by mapping the entries of the alternating matrix of indeterminates $X$ to the corresponding entries of $Y^\tr\Omega Y$. The displayed ring admits the structure of an algebra with a straightening law (ASL), see~\cite[\S6]{DeConcini-Procesi} or~\cite[\S4]{Barile}, so one obtains $\height I_1(Y^\tr\Omega Y)$ many elements that generate an ideal with radical $I_1(Y^\tr\Omega Y)$ in view of \cite[Proposition~5.10]{Bruns-Vetter}. Using the lower bound from the proof of Theorem~\ref{theorem:ara:pfaffian}, the formula for arithmetic rank continues to hold when $\KK$ is replaced by $\ZZ$, as recorded next.
\end{remark}

\begin{corollary}
\label{corollary:cd:pfaffian}
Let $Y$ be a~$2t\times n$ matrix of indeterminates over a torsionfree $\ZZ$-algebra $B$. Set $\frakP\colonequals I_1(Y^\tr\Omega Y)$ in $B[Y]$. Then
\[
\ara\frakP\ =\ \lcd\frakP\ =\ \binom{n}{2} - \binom{n-2t}{2}.
\]
\end{corollary}

\begin{proof}
Note that $\lcd\frakP\le\ara\frakP\le \binom{n}{2} - \binom{n-2t}{2}\equalscolon c$, where the second inequality uses Remark~\ref{remark:pfaffian:integers}. It remains to prove that $H^c_\frakP(B[Y])$ is nonzero. But $H^c_\frakP(\ZZ[Y])$ is a torsionfree $\ZZ$-module by Theorem~\ref{theorem:torsion:free:pfaffian}\,(1), while $B$ is a torsionfree $\ZZ$-module by hypothesis, so~$H^c_\frakP(\ZZ[Y])\otimes_\ZZ\QQ$ and $B\otimes_\ZZ\QQ$ are nonzero $\QQ$-vectorspaces. Hence their tensor product $H^c_\frakP(B[Y])\otimes_\ZZ\QQ$ is a nonzero $\QQ$-vectorspace.
\end{proof}

The next theorem is the analogue of \cite[Theorem~1.2]{LSW} for the Pfaffian nullcone ideals considered in this section.

\begin{theorem}
\label{theorem:torsion:free:pfaffian}
Let $Y$ be a~$2t\times n$ matrix of indeterminates; set $\frakP\colonequals I_1(Y^\tr\Omega Y)$ in the polynomial ring $\ZZ[Y]$. Then:
\begin{enumerate}[\quad\rm(1)]
\item $H^k_\frakP(\ZZ[Y])$ is a torsionfree $\ZZ$-module for each integer $k$.
\item If $k$ differs from the height of $\frakP$, then $H^k_\frakP(\ZZ[Y])$ is a $\QQ$-vectorspace.
\item Set $c\colonequals\binom{n}{2} - \binom{n-2t}{2}$, which is the cohomological dimension of $\frakP$. If $n\ge 2t+1$, then one has a degree-preserving isomorphism
\[
H^c_\frakP(\ZZ[Y])\ \cong\ H^{2tn}_\frakm(\QQ[Y]),
\]
where $\frakm$ is the homogeneous maximal ideal of $\QQ[Y]$ under the standard grading.
\end{enumerate}
\end{theorem}

\begin{proof}
In view of the $F$-regularity of $S/(\frakP+pS)$ mentioned earlier, (1) and (2) follow from Theorem~\ref{theorem:vector:space}. It remains to prove (3). Since $n\ge 2t+1$, the equivalent conditions in Theorem~\ref{theorem:ara:pfaffian} give~$c>\height\frakP$, so $H^c_\frakP(\ZZ[Y])$ is indeed a nonzero $\QQ$-vectorspace. We change notation and work with $S\colonequals\QQ[Y]$ for the remainder of the proof. We claim that the support of $H^c_\frakP(S)$ is the homogeneous maximal ideal $\frakm$ of~$S$, for which it suffices, without loss of generality, to verify that $H^c_\frakP(S)_{y_{11}} = H^c_\frakP(S_{y_{11}})$ is zero:

By \cite[Lemma~3.2]{PTW}, there exists a $(2t-2)\times(n-1)$ matrix $Y'$ with entries from $S_{y_{11}}$, and elements $f_2,\dots,f_n$ in $S_{y_{11}}$, such that the elements $Y', f_2,\dots,f_n$ are algebraically independent over $\QQ$, and
\[
\frakP S_{y_{11}}\ =\ \frakP(Y')S_{y_{11}}+(f_2,\dots,f_n)S_{y_{11}},
\]
and $S_{y_{11}}$ is free over the polynomial subring $\QQ[Y',f_2\dots,f_n]$. It follows that $H^c_\frakP(S_{y_{11}})$ is a direct sum of copies of
\[
H^c_{\frakP(Y')+(f_2,\dots,f_n)}(\QQ[Y',f_2,\dots,f_n])\ \cong\ H^{c-n+1}_{\frakP(Y')}(\QQ[Y']) \otimes_\QQ H^{n-1}_{(f_2,\dots,f_n)}(\QQ[f_2,\dots,f_n]).
\]
But
\[
H^{c-n+1}_{\frakP(Y')}(\QQ[Y'])\ =\ 0
\]
since
\[
\ara\frakP(Y')\ =\ \binom{n-1}{2}-\binom{(n-1)-(2t-2)}{2}\ <\ c-n+1.
\]

Note that $H^c_\frakP(S)$ is a holonomic $D_{S|\QQ}$-module, \cite[Section~2]{Lyubeznik:Invent} or \cite[Lecture~23]{24hours}. Since it has support~$\{\frakm\}$, it is isomorphic, as a $D_{S|\QQ}$-module, to a finite direct sum of copies of $H^{2tn}_\frakm(S)$, as follows from \cite[Proposition~4.3]{Kashiwara} or~\cite[Lemma~(c),~page~208]{Lyubeznik:injective}. This isomorphism is degree-preserving by \cite[Theorem~1.1]{Ma-Zhang}, see also~\cite[Section~3.2]{BBLSZ}. It remains to check that $H^c_\frakP(S)$ is isomorphic to \emph{one} copy of $H^{2tn}_\frakm(S)$. For this, replace $\QQ$ by~$\CC$ and use Lemma~\ref{lemma:lsw} and Theorem~\ref{theorem:cohomology:open:pfaffian}.
\end{proof}

We next prove a vanishing theorem analogous to~\cite[Theorem~1.1]{LSW}:

\begin{theorem}
\label{theorem:vanish:pfaffian}
Let $M=(m_{ij})$ be a $2t\times n$ matrix with entries from a commutative Noetherian ring~$A$, where $n\ge 2t+1$. Set $\frakp\colonequals I_1(M^\tr\Omega M)$ and $c\colonequals\binom{n}{2} - \binom{n-2t}{2}$. Then:

\begin{enumerate}[\quad\rm(1)]
\item The local cohomology module $H^c_\frakp(A)$ is a $\QQ$-vectorspace, and thus vanishes if the canonical homomorphism $\ZZ\to A$ is not injective.

\item If $\dim A\otimes_\ZZ\QQ<2tn$, then $H^c_\frakp(A)=0$, i.e., $\lcd\frakp<c$.

\item If the images of the matrix entries $m_{ij}$ in the ring $A\otimes_\ZZ\QQ$ are algebraically dependent over a field that is a subring of $A\otimes_\ZZ\QQ$, then $H^c_\frakp(A)=0$.
\end{enumerate}
\end{theorem}

\begin{proof}
For $Y$ a $2t\times n$ matrix of indeterminates and $\frakP\colonequals I_1(Y^\tr\Omega Y)$, Theorem~\ref{theorem:torsion:free:pfaffian}\,(3) gives
\[
H^c_\frakP(\ZZ[Y])\ \cong\ H^{2tn}_\frakm(\QQ[Y]).
\]
Consider $A$ as a $\ZZ[Y]$-algebra via $y_{ij}\mapsto m_{ij}$, so that $\frakP A$ equals $\frakp$. By base change using the right-exactness of $A\otimes_{\ZZ[Y]}-$, see for example,~\cite[Lemma~3.3]{LSW}, one obtains
\begin{equation}
\label{equation:iso}
H^c_\frakp(A)\ \cong\ H^{2tn}_{\frakm A}(A\otimes_\ZZ\QQ).
\end{equation}
It follows that $H^c_\frakp(A)$ is a $\QQ$-vectorspace, which settles (1).

For (2), note that $H^{2tn}_{\frakm A}(A\otimes_\ZZ\QQ)$ vanishes if $\dim A\otimes_\ZZ\QQ<2tn$. But then, by~\eqref{equation:iso},
\[
H^c_\frakp(A)\ =\ 0.
\]

For (3), suppose the matrix entries $m_{ij}$ are algebraically dependent over a field $\FF$ contained in $A\otimes_\ZZ\QQ$. Take $B$ to be the $\FF$-subalgebra of~$A\otimes_\ZZ\QQ$ generated by the images of $m_{ij}$ and consider the ideal $I_1(M^\tr\Omega M)$ in $B$. Then $\dim B<2tn$, so (2) gives $H^c_{I_1(M^\tr\Omega M)}(B)=0$. But then
\[
H^c_\frakp(A)\ \cong \ H^c_\frakp(A\otimes_\ZZ\QQ)\ \cong \ H^c_{I_1(M^\tr\Omega M)}(B)\otimes_B(A\otimes_\ZZ\QQ)
\]
vanishes as well.
\end{proof}

\begin{remark}
The bound $\lcd\frakp<c$ in Theorem~\ref{theorem:vanish:pfaffian}\,(2) is optimal: take $S\colonequals\QQ[Y]$ for~$Y$ a~$2t\times n$ matrix of indeterminates with $n\ge 2t+1$ and $A\colonequals S/y_{11}S$. Note that $\dim A<2tn$. Set $\frakp\colonequals\frakP A$. Multiplication by~$y_{11}$ on $S$ induces the exact sequence
\[
\CD
@>>>H^{c-1}_\frakp(A) @>>> H^c_{\frakP}(S) @>y_{11}>> H^c_{\frakP}(S) @>>>0,
\endCD
\]
where the vanishing on the right is by Theorem~\ref{theorem:vanish:pfaffian}\,(2). But multiplication by $y_{11}$ on $H^c_{\frakP}(S)$ has a nonzero kernel by Theorem~\ref{theorem:torsion:free:pfaffian}\,(3), so $H^{c-1}_\frakp(A)$ is nonzero, i.e., $\lcd\frakp=c-1$.

The requirement $n\ge 2t+1$ in Theorem~\ref{theorem:vanish:pfaffian} is essential: Suppose instead that $n\le 2t$. For indeterminates $y_{ij}$ over $\QQ$, set
\[
M\colonequals\begin{bmatrix}
1 & 0 & 0 & \cdots & 0\\
0 & y_{22} & y_{23} & \cdots & y_{2n}\\
\vdots & \vdots & \vdots & & \vdots\\
0 & y_{2t,2} & y_{2t,3} & \cdots & y_{2t,n}
\end{bmatrix}
\]
and $A\colonequals\QQ[M]$. Then $\dim A<2tn$, and we claim that $H^c_\frakp(A)$ is nonzero for $\frakp\colonequals I_1(M^\tr\Omega M)$. Note that $c=\binom{n}{2}$ in this case, and that
\[
\frakp\ =\ \frakP(M')+(y_{22},\dots,y_{2n}),
\]
where $M'$ is the $(2t-2)\times(n-1)$ submatrix of $M$ obtained by deleting the first column and the first two rows. But then $H^c_\frakp(A)$ is a direct sum of copies of
\[
H^{c-n+1}_{\frakP(M')}(\QQ[M']),
\]
which is nonzero by Corollary~\ref{corollary:cd:pfaffian}.

The requirement that $n\ge 2t+1$ in Theorem~\ref{theorem:torsion:free:pfaffian}\,(3) is needed as well: if $n\le 2t$, one may see that the isomorphism $H^c_\frakP(\ZZ[Y])\cong H^{2tn}_\frakm(\QQ[Y])$ does not hold by specializing $Y$ to the matrix $M$ displayed above.
\end{remark}

\section{Local cohomology of generic determinantal nullcones}
\label{section:local:determinantal}

Let $X$ be an $m \times n$ matrix of indeterminates over a field $\KK$; we use $I_{t+1}(X)$ to denote the ideal of $\KK[X]$ generated by the size $t+1$ minors of $X$. The \emph{determinantal ring} $\KK[X]/I_{t+1}(X)$ is a subring of a polynomial ring as follows: Taking $Y$ and $Z$ to be matrices of indeterminates of sizes $m \times t$ and $t \times n$ respectively, the product matrix $YZ$ has rank at most $t$, and the entrywise map provides an isomorphism
\[
\KK[X]/I_{t+1}(X)\ \to\ \KK[YZ].
\]
It follows that the subring $R\colonequals \KK[YZ]$ of~$S\colonequals \KK[Y,Z]$ is isomorphic to $\KK[X]/I_{t+1}(X)$. This isomorphism remains valid if the field $\KK$ is replaced by $\ZZ$.

Consider the $\KK$-linear action of the general linear group~$\GL_t(\KK)$ on $S$, where an element~$M$ in~$\GL_t(\KK)$ acts via
\begin{equation}
\label{equation:action:general}
M \colon \begin{cases}
Y & \mapsto YM^{-1}\\
Z & \mapsto MZ.
\end{cases}
\end{equation}
When $\KK$ is infinite, the invariant ring is precisely the subring $R$, see \cite{Weyl},~\cite[\S3]{DeConcini-Procesi}, or~\cite[Theorem~4.1]{Hashimoto}.

We use $\frakA$ to denote the ideal of $S$ generated by the entries of the product matrix $YZ$. Unlike the Pfaffian case, the ideal $\frakA$ need not be prime or even equidimensional; its irreducible components correspond to varieties of complexes that have been studied extensively, beginning with Buchsbaum--Eisenbud~\cite{Buchsbaum-Eisenbud:1975}. For the case at hand, consider a complex of~$\KK$-vectorspaces
\[
\CD
\KK^m @<M<< \KK^t @<N<< \KK^n,
\endCD
\]
and regard the matrix entries of $M,N$ as a point in affine space $\AA^{mt+tn}_\KK$. Note that
\[
\rank M + \rank N\ \le\ t.
\]
Fixing nonnegative integers $i,j$ with $i+j\le t$, the corresponding \emph{variety of complexes} is the algebraic set consisting of matrices $M,N$ with $\rank M\le i$, $\rank N\le j$, and $MN=0$. The defining ideal of this variety is
\[
\frakp_{i,j}\colonequals I_{i+1}(Y) + I_{j+1}(Z) + \frakA.
\]
The ring $S/\frakp_{i,j}$ has rational singularities if $\KK$ has characteristic zero, \cite{Kempf:BAMS, Kempf:Invent}, and is $F$-regular if $\KK$ has positive characteristic, \cite[Corollary~4.2]{Lorincz:ASENS}, \cite[Theorem~5.6]{PTW}. The ideal $\frakA$ equals the intersection of the $\frakp_{i,j}$ with $i+j=t$. If $i\le m$ and $j\le n$, then
\begin{equation}
\label{equation:height:pij}
\height \frakp_{i,j}\ =\ (m-i)(t-i) + (n-j)(t-j) + ij,
\end{equation}
see for example~\cite{Huneke:TAMS} or~\cite{DeConcini-Strickland}. Our first result in this section concerns the arithmetic rank of the ideal $\frakA$:

\begin{theorem}
\label{theorem:ara:determinantal}
Let $Y$ and $Z$ be matrices of indeterminates of sizes $m \times t$ and $t \times n$ respectively, over a field $\KK$ of characteristic other than two. Then the arithmetic rank of the ideal~$\frakA \colonequals I_1(YZ)$ in $\KK[Y,Z]$ is
\[
\ara \frakA\ =\ \begin{cases}
mt+nt-t^2 & \text{if }\ t< \min \{m,n\},\\
mn & \text{otherwise}.
\end{cases}
\]
The following are equivalent:
\begin{enumerate}[\quad\rm(1)]
\item the ideal $\frakA$ is generated by a regular sequence;
\item the ideal $\frakA$ is a set theoretic complete intersection;
\item $m+n\le t+1$.
\end{enumerate}
\end{theorem}

\begin{proof}
Let $X$ be an $m \times n$ matrix of indeterminates. The subring $R\colonequals\KK[YZ]$ of $S\colonequals\KK[Y,Z]$ is isomorphic to the generic determinantal ring $\KK[X]/I_{t+1}(X)$, and this has dimension
\[
c\colonequals
\begin{cases}
mt+nt-t^2 &\text{if }\ t < \min \{m,n\},\\
mn & \text{otherwise}.
\end{cases}
\]
If the field $\KK$ has characteristic zero, the dimension $c$ equals $\ara\frakA$ by Theorem~\ref{theorem:ara:intro}, using the $\GL_t(\KK)$ action on $S$ as in~\eqref{equation:action:general}. More generally, as in the proof of Theorem~\ref{theorem:ara:intro}, the homogeneous maximal ideal of $R$ is generated, up to radical, by a homogeneous system of parameters for~$R$, and these~$c$ elements generate an ideal of $S$ that has radical $\frakA$. Thus, independent of the characteristic of $\KK$, one has
\[
\ara\frakA\ \le\ c.
\]

Assume for the rest of the proof that the characteristic of $\KK$ is not two; we may also assume that $\KK$ is algebraically closed. If $t\le\min\{m,n\}$, Theorem~\ref{theorem:cohomology:open:generic} along with Lemma~\ref{lemma:lsw} yields $\ara\frakA\ge c$. For the remaining case, assume without loss of generality that $m\le n$, and that $t>m$. We need to verify that $c=mn$ is a lower bound for~$\ara\frakA$. Suppose $\frakA$ can be generated up to radical by fewer than $mn$ elements. Consider the specializations
\[
Y'\colonequals\begin{bmatrix}
y_{11} & \cdots & y_{1m} & 0 & \cdots & 0\\
\vdots & & \vdots & \vdots & & \vdots\\
y_{m1} & \cdots & y_{mm} & 0 & \cdots & 0\\
\end{bmatrix}
\quad\text{ and }\quad
Z'\colonequals\begin{bmatrix}
z_{11} & \cdots & z_{1n}\\
\vdots & & \vdots\\
z_{m1} & \cdots & z_{mn}\\
0 & \cdots & 0\\
\vdots & & \vdots\\
0 & \cdots & 0\\
\end{bmatrix}
\]
of $Y$ and $Z$ respectively, i.e., the entries of $Y$ beyond the first $m$ columns and the entries of $Z$ beyond the first $m$ rows are specialized to $0$. Let $S'$ denote the corresponding specialization of $S$. Then the ideal $I_1(Y'Z')S'$ is generated up to radical by fewer than $mn$ elements, contradicting what we have verified earlier.

Among the equivalent conditions, $(1) \implies (2)$ is immediate. For $(2) \implies (3)$, since the ideal $\frakA$ is a set theoretic complete intersection by assumption, minimal primes of $\frakA$ must have the same height. If $t<n$, then $\frakp_{0,\,t}$ is one of the minimal primes, so
\[
\height\frakA\ =\ \height \frakp_{0,\,t}\ =\ mt\ < mt +nt-t^2\ =\ \ara\frakA,
\]
a contradiction. Similarly, one cannot have $t<m$. Thus, $t\ge m$ and $t\ge n$; it follows that $\ara\frakA=mn$. If $t \le m+n-2$, we obtain a contradiction: the ideal $\frakp_{m-1,\,t-m+1}$ is a minimal prime of $\frakA$, but
\[
\height\frakp_{m-1,\,t-m+1}\ =\ mn+t-m-n+1\ <\ mn\ =\ \ara\frakA.
\]

It remains to prove $(3) \implies (1)$. For this, it suffices to prove that $S/\frakA$ is a complete intersection ring after specializing the entries of $t+1-(m+n)$ columns of~$Y$ and the corresponding $t+1-(m+n)$ rows of $Z$ to zero, since this leaves the number of defining equations unchanged. Thus, we may assume that $m+n=t+1$, in which case
\[
\frakA\ =\ \frakp_{m-1,\,n}\cap \frakp_{m,\,n-1}.
\]
Since $\height\frakp_{m-1,\,n}=mn=\height\frakp_{m,\,n-1}$, it follows that $\height\frakA=mn$, which is the number of generators of $\frakA$.
\end{proof}

As in Remark~\ref{remark:pfaffian:integers}, the theorem remains valid if the field $\KK$ is replaced by $\ZZ$, since one has an isomorphism
$\ZZ[X]/I_{t+1}(X)\ \cong\ \ZZ[YZ]$,
where the ring above has an ASL structure by~\cite[Chapter~4]{Bruns-Vetter}. Using this, along with Theorem~\ref{theorem:torsion:free:generic}\,(1), one obtains the following result; the proof follows that of Corollary~\ref{corollary:cd:pfaffian}.

\begin{corollary}
\label{corollary:cd:determinantal}
Let $Y$ and $Z$ be $m \times t$ and $t \times n$ matrices of indeterminates over a torsionfree $\ZZ$-algebra $B$. Set $\frakA\colonequals I_1(YZ)$ in $B[Y,Z]$. Then
\[
\ara\frakA\ =\ \lcd\frakA\ =\
\begin{cases}
mt+nt-t^2 &\text{if }\ t < \min \{m,n\},\\
mn & \text{otherwise}.
\end{cases}
\]
\end{corollary}

The answer is different in positive characteristic:

\begin{proposition}
\label{proposition:det:lcd}
Let $Y$ and $Z$ be $m\times t$ and $t\times n$ matrices of indeterminates respectively, over a field $\KK$ of positive characteristic. If $1 < t < \min\{m,n\}$, then
\[
\lcd I_1(YZ)\ <\ mt+nt-t^2.
\]
\end{proposition}

\begin{proof}
Set~$S\colonequals\KK[Y,Z]$. Since $S/\frakp_{i,j}$ is a Cohen--Macaulay ring of positive prime characteristic, one has $\lcd\frakp_{i,j} = \height\frakp_{i,j}$ by~\cite[Proposition~III.4.1]{PS}. For $\ell$ with $0\le\ell\le t$, set
\[
\fraka_\ell\colonequals\bigcap_{i=0}^\ell\frakp_{i,\,t-i}.
\]
Suppose $\ell\le t-1$. Since 
\begin{equation}
\label{equation:sum:pij}\frakp_{i_1,j_1}+\frakp_{i_2,j_2}=\frakp_{i,j} \quad \text{for} \ i\colonequals\min\{i_1,i_2\} \ \text{and} \ j\colonequals\min\{j_1,j_2\},
\end{equation} up to taking radicals, the ideal
\[
\fraka_\ell+\frakp_{\ell+1,\,t-\ell-1}
\]
coincides with
\begin{multline*}
(\frakp_{0,\,t}+\frakp_{\ell+1,\,t-\ell-1})\cap(\frakp_{1,\,t-1}+\frakp_{\ell+1,\,t-\ell-1})\cap\dots
\cap(\frakp_{\ell,\,t-\ell}+\frakp_{\ell+1,\,t-\ell-1})\\
=\
\frakp_{0,\,t-\ell-1}\cap\frakp_{1,\,t-\ell-1}\cap\dots\cap\frakp_{\ell,\,t-\ell-1}
\ =\
\frakp_{\ell,\,t-\ell-1}.
\end{multline*}
But
\[
\frakp_{\ell,\,t-\ell-1}\ \subseteq\ \fraka_\ell+\frakp_{\ell+1,\,t-\ell-1},
\]
and $\frakp_{\ell,\,t-\ell-1}$ is prime, so one has the equality
\begin{equation}
\label{equation:intersect:pij}
\frakp_{\ell,\,t-\ell-1}\ =\ \fraka_\ell+\frakp_{\ell+1,\,t-\ell-1}.
\end{equation}

We use induction on $\ell$ to prove that
\[
\lcd\fraka_\ell\ <\ mt+nt-t^2
\]
for each $\ell$ with $0\le\ell\le t$. When $\ell=0$, this is simply the verification that
\[
\height \frakp_{0,\,t}\ =\ mt\ <\ mt+nt-t^2.
\]
For the inductive step, \eqref{equation:intersect:pij} gives a Mayer--Vietoris sequence
\[
\CD
@>>> H^k_{\fraka_\ell}(S)\oplus H^k_{\frakp_{\ell+1,\,t-\ell-1}}(S) @>>> H^k_{\fraka_{\ell+1}}(S) @>>> H^{k+1}_{\frakp_{\ell,\,t-\ell-1}}(S) @>>>
\endCD
\]
which gives
\[
\lcd\fraka_{\ell+1}\ \le\ \max\left\{\lcd\fraka_\ell,\ \lcd\frakp_{\ell+1,\,t-\ell-1},\ \lcd\frakp_{\ell,\,t-\ell-1}-1 \right\}.
\]
It now suffices to verify that
\[
\height\frakp_{\ell+1,\,t-\ell-1}\ <\ mt+nt-t^2
\quad\text{ and }\quad
\height\frakp_{\ell,\,t-\ell-1}-1\ <\ mt+nt-t^2
\]
for $0\le\ell\le t-1$. In view of~\eqref{equation:height:pij}, these simplify respectively as
\[
0\ <\ (n-t)(t-\ell-1)+(m-\ell-1)(\ell+1)
\quad\text{ and }\quad
0\ <\ (n-t)(t-\ell-1)+\ell(m-\ell-1),
\]
each of which holds since $1 < t < \min\{m,n\}$.
\end{proof}

\begin{remark}
\label{remark:generic:prime:case}
Set $S\colonequals\ZZ[Y,Z]$ where $Y$ and $Z$ are $m \times t$ and $t \times n$ matrices of indeterminates, and $\frakp_{i,j}\colonequals I_{i+1}(Y) + I_{j+1}(Z) + I_1(YZ)$ in $S$. Let $p$ be a positive prime integer. Using Theorem~\ref{theorem:vector:space} and the $F$-regularity of $S/(\frakp_{i,j}+pS)$ mentioned earlier, each $H^k_{\frakp_{i,j}}(S)$ is a torsionfree $\ZZ$-module; it is a $\QQ$-vectorspace precisely when $k$ differs from $\height\frakp_{i,j}$.
\end{remark}

Corresponding to Theorem~\ref{theorem:torsion:free:pfaffian}, for determinantal nullcones one has:

\begin{theorem}
\label{theorem:torsion:free:generic}
Let $Y$ and $Z$ be $m \times t$ and $t\times n$ matrices of indeterminates respectively. Consider the ideal $\frakA\colonequals I_1(YZ)$ in $\ZZ[Y,Z]$. Then:
\begin{enumerate}[\quad\rm(1)]
\item $H^k_\frakA(\ZZ[Y,Z])$ is a torsionfree $\ZZ$-module for each integer $k$.
\item Suppose $1 < t < \min\{m,n\}$. Set $c\colonequals mt+nt-t^2$, which is the cohomological dimension of $\frakA$. Then one has a degree-preserving isomorphism
\[
H^c_\frakA(\ZZ[Y,Z])\ \cong\ H^{mt+tn}_\frakm(\QQ[Y,Z]),
\]
where $\frakm$ is the homogeneous maximal ideal of $\QQ[Y,Z]$ under the standard grading.
\end{enumerate}
\end{theorem}

\begin{proof}
Set $S\colonequals\ZZ[Y,Z]$ and let $p$ be a positive prime integer. As in the proof of Proposition~\ref{proposition:det:lcd}, set $\fraka_\ell\colonequals\bigcap_{i=0}^\ell\frakp_{i,\,t-i}$. We induce on $\ell$ to prove that each local cohomology of the form $H^k_{\fraka_\ell}(S)$ is a torsionfree $\ZZ$-module. The case $\ell=0$ holds by Remark~\ref{remark:generic:prime:case}. Assuming the result for $\ell\le t-1$, the result for $H^k_{\fraka_{\ell+1}}(S)$ follows using Proposition~\ref{prop:MV:torsion}: note that
\[
\fraka_{\ell+1}\ =\ \fraka_\ell\cap\frakp_{\ell+1,\,t-\ell-1},
\]
and that $S/\frakQ$ is $F$-rational where
\[
\frakQ\colonequals\fraka_\ell+\frakp_{\ell+1,\,t-\ell-1}+pS\ =\ \frakp_{\ell,\,t-\ell-1}+pS.
\]
The conditions on dimension and on the punctured spectrum are readily verified using \eqref{equation:height:pij} and \eqref{equation:sum:pij}. The case $\ell=t$ gives (1).

Next, assume that $1 < t < \min\{m,n\}$; we prove that $H^c_\frakA(S)$ is $p$-divisible. Since $c$ is the cohomological dimension of $\frakA$, one has an exact sequence
\[
\CD
@>>> H^c_\frakA(S) @>\cdot p>> H^c_\frakA(S) @>>> H^c_\frakA(S/pS) @>>>0.
\endCD
\]
But $H^c_\frakA(S/pS)=0$ by Proposition~\ref{proposition:det:lcd}, so one obtains
\[
H^c_\frakA(\ZZ[Y,Z])\ =\ H^c_\frakA(\QQ[Y,Z]).
\]

To complete the proof of (2), we change notation and work with $S \colonequals \QQ[Y,Z]$. We claim that $H^c_\frakA(S)$ has support $\{\frakm\}$; it suffices, without loss of generality, to verify that~$H^c_\frakA(S)_{y_{11}}$ vanishes. Let $Z'$ be the submatrix of $Z$ obtained by deleting the first row. By~\cite[Lemma 5.1]{PTW}, there exists an $(m-1)\times (t-1)$ matrix $Y'$ with entries from $S_{y_{11}}$, and elements $f_1,\dots,f_n\in S_{y_{11}}$ such that $Y'$, $Z'$ and $f_1,\dots,f_n$ are algebraically independent over $\QQ$, and 
\[
\frakA S_{y_{11}}\ =\ I_1(Y'Z')S_{y_{11}} + (f_1,\dots,f_n)S_{y_{11}},
\]
and the ring $S_{y_{11}}$ is a free module over the polynomial subring $\QQ[Y',Z',f_1,\dots,f_n]$. It follows that $H^c_\frakA(S)_{y_{11}} = H^c_\frakA(S_{y_{11}})$ is a direct sum of copies of
\[
H^{c-n}_{I_1(Y'Z')}(\QQ[Y',Z'])\otimes_\QQ H^n_{(f_1,\dots,f_n)}(\QQ[f_1,\dots,f_n]).
\]
But $H^{c-n}_{I_1(Y'Z')}(\QQ[Y',Z'])=0$, since
\[
\ara I_1(Y'Z')\ =\ (m-1)(t-1)+n(t-1)-(t-1)^2\ =\ c-n-m+t\ <\ c-n.
\]
As in the proof of Theorem~\ref{theorem:torsion:free:pfaffian}\,(3), $H^c_\frakA(\QQ[Y,Z])$ is a holonomic $D_{S|\QQ}$-module with support~$\{\frakm\}$, hence a direct sum of copies of $H^{mt+nt}_\frakm(S)$. To determine the number of copies, replace $\QQ$ by $\CC$ and use Lemma~\ref{lemma:lsw} and Theorem~\ref{theorem:cohomology:open:generic}.
\end{proof}

We record the analogue of Theorem~\ref{theorem:vanish:pfaffian} for determinantal nullcones:

\begin{theorem}
\label{theorem:vanish:determinantal}
Let $M$ and $N$ be $m\times t$ and $t\times n$ matrices with entries from a commutative Noetherian ring $A$, where $1<t<\min\{m,n\}$. Set $\fraka\colonequals I_1(MN)$ and $c\colonequals mt+nt-t^2$. Then:
\begin{enumerate}[\quad\rm(1)]
\item The local cohomology module $H^c_\fraka(A)$ is a $\QQ$-vectorspace, and thus vanishes if the canonical homomorphism $\ZZ \to A$ is not injective.
\item If $\dim A\otimes_\ZZ\QQ<mt+nt$, then $H^c_\fraka(A)=0$, i.e., $\lcd\fraka<c$.
\item If the images of the matrix entries in the ring $A\otimes_\ZZ\QQ$ are algebraically dependent over a field that is a subring of $A\otimes_\ZZ\QQ$, then $H^c_\fraka(A)=0$.
\end{enumerate}
\end{theorem}

\begin{proof}
For matrices of indeterminates $Y$ and $Z$, set $S\colonequals\ZZ[Y,Z]$, and regard $A$ as an $S$-algebra via $y_{ij}\mapsto m_{ij}$ and $z_{ij}\mapsto n_{ij}$, so that $\fraka$ equals $I_1(YZ)A$. By Theorem~\ref{theorem:torsion:free:generic}\,(2)
\[
H^c_{I_1(YZ)}(\ZZ[Y,Z])\ \cong\ H^{mt+tn}_\frakm(\QQ[Y,Z]),
\]
with $\frakm$ the homogeneous maximal ideal of $\QQ[Y,Z]$ and base change along $S\to A$ gives
\[
H^c_\fraka(A)\ \cong\ H^{mt+nt}_{\frakm A}(A\otimes_\ZZ\QQ),
\]
from which the assertions follow as in the proof of Theorem~\ref{theorem:vanish:pfaffian}.
\end{proof}

\begin{remark}
The bound $\lcd\fraka<c$ in Theorem~\ref{theorem:vanish:determinantal}\,(2) is optimal: take $S\colonequals\QQ[Y,Z]$ where $Y$ and $Z$ are $m \times t$ and $t \times n$ matrices of indeterminates respectively, and $t < \min\{m,n\}$. Set $\frakA\colonequals I_1(YZ)$, $A\colonequals S/y_{11}S$, and $\fraka\colonequals \frakA A$. Note that $\dim A < mt+nt$. Let $c$ be as in the above theorem. Multiplication by $y_{11}$ on $S$ induces the exact sequence
\[
\CD
@>>>H^{c-1}_\fraka(A) @>>> H^c_\frakA(S) @>y_{11}>> H^c_\frakA(S) @>>>0,
\endCD
\]
where the vanishing on the right is by Theorem~\ref{theorem:vanish:determinantal}\,(2). Multiplication by $y_{11}$ on $H^c_\frakA(S)$ has a nonzero kernel by Theorem~\ref{theorem:torsion:free:generic}\,(2), so $H^{c-1}_\fraka(A)$ is nonzero, i.e., $\lcd\fraka=c-1$.

The requirement $t <\min\{m,n\}$ in Theorem~\ref{theorem:vanish:determinantal} is needed: Suppose instead that one has $t=m$. For indeterminates $y_{ij}$ and $z_{ij}$ over $\QQ$, set
\[
M \colonequals \begin{bmatrix}
1 & 0 & \cdots & 0 \\
0 & y_{22} & \cdots & y_{2m}\\
\vdots & \vdots & & \vdots\\
0 & y_{m2} & \cdots & y_{mm}
\end{bmatrix},
\qquad
N \colonequals \begin{bmatrix}
z_{11} & z_{12} & \cdots & z_{1n} \\
z_{21} & z_{22} & \cdots & z_{2n}\\
\vdots & \vdots & & \vdots\\
z_{m1} & z_{m2} & \cdots & z_{mn}
\end{bmatrix},
\]
and $A\colonequals \QQ[M,N]$. Then $\dim A<m^2+mn$ and $c=mn$, and we shall see that $H^{mn}_{I_1(MN)}(A)$ is nonzero. Note that
\[
I_1(MN)\ =\ I_1(M'N') + (z_{11},\dots,z_{1n}),
\]
where $M'$ is the $(m-1)\times (m-1)$ submatrix of $M$ obtained by deleting the first row and the first column, and $N'$ is the submatrix of $N$ obtained by deleting the first row. The nonvanishing of $H^{mn}_{I_1(MN)}(A)$ is now a consequence of the nonvanishing of $H^{mn-n}_{I_1(M'N')}(A)$, which in turn holds by Corollary~\ref{corollary:cd:determinantal}.

Similarly, the hypothesis $t <\min\{m,n\}$ in Theorem~\ref{theorem:torsion:free:generic}\,(2) is needed: If, for example, $t=m$, then $c=mn$, and it can be seen that the isomorphism
\[
H^c_\frakA(\ZZ[Y,Z])\ \cong\ H^{mt+tn}_\frakm(\QQ[Y,Z])
\]
does not hold by specializing $Y$ and $Z$ to the matrices $M$ and $N$ above. Lastly, if~$t=1$, then $H^c_\frakA(\ZZ[Y,Z])=H^{m+n-1}_\frakA(\ZZ[Y,Z])$ is not a $\QQ$-vectorspace since $\frakA=I_1(Y)\cap I_1(Z)$, and a Mayer--Vietoris sequence shows that the cokernel of
\[
\CD
H^{m+n-1}_\frakA(S) @>\cdot p>> H^{m+n-1}_\frakA(S)
\endCD
\]
is nonzero for $p$ any positive prime integer.

\end{remark}

\section{Local cohomology of symmetric determinantal nullcones}
\label{section:local:symmetric}

Let $X$ be a symmetric $n \times n$ matrix of indeterminates over a field $\KK$, and $I_{t+1}(X)$ the ideal generated by the size $t+1$ minors of $X$. The \emph{symmetric determinantal ring} $\KK[X]/I_{t+1}(X)$ is a subring of a polynomial ring: take $Y$ to be a $t\times n$ matrix of indeterminates, in which case the product matrix $Y^\tr Y$ has rank at most $t$, and the entrywise map yields an isomorphism
\[
\CD
\KK[X]/I_{t+1}(X) @>>> \KK[Y^\tr Y].
\endCD
\]
Hence the subring $R\colonequals \KK[Y^\tr Y]$ of~$S\colonequals \KK[Y]$ is isomorphic to $\KK[X]/I_{t+1}(X)$, and the isomorphism remains valid if the field $\KK$ is replaced by $\ZZ$.

The orthogonal group~$\Ort_t(\KK)$ acts $\KK$-linearly on $S$ via
\begin{equation}
\label{equation:action:orthogonal}
M\colon Y\mapsto MY\quad\text{ for }\ M\in\Ort_t(\KK).
\end{equation}
When the field $\KK$ is infinite, of characteristic other than $2$, the invariant ring is precisely the subring~$R$, see \cite{Weyl} or \cite[\S5]{DeConcini-Procesi}.

We use $\frakS$ to denote the ideal of $S$ generated by the entries of $Y^\tr Y$. Since $Y^\tr Y$ is symmetric, the ideal $\frakS$ has $\binom{n+1}{2}$ minimal generators; in the case $n\le (t+1)/2$, this is also the height of $\frakS$, see~\cite[Theorem~7.1]{HJPS}. Write $t$ as $2s$ or $2s+1$, for $s$ an integer. Suppose next that $n>(t+1)/2$. Then
\[
\height\frakS\ =\
\begin{cases}
ns-\displaystyle{\binom{s}{2}}& \text{if }\ t=2s,\\
ns+n-\displaystyle{\binom{s+1}{2}}& \text{if }\ t=2s+1.
\end{cases}
\]

Suppose $\KK$ is an algebraically closed field of characteristic other than two. If $t$ is odd, or if $n<s$, then $S/\rad\frakS$ is $F$-regular if $\KK$ has positive characteristic, and has rational singularities if $\KK$ has characteristic zero; see \cite[Proposition~5.4]{Lorincz:symplectic}. If $t=2s$ and $n\ge s$, then $\frakS$ has minimal primes $\frakP$ and $\frakQ$ of the same height, with each of the rings $S/\frakP$, $S/\frakQ$, $S/(\frakP+\frakQ)$ being $F$-regular/having rational singularities by \cite[Proposition~5.4]{Lorincz:symplectic}. Moreover,
\[
\height (\frakP+\frakQ)\ =\ ns+n+1-\binom{s+1}{2},
\]
see Theorems~7.2,~7.12, and~7.13 of~\cite{HJPS}. The ideals $\frakP$ and $\frakQ$ may be defined over the ring of Gaussian integers, see Definition~\ref{definition:P:Q}. 

In the case that $\KK$ has characteristic two, the ring $S/\rad\frakS$ is $F$-regular as well: it is isomorphic to a Pfaffian nullcone $\KK[Z]/(Z^\tr\Omega Z)$ if $t$ is odd, and to a polynomial extension of a Pfaffian nullcone if $t$ is even; see the proof of \cite[Theorem~7.2]{HJPS}.

Regarding the arithmetic rank of $\frakS$, we prove:

\begin{theorem}
\label{theorem:ara:symmetric}
Let $Y$ be a $t \times n$ matrix of indeterminates over a field $\KK$ of characteristic other than two. Then the arithmetic rank of the ideal $\frakS \colonequals I_1(Y^\tr Y)$ in $\KK[Y]$ is
\[
\binom{n+1}{2} - \binom{n+1-t}{2}.
\]
The ideal $\frakS$ is a set theoretic complete intersection if and only if $t=1$ or $n\le (t+1)/2$.
\end{theorem}

\begin{proof}
Let $X$ be a symmetric $n \times n$ matrix of indeterminates. The subring $R\colonequals\KK[Y^\tr Y]$ of~$S\colonequals\KK[Y]$ is isomorphic to $\KK[X]/I_{t+1}(X)$ that has dimension
\[
c\colonequals\binom{n+1}{2} - \binom{n+1-t}{2}.
\]
If $\KK$ has characteristic zero, then $c$ equals $\ara\frakS$ by Theorem~\ref{theorem:ara:intro} in view of the $\Ort_t(\KK)$ action~\eqref{equation:action:orthogonal}. More generally, a homogeneous system of parameters for $R$ generates $\frakS$ up to radical, so
\[
\ara\frakS\ \le\ c
\]
independently of the characteristic.

Assume next that $\KK$ has characteristic other than two; for the reverse inequality, it suffices to assume that $\KK$ is algebraically closed. If $n\ge t$, this inequality follows from Theorem~\ref{theorem:cohomology:open:symmetric} and Lemma~\ref{lemma:lsw}. If $t>n$, we need to verify that $c=\binom{n+1}{2}$ is a lower bound for~$\ara\frakS$. If the arithmetic rank were less than $c$, considering the specialization
\[
Y'\colonequals\begin{bmatrix}
y_{11} & \cdots & y_{1n}\\
\vdots & & \vdots\\
y_{n1} & \cdots & y_{nn}\\
0 & \cdots & 0\\
\vdots & & \vdots\\
0 & \cdots & 0\\
\end{bmatrix}
\]
yields a contradiction.

For the equivalent statements, we have already observed that the ideal $\frakS$ is generated by a regular sequence if $n\le (t+1)/2$, whereas, if $t=1$, then $\frakS$ is the square of the homogeneous maximal ideal of $\KK[Y]$. It only remains to verify that $\ara\frakS>\height\frakS$ outside of these cases.

If $t=2s$, the required verification is
\begin{equation}
\label{equation:symmetric:ara:even}
\binom{n+1}{2}-\binom{n+1-2s}{2}\ >\ ns-\binom{s}{2},
\end{equation}
equivalently,
\[
\binom{n+1}{2}-ns+\binom{s}{2}\ =\ \binom{n+1-s}{2}\ >\ \binom{n+1-2s}{2},
\]
which holds if $n\ge s+1$. If $t=2s+1$, we need to check that
\begin{equation}
\label{equation:symmetric:ara:odd}
\binom{n+1}{2}-\binom{n-2s}{2}\ >\ ns+n-\binom{s+1}{2},
\end{equation}
which holds since
\[
\binom{n+1}{2}-ns-n+\binom{s+1}{2}\ =\ \binom{n-s}{2}\ >\ \binom{n-2s}{2}
\]
as long as $n\ge s+2$ and $s\ge 1$.
\end{proof}

In Theorem~\ref{theorem:ara:symmetric}, one needs the hypothesis on the characteristic of the field:

\begin{remark}
\label{remark:characteristic:two}
Let $Y$ be a $t\times n$ matrix of indeterminates over a field $\KK$ of characteristic two, and set
$\frakS\colonequals I_1(Y^\tr Y)$. As observed in the proof of Theorem~\ref{theorem:ara:symmetric}, one has 
\[
\ara\frakS\ \le\ \binom{n+1}{2} - \binom{n+1-t}{2},
\]
though the inequality may be strict: for example, if $Y$ is $2\times n$, then $\rad\frakS$ is generated by the $n$ linear forms
\[
y_{11}+y_{21},\ y_{12}+y_{22},\ \dots,\ y_{1n}+y_{2n}.
\qedhere
\]
\end{remark}

\begin{definition}
\label{definition:P:Q}
Let $Y$ be a matrix of indeterminates of size $2s\times n$ over the ring of Gaussian integers $\ZZ[i]$, where $n\ge s$. For a subset $\alpha$ of the row indices, and a subset $\beta$ of the column indices, set~$Y_{\alpha|\beta}$ to be the submatrix of $Y$ with rows $\alpha$ and columns $\beta$. Let~$\alpha^\comp$ denote the complement of $\alpha$ in $\{1,\dots,2s\}$; set $\sgn(\alpha)$ to be the sign of the permutation that sends the~$2s$-tuple $(1,\dots,2s)$ to the $2s$-tuple $(\alpha,\alpha^\comp)$, where the entries of each of~$\alpha$ and~$\alpha^\comp$ are in ascending order.

Following \cite[\S7.3]{HJPS}, set $\frakP$ to be the ideal generated by $I_1(Y^\tr Y)$ and the polynomials
\[
\det(Y_{\alpha|\beta}) - i\,^s\sgn(\alpha)\det(Y_{\alpha^\comp|\beta}),
\]
for all subsets $\alpha\subseteq\{1,\dots,2s\}$ and $\beta\subseteq\{1,\dots,n\}$ of size $s$.

Similarly, set $\frakQ$ to be the ideal generated by $I_1(Y^\tr Y)$ and the polynomials
\[
\det(Y_{\alpha|\beta}) + i\,^s\sgn(\alpha)\det(Y_{\alpha^\comp|\beta}),
\]
for all $\alpha$ and $\beta$ as above.
\end{definition}

The ideals $\frakP$ and $\frakQ$ arise as the minimal primes of the ideal $I_1(Y^\tr Y)$. With this notation, we prove:

\begin{proposition}
\label{proposition:symmetric:no:torsion}
Let $Y$ be a $t \times n$ matrix of indeterminates over the ring of Gaussian integers $A\colonequals\ZZ[i]$. Set $S\colonequals A[Y]$ and $\frakS\colonequals I_1(Y^\tr Y)$.
\begin{enumerate}[\quad\rm(1)]
\item Then $H^k_\frakS(S)$ is a torsionfree $A$-module for each integer $k$. 
\item If $t=2s$ is even, and $n\ge s$, consider the ideals $\frakP$ and $\frakQ$ as above. Then $H^k_\frakP(S)$ and~$H^k_\frakQ(S)$ are torsionfree $A$-modules for each integer $k$.
\end{enumerate}
\end{proposition}

\begin{proof}
Let $\pi$ be a nonzero prime element of $A$. Since $S/\rad(\frakP+\pi S)$ and $S/\rad(\frakQ+\pi S)$ are $F$-regular, Theorem~\ref{theorem:F:rational:torsion} gives (2).

For (1), note that the ring $S/\rad(\frakS+\pi S)$ is $F$-regular in each of the cases (a) $\pi=1+i$; or (b) $t$ is odd; or (c) $t=2s$ and $n\le s-1$, so Theorem~\ref{theorem:F:rational:torsion} applies. The remaining case follows from Proposition~\ref{prop:MV:torsion}, bearing in mind that $S/(\frakP+\frakQ+\pi S)$ is $F$-regular.
\end{proof}

As with Pfaffian rings and determinantal rings, e.g., Remark~\ref{remark:pfaffian:integers}, the ring
\[
\ZZ[X]/I_{t+1}(X)\ \cong\ \ZZ[Y^\tr Y]
\]
has an ASL structure, see~\cite[\S5]{DeConcini-Procesi} or~\cite[\S3]{Barile}. Using this along with Theorem~\ref{theorem:vector:space:symmetric}\,(1), one obtains the following; the proof follows that of Corollary~\ref{corollary:cd:pfaffian}.

\begin{corollary}
\label{corollary:cd:symmetric}
Let $Y$ be a~$t\times n$ matrix of indeterminates over a torsionfree $\ZZ$-algebra $B$. Set $\frakS\colonequals I_1(Y^\tr Y)$ in $B[Y]$. Then
\[
\ara\frakS\ =\ \lcd\frakS\ =\ \binom{n+1}{2} - \binom{n+1-t}{2}.
\]
\end{corollary}

We next record a preliminary result towards studying $H^c_\frakS({\ZZ[Y]})$.

\begin{lemma}
\label{lemma:matrix:invert:symmetric}
Let $Y$ be a $t \times n$ matrix of indeterminates over an infinite field~$\KK$ of characteristic other than two. Set $\frakS$ to be the ideal~$I_1(Y^\tr Y)$ in $S\colonequals \KK[Y]$. If $t\ge 2$, then
\[
\ara\frakS S_{y_{11}}\ \le\ \binom{n+1}{2} - \binom{n+2-t}{2}.
\]
\end{lemma}

\begin{proof}
The assertion is immediate if $n\le t-1$, since the upper bound asserted in that case is simply $\binom{n+1}{2}$, which is the number of generators of $\frakS$. We assume $n\ge t$ henceforth.

The ideal $I_1(Y^\tr Y)S_{y_{11}}$ is unaffected by elementary column operations on $Y$ so, after renaming variables, we may take
\[
Y\colonequals\begin{bmatrix}
1 & 0 & 0 & \cdots & 0\\
y_{21} & y_{22} & y_{23} & \cdots & y_{2n}\\
\vdots & \vdots & \vdots & & \vdots\\
y_{t1} & y_{t2} & y_{t3} & \cdots & y_{tn}
\end{bmatrix}
\]
and work in affine space with coordinates $y_{ij}$ as above. After a change of notation, take $S$ to be the polynomial ring in the $(t-1)\times n$ indeterminates above, and set $\frakS\colonequals I_1(Y^\tr Y)S$. The task at hand is to prove that the arithmetic rank of $\frakS$ is bounded above by $\binom{n+1}{2} - \binom{n+2-t}{2}$.

If $t=2$, then $\frakS$ agrees up to radical with $(1+y_{21}^2,\ y_{22},\ y_{23},\ \dots,\ y_{2n})$, so the inequality holds. Assume $t\ge 3$. Let $Z$ be a $(t-2)\times (n-1)$ matrix of indeterminates over $\KK$. By Theorem~\ref{theorem:ara:symmetric}, the ideal generated by the entries of $Z^\tr Z$ has arithmetic rank
\[
\ell\colonequals \binom{n}{2} - \binom{n+2-t}{2}
\]
and, in fact, the proof shows that $\ell$ general $\KK$-linear combinations of the entries of $Z^\tr Z$ generate $I_1(Z^\tr Z)$ up to radical. With this in mind, let $\frakC$ denote the ideal of $S$ generated by the entries of the first row of $Y^\tr Y$, along with $\ell$ general linear combinations of the entries of the bottom right $(n-1)\times (n-1)$ submatrix of $Y^\tr Y$. Note that $\frakC$ has
\[
n+\ell\ =\ \binom{n+1}{2} - \binom{n+2-t}{2}
\]
generators, so it suffices to prove that $\frakS$ and $\frakC$ agree up to radical; for this, we replace the field $\KK$ by an algebraic closure, and use the Nullstellensatz as follows:

Consider a specialization
\[
M\colonequals
\begin{bmatrix}
1 & 0 & \cdots & 0\\
m_1 & m_2 & \cdots & m_n
\end{bmatrix}
\]
of $Y$ that belongs to the algebraic set defined by $\frakC$, where each $m_i$ is a size $t-1$ column vector. The condition that the first row of $M^\tr M$ is zero (as enforced on $M$ by the first $n$ generators of the ideal $\frakC$) reads
\[
1+m_1^\tr m_1=0, \quad m_2^\tr m_1=0, \quad \dots, \quad m_n^\tr m_1=0.
\]
Setting $\iota\colonequals\sqrt{-1}$ in $\KK$, the vector $\iota m_1$ satisfies $(\iota m_1)^\tr (\iota m_1)=1$, and may be taken as the first column of a matrix
\[
A\colonequals
\begin{bmatrix}
\iota m_1 & a_2 & \cdots & a_{t-1}
\end{bmatrix}
\]
in $\Ort_{t-1}(\KK)$, see Lemma~\ref{lemma:section:o}. Set
\[
B\colonequals
\begin{bmatrix}
1 & 0 \\
0 & A^\tr
\end{bmatrix}.
\]
Since $B$ is an orthogonal matrix, $\tilde{M}\colonequals BM$ satisfies $\tilde{M}^\tr\tilde{M}=M^\tr M$. It suffices to verify that~$\tilde{M}$ belongs to the algebraic set defined by $\frakS$, i.e., that $\tilde{M}^\tr\tilde{M}$ is zero, under the assumption that $\tilde{M}$ belongs to the algebraic set defined by $\frakC$.

The matrix $\tilde{M}$ has the form
\begin{multline*}
\tilde{M}\ =\ \begin{bmatrix}
1 & 0 \\
0 & \iota m_1^\tr\\
0 & a_2^\tr\\
\vdots & \vdots\\
0 & a_{t-1}^\tr
\end{bmatrix}
\begin{bmatrix}
1 & 0 & \cdots & 0\\
m_1 & m_2 & \cdots & m_n
\end{bmatrix}\\
\ =\
\begin{bmatrix}
1 & 0 & \cdots & 0\\
\iota m_1^\tr m_1 & \iota m_1^\tr m_2 & \cdots & \iota m_1^\tr m_n\\
a_2^\tr m_1 & a_2^\tr m_2 & \cdots & a_2^\tr m_n\\
\vdots & \vdots & & \vdots\\
a_{t-1}^\tr m_1 & a_{t-1}^\tr m_2 & \cdots & a_{t-1}^\tr m_n
\end{bmatrix}
\ =\
\begin{bmatrix}
1 & 0 & \cdots & 0\\
-\iota & 0 & \cdots & 0\\
0 & a_2^\tr m_2 & \cdots & a_2^\tr m_n\\
\vdots & \vdots & & \vdots\\
0 & a_{t-1}^\tr m_2 & \cdots & a_{t-1}^\tr m_n
\end{bmatrix}.
\end{multline*}
Setting $N$ to be the bottom right $(t-2)\times (n-1)$ submatrix of $\tilde{M}$, it follows that
\[
\tilde{M}^\tr\tilde{M}\ =\
\begin{bmatrix}
0 & 0 \\
0 & N^\tr N
\end{bmatrix}.
\]
Since the $\ell$ general linear combinations that constitute the second set of generators for $\frakC$ vanish on $\tilde{M}$, it follows that $N^\tr N$ is zero, and hence that $\tilde{M}^\tr\tilde{M}$ is zero.
\end{proof}

Corresponding to Theorems~\ref{theorem:torsion:free:pfaffian} and~\ref{theorem:torsion:free:generic}, we prove next:

\begin{theorem}
\label{theorem:vector:space:symmetric}
Let $Y$ be a $t \times n$ matrix of indeterminates, and set $\frakS\colonequals I_1(Y^\tr Y)$ in $\ZZ[Y]$. Set $c\colonequals \binom{n+1}{2} - \binom{n+1-t}{2}$, which is the cohomological dimension of the ideal $\frakS$. Then:
\begin{enumerate}[\quad\rm(1)]
\item $H^k_\frakS(\ZZ[Y])$ is a torsionfree $\ZZ$-module for each integer $k$.
\item If~$t\ge 3$ and~$n\ge(t+3)/2$, then $H^c_\frakS(\ZZ[Y])$ is a $\QQ$-vectorspace.
\item If~$n\ge t\ge 3$, then one has a degree-preserving isomorphism
\[
H^c_\frakS(\ZZ[Y])\ \cong\ H^{tn}_\frakm(\QQ[Y]),
\]
with $\frakm$ is the homogeneous maximal ideal of $\QQ[Y]$ under the standard grading.
\end{enumerate}
\end{theorem}

\begin{proof}
Assertion (1) follows readily from Proposition~\ref{proposition:symmetric:no:torsion}\,(1). Set $S\colonequals\ZZ[Y]$. For each positive prime integer $p$, one has an exact sequence
\[
\CD
@>>> H^c_\frakS(S) @>\cdot p>> H^c_\frakS(S) @>>> H^c_\frakS(S/pS) @>>>,
\endCD
\]
so (2) follows once we verify that $H^c_\frakS(S/pS)=0$. If $p=2$ or if~$t$ is odd, this follows using~\cite[Proposition~III.4.1]{PS} since $\rad\frakS$ is perfect and $c>\height\frakS$ under our assumptions using~\eqref{equation:symmetric:ara:even} and~\eqref{equation:symmetric:ara:odd}. The remaining case is where $p$ is odd and $t=2s$. Replace $S$ by $\ZZ[i][Y]$. For $\pi$ a Gaussian prime lying over $p$, consider the exact sequence
\[
\CD
@>>> H^c_\frakP(S/\pi S)\oplus H^c_\frakQ(S/\pi S) @>>> H^c_\frakS(S/\pi S) @>>> H^{c+1}_{\frakP+\frakQ}(S/\pi S) @>>>
\endCD
\]
The modules to the left vanish by~\cite[Proposition~III.4.1]{PS} since $\frakP(S/\pi S)$ and $\frakQ(S/\pi S)$ are perfect of height less than $c$. The remaining verification $c+1>\height(\frakP+\frakQ)(S/\pi S)$ is
\[
\binom{n+1}{2}-\binom{n+1-2s}{2}+1\ >\ ns+n+1-\binom{s+1}{2}.
\]
Rearranging terms, this is
\[
\binom{n-s}{2}\ >\ \binom{n+1-2s}{2}
\]
which holds since $s\ge 2$ and $n\ge s+2$.

For (3), first note that the hypotheses imply that $n\ge(t+3)/2$, so $H^c_\frakS(\ZZ[Y])$ is indeed a $\QQ$-vectorspace by (2). We change notation and work with $S\colonequals\QQ[Y]$, and show that $H^c_\frakS(S)$ has support~$\{\frakm\}$. For this, it suffices to verify that $H^c_\frakS(S)_{y_{ij}}$ vanishes. This follows from Lemma~\ref{lemma:matrix:invert:symmetric} since
\[
\ara\frakS S_{y_{ij}}\ \le\ \binom{n+1}{2} - \binom{n+2-t}{2}\ <\ \binom{n+1}{2} - \binom{n+1-t}{2}\ =\ c
\]
whenever $n\ge t$. The familiar $D$-module argument, e.g., from the proof of Theorem~\ref{theorem:torsion:free:pfaffian}\,(3), implies that $H^c_\frakS(S)$ is a direct sum of copies of $H^{tn}_\frakm(S)$, while the fact that it is exactly one copy follows from Theorem~\ref{theorem:cohomology:open:symmetric} and Lemma~\ref{lemma:lsw}.
\end{proof}

Lastly, we have the analogue of Theorem~\ref{theorem:vanish:pfaffian} and Theorem~\ref{theorem:vanish:determinantal}:

\begin{theorem}
\label{theorem:vanish:symmetric}
Let $M$ be a $t\times n$ matrix with entries from a commutative Noetherian ring~$A$, where $n\ge t\ge 3$. Set $\frakb\colonequals I_1(M^\tr M)$ and $c\colonequals \binom{n+1}{2} - \binom{n+1-t}{2}$. Then:
\begin{enumerate}[\quad\rm(1)]
\item The local cohomology module $H^c_\frakb(A)$ is a $\QQ$-vectorspace, and thus vanishes if the canonical homomorphism $\ZZ \to A$ is not injective.
\item If $\dim A\otimes_\ZZ\QQ<tn$, then $H^c_\frakb(A)=0$, i.e., $\lcd\frakb<c$.
\item If the images of the matrix entries in the ring $A\otimes_\ZZ\QQ$ are algebraically dependent over a field that is a subring of $A\otimes_\ZZ\QQ$, then $H^c_\frakb(A)=0$.
\end{enumerate}
\end{theorem}

\begin{proof}
Set $S\colonequals\ZZ[Y]$ for $Y$ a $t\times n$ matrix of indeterminates, and regard $A$ as an $S$-algebra via $y_{ij}\mapsto m_{ij}$, so that $\frakb$ equals $I_1(Y^\tr Y)A$. Using Theorem~\ref{theorem:vector:space:symmetric}\,(2) and base change along the map $S\to A$, one has
\[
H^c_\frakb(A)\ \cong\ H^{tn}_{\frakm A}(A\otimes_\ZZ\QQ),
\]
from which the assertions follow as in the earlier cases.
\end{proof}

\begin{remark}
The hypothesis $t\ge3$ in Theorem~\ref{theorem:vector:space:symmetric}\,(2) is essential: if $t=1$, then
\[
H^c_\frakS(\ZZ[Y])\ =\ H^n_{(y_{11},\dots,y_{1n})}(\ZZ[Y])
\]
is not a $\QQ$-vectorspace; if $t=2$ then $c=2n-1$, and we claim that multiplication by an odd prime $p$ is not surjective on
$H^{2n-1}_\frakS(\ZZ[Y])$. Considering the exact sequence
\[
\CD
@>>> H^{2n-1}_\frakS(S) @>\cdot p>> H^{2n-1}_\frakS(S) @>>> H^{2n-1}_\frakS(S/pS) @>>>0,
\endCD
\]
it suffices to verify that $H^{2n-1}_\frakS(S/pS)$ is nonzero; for simplicity, take $p\equiv 1\mod 4$, so that the ideals $\frakP$ and $\frakQ$ are defined in $S/pS$. One then has an exact sequence 
\[
\CD
@>>> H^{2n-1}_\frakS(S/pS) @>>> H^{2n}_{\frakP+\frakQ}(S/pS) @>>> H^{2n}_\frakP(S/pS)\oplus H^{2n}_\frakQ(S/pS) @>>>.
\endCD
\]
The ideals $\frakP(S/pS)$ and $\frakQ(S/pS)$ are perfect of height $n$, while $(\frakP+\frakQ)(S/pS)$ is perfect of height $2n$.

It follows from the above that Theorem~\ref{theorem:vanish:symmetric}\,(1) fails if $t\le 2$. We next observe next that Theorem~\ref{theorem:vanish:symmetric}\,(2) fails if $t\ge 3$ but $n<t$. For indeterminates $y_{ij}$ over $\QQ$, set
\[
M \colonequals \begin{bmatrix}
1 & 0 & \cdots & 0 \\
y_{21} & y_{22} & \cdots & y_{2n}\\
0 & y_{32} & \cdots & y_{3n}\\
\vdots & \vdots & & \vdots\\
0 & y_{t2} & \cdots & y_{tn}
\end{bmatrix}
\]
and $A\colonequals \QQ[M]$, in which case $\dim A<tn$. Set $\frakb\colonequals I_1(M^\tr M)$, and note that $c=\binom{n+1}{2}$ since $n<t$. We claim that $H^c_\frakb(A)$ is nonzero.

Setting $M'$ to be the $(t-2)\times(n-1)$ submatrix of $M$ obtained by deleting the first column and the first two rows, one has
\[
\frakb\ =\ I_1(M'^\tr M') + (1+y_{21}^2,\ y_{21}y_{22},\ \dots,\ y_{21}y_{2n}),
\]
and this agrees up to radical with
\[
I_1(M'^\tr M') + (1+y_{21}^2,\ y_{22},\ \dots,\ y_{2n}).
\]
It follows that $H^c_\frakb(A)$ is a direct sum of copies of
\[
H^{c-n}_{I_1(M'^\tr M')}(\QQ[M']),
\]
and this is nonzero by Corollary~\ref{corollary:cd:symmetric}.
\end{remark}

\section{Preliminaries on cohomology}
\label{section:prelim:cohomology}

\subsection{Notation and terminology}
\label{ssec:notation}

For our main results on arithmetic rank and structure of local cohomology modules, we will require computations of singular cohomology of complex varieties and \'etale cohomology of varieties over fields of positive characteristic. We denote by
\[
H^i_\sing(X,G) \quad\text{ and }\quad H^i_{c,\sing}(X,G)
\]
the singular cohomology and the compactly supported singular cohomology, respectively, of a complex quasiprojective variety $X$ in the analytic (Euclidean) topology, with coefficients in an Abelian group $G$, or more generally, a local system of Abelian groups $\calG$. We will implicitly use the identification of singular cohomology and sheaf cohomology with local coefficients for paracompact locally contractible spaces (e.g., quasiprojective complex varieties), cf.~\cite[Theorem~III.1.1]{Bredon1997} and \cite[\S2.5]{Dimca}. We refer the reader to \cite{Hatcher} and \cite{Dimca} as general sources on these notions, though we will review, in the next subsection, the basic facts that we use in the sequel.

Similarly, for an Abelian group $G$, and a quasiprojective variety $X$ over an algebraically closed field, we denote by
\[
H^i_\et(X,G) \quad\text{ and }\quad H^i_{c,\et}(X,G)
\]
the \'etale cohomology groups and compactly supported cohomology groups, respectively, with coefficients in the constant sheaf on $X$ given by $G$. More generally, we use the analogous notation with a local system of Abelian groups $\calG$ on $X$. We refer the reader to \cite{Milne} as a general source on these, and we will review in Subsection~\ref{ssec:etale} the basic facts that we use in the sequel.

We will be particularly interested in the first nonvanishing compact singular or compact \'etale cohomology groups of various spaces. We say that the \emph{singular compact dimension} of a space $X$ with coefficient group $G$ is
\[
\cptdim(X,G) \colonequals \inf\{ i \ge 0 \mid H^i_{c,\sing}(X,G)\neq 0\},
\]
sometimes abbreviated as $\cptdim X$. By the \emph{critical cohomology group}, we mean
\[
H^{\cptdim X}_{c,\sing}(X,G).
\]
We will use corresponding terminology and notation in the \'etale setting.

In many of our arguments, the computations of singular cohomology and \'etale cohomology are formally identical, and in these cases, we may drop the subscripts $\sing$ or $\et$. In particular, we will refer to \emph{Setting \AN} to mean that the ground field $\KK$ equals $\CC$, and that the varieties under consideration are quasiprojective complex varieties. Let $X$ be a given such variety. We set
\[
H^i(X)\colonequals H^i_\sing(X,\QQ) \quad\text{ and }\quad H^i_c(X)\colonequals H^i_{c,\sing}(X,\QQ).
\]
We use $\cptdim X$ for the singular compact dimension, taking coefficients in $\QQ$. The \emph{rank} of the critical cohomology group refers to the rank of $H^{\cptdim X}_{c,\sing}(X,\QQ)$ as a $\QQ$-vectorspace.

Similarly, we will refer to \emph{Setting \ET} to mean that the ground field $\KK$ is algebraically closed, of odd characteristic, and that the varieties under consideration are quasiprojective $\KK$-varieties. Let $X$ be a given such variety. We set
\[
H^i(X)\colonequals H^i_\et(X,\ZZ/2) \quad\text{ and }\quad H^i_c(X)\colonequals H^i_{c,\et}(X,\ZZ/2).
\]
We use $\cptdim X$ for the \'etale compact dimension with coefficients in $\ZZ/2$. The \emph{rank} of the critical cohomology group refers to the rank of~$H^{\cptdim X}_{c,\et}(X,\ZZ/2)$ as a $\ZZ/2$-vectorspace.

\subsection{Singular cohomology}
\label{ssec:singular}

We review some general facts about singular cohomology; we tailor our discussion to the settings used in the sequel, rather than record the most general statements.

\subsubsection*{Poincar\'e duality} \cite[Corollary~3.3.12]{Dimca}
Let $X$ be an $n$-dimensional real connected manifold, not necessarily compact, that is oriented with respect to a field $\LL$. Let $\calL$ be a locally constant sheaf of finite dimensional $\LL$-vectorspaces on $X$, and let $\calL^*$ denote the $\LL$-dual sheaf. Then, for each $i$, there is an isomorphism
\[
H^i_{c,\sing}(X,\calL)\ \cong\ H^{n-i}_\sing(X,\calL^*).
\]
These isomorphisms hold in particular when $X$ is a complex manifold and $\LL$ equals $\QQ$, since then $X$ is oriented by \cite[Example~3.2.10]{Dimca}; a particular case arises when $\calL$ is the constant sheaf with stalk $\QQ$, in which case $\calL^*=\calL$.

\subsubsection*{Mayer--Vietoris sequence} \cite[p.~103~and~185]{Iversen}
Let $X$ be a topological space with open subsets $U,V$ such that $U\cup V=X$. Let $\calF$ be a sheaf on $X$. Then there are long exact sequences of sheaf cohomology
\[
\to H^i(X, \calF) \to H^i(U, \calF) \oplus H^i(V, \calF) \to H^i(U \cap V, \calF) \to H^{i+1}(X, \calF) \to,
\]
and of compactly supported cohomology
\[
\to H^i_{c}(U \cap V, \calF) \to H^i_{c}(U, \calF) \oplus H^i_{c}(V, \calF) \to H^i_{c}(X, \calF) \to H^{i+1}_{c}(U \cap V, \calF) \to.
\]
In particular, if $X$ is a complex quasiprojective variety, and $\calL$ is a local system of $\QQ$-vectorspaces, we have the long exact sequences of singular cohomology
\begin{multline*}
\to H^i_\sing(X, \calL) \to H^i_\sing(U, \calL) \oplus H^i_\sing(V, \calL) \to H^i_\sing(U \cap V, \calL) \\
\to H^{i+1}_\sing(X, \calL) \to,
\end{multline*}
and of compactly supported singular cohomology
\begin{multline*}
\to H^i_{c,\sing}(U \cap V, \calL) \to H^i_{c,\sing}(U, \calL) \oplus H^i_{c,\sing}(V, \calL) \to H^i_{c,\sing}(X, \calL) \\
\to H^{i+1}_{c,\sing}(U \cap V, \calL) \to.
\end{multline*}

\subsubsection*{Long exact sequence of a subspace} \cite[p.~185]{Iversen}
Let $X$ be a topological space, and consider a triple $U\subseteq X\supseteq Z$, where $Z\subseteq X$ is a closed subspace, and $U=X\smallsetminus Z$. Then, for $G$ an Abelian group, there is a long exact sequence of compactly supported singular cohomology
\[
\CD
@>>> H^i_{c,\sing}(U,G) @>>> H^i_{c,\sing}(X,G) @>>> H^i_{c,\sing}(Z,G) @>>> H^{i+1}_{c,\sing}(U,G) @>>>.
\endCD
\]

\subsubsection*{Affine vanishing}
If $X$ is a smooth complex affine variety of algebraic dimension $d$, then $X$ is homotopy equivalent to a CW complex $Y$ of dimension $d$ \cite{Andreotti-Frankel}. It follows from this that for every locally constant sheaf $\calL$ of $\QQ$-vectorspaces on $X$, one has
\[
H^i_\sing(X,\calL)=0 \quad \text{ for }\ i>d.
\]
Indeed, if $f\colon Y\to X$ is the homotopy equivalence, then the pullback $f^{-1}(\calL)$ is a locally constant sheaf of $\QQ$-vectorspaces on $Y$, and
\[
H^i(Y,f^{-1}(\calL))\ \cong\ H^i(X,\calL) \quad \text{ for each $i$}
\]
by \cite[Remark~2.5.12]{Dimca}, while $H^i(Y,f^{-1}(\calL))=0$ for $i>d$ by \cite[Proposition~2.5.4]{Dimca}.

If $X$ is a smooth complex variety of algebraic dimension $d$ that admits an open cover by~$k$ affines, it follows from the vanishing above and the Mayer--Vietoris sequence that
\[
H^i_\sing(X,\calL)=0 \quad \text{ for }\ i>d+k-1.
\]
Combining this with Poincar\'e duality, under the same hypotheses we have
\[
H^i_{c,\sing}(X,\calL)=0 \quad \text{ for }\ i<d-k+1.
\]

\subsubsection*{Leray--Serre spectral sequence}
We say that
\[
\CD
F @>>> E @>>> B
\endCD
\]
is a \emph{locally trivial fiber bundle in the analytic topology} if there is a surjective map $E \stackrel{\pi}{\to} B$ and an open cover $U_1,\dots,U_t$ of $B$ in the analytic topology such that
\[
\pi|_{\pi^{-1}(U_i)}\colon {\pi^{-1}(U_i) \to U_i}
\]
is isomorphic to a projection $U_i \times F \to U_i$. Given such a locally trivial fiber bundle and a coefficient group $G$, there is a \emph{Leray--Serre spectral sequence}
\[
H^i_{c,\sing}(B,\calG^j) \Longrightarrow H^{i+j}_{c,\sing}(E,G),
\]
where $\calG^j$ is a locally constant sheaf on $B$ with stalk $H^j_{c,\sing}(F,G)$. The monodromy action on $\calG^j$ is induced by the monodromy action on $F$, i.e., the action of an element $[\gamma]$ of the fundamental group $\pi_1(B)$ on $\calG^j$ is the map on cohomology induced by the automorphism of $F$ given by lifting $\gamma$ to $E$.

\subsection{\'Etale cohomology}
\label{ssec:etale}
We discuss related statements for \'etale cohomology. Throughout this subsection, $\KK$ is an algebraically closed field, and $\ell$ a prime integer invertible in $\KK$.

\subsubsection*{Poincar\'e duality}\cite[VI.11.1]{Milne}
Let $X$ be a smooth quasiprojective variety over $\KK$, and let $\calL$ be a locally constant constructible sheaf of $\ZZ/\ell$-modules on $X$. Then there is a perfect pairing
\[
\CD
H^i_{c,\et}(X,\calL) \times H^{2d-i}_\et(X,\calL^*) @>>> \ZZ/\ell,
\endCD
\]
where $\calL^*=\Hom(\calL, \ZZ/\ell)$.

\subsubsection*{Mayer--Vietoris sequences}
Let $X$ be a quasiprojective variety over $\KK$, let $U$ and $V$ be Zariski open subsets of $X$ with $U\cup V = X$, and let $\calF$ be a sheaf of $\ZZ/\ell$-modules on $X$. Then there is a Mayer--Vietoris sequence
\[
\to H^i_\et(X, \calF) \to H^i_\et(U, \calF) \oplus H^i_\et(V, \calF) \to H^i_\et(U\cap V, \calF) \to H^{i+1}_\et(X, \calF) \to,
\]
see~\cite[III.2.24]{Milne}, and a Mayer--Vietoris sequence with compact supports
\begin{multline*}
\to H^i_{c,\et}(U \cap V, \calF) \to H^i_{c,\et}(U, \calF) \oplus H^i_{c,\et}(V, \calF) \to H^i_{c,\et}(X, \calF)\\
\to H^{i+1}_{c,\et}(U \cap V, \calF) \to.
\end{multline*}
In lieu of a reference, we give a brief argument for the Mayer--Vietoris sequence with compact supports. Write $u\colon U\to X$, $v\colon V\to X$, and $w\colon U\cap V\to X$ for the inclusion maps. Then $u_! u^{-1}(\calF)$ is the sheafification of the presheaf with sections on $W$ given by
\[
\begin{cases}
\Gamma(\calF,W) & \text{if }\ W\subseteq U, \\
0 & \text{if }\ W\not\subseteq U,
\end{cases}
\]
and likewise for $v_! v^{-1}(\calF)$ and $w_! w^{-1}(\calF)$. We then have a complex of sheaves
\begin{equation}
\label{equation:MV}
\CD
0 @>>> w_! w^{-1} \calF @>>> u_! u^{-1} \calF \oplus v_! v^{-1} \calF @>>> \calF @>>> 0.
\endCD
\end{equation}
For any geometric point $x\in X$, one has
\[
(u_! u^{-1} \calF)_x\ \cong\
\begin{cases}
\calF_x & \text{if }\ x\in U,\\
0 & \text{if }\ x\notin U,
\end{cases}
\]
and likewise for $v,w$. Thus, the sequence~\eqref{equation:MV} is exact, and the Mayer--Vietoris sequence is the long exact sequence obtained by applying $H^\bullet_{c,\et}(X,-)$, bearing in mind the isomorphisms $H^i_{c,\et}(X,u_! \calF) \cong H^i_{c,\et}(U,\calF)$ and their analogues for $v,w$.

\subsubsection*{Long exact sequence of a subspace}
Let $X$ be a variety over $\KK$, and $\calF$ a sheaf of $\ZZ/\ell$-modules on $X$. Consider a triple $U\subseteq X\supseteq Z$, where $Z\subseteq X$ is closed, and $U=X\smallsetminus Z$. Then there is a long exact sequence \cite[p.~94]{Milne}
\[
\CD
H^i_{c,\et}(U,\calF) @>>> H^i_{c,\et}(X,\calF) @>>> H^i_{c,\et}(Z,\calF) @>>> H^{i+1}_{c,\et}(U,\calF) @>>>.
\endCD
\]

\subsubsection*{Affine vanishing}
Let $X$ be an affine variety of dimension $d$ over an algebraically closed field $\KK$. According to \cite[VI.7.2]{Milne}, for any $\ell$-torsion sheaf $\calF$ with $\ell$ invertible in $\KK$,
\[
H^i_\et(X,\calF)=0 \quad \text{ for }\ i>d.
\]
It follows from this and the Mayer--Vietoris sequence, that if $X$ is a variety of dimension $d$ over an algebraically closed field, and admits an open cover by $k$ affines, then
\[
H^i_\et(X,\calF)=0 \quad \text{ for }\ i>d+k-1.
\]
By Poincar\'e duality, if $X$ is also smooth, it follows that for a locally constant invertible $\ZZ/\ell$ sheaf $\calL$, one has
\[
H^i_{c,\et}(X,\calL)=0 \quad \text{ for }\ i<d-k+1.
\]

\subsubsection*{Leray--Serre spectral sequence for compact supports}
We say that
\[
\CD
F @>>> E @>>> B
\endCD
\]
is a locally trivial fiber bundle in the \'etale topology if there is a surjective map $\pi\colon E\to B$ and an open cover $U_1,\dots,U_t$ of $B$ on the \'etale site such that the map $U_i \times_B E \to U_i$ obtained by base change is isomorphic to the projection map $U_i \times F \to U_i$.

In this setting, suppose that $\pi\colon E\to B$ is a morphism of quasiprojective varieties over~$\KK$. Then the functor $\pi_!$ exists by \cite[VI.3.3(e)]{Milne}. Let $\rho\colon B\to \Spec \KK$ be the constant map. By \cite[VI.3.2(c)]{Milne}, there is a spectral sequence
\[
R^i \rho_! \circ R^j \pi_! (\ZZ/\ell) \Longrightarrow R^{i+j} (\rho \pi)_! (\ZZ/\ell),
\]
which in this setting takes the form
\[
H^i_{c,\et}( B, R^j \pi_!(\ZZ/\ell)) \Longrightarrow H^{i+j}_{c,\et}(E,\ZZ/\ell).
\]
In a fibration, for each $j$, the sheaf $R^j \pi_! (\ZZ/\ell)$ is a locally constant constructible sheaf by \cite[VI.3.2(d)]{Milne}, with stalk $H^j_{c,\et}(F,\ZZ/\ell)$. Indeed, let $\{U_i\}$ be an open cover of $B$ such that
\begin{equation}
\label{equation:open:cover}
\begin{tikzcd}
E \arrow{r}{\ \ \pi} & B \\
U_i \times F \arrow{u} \arrow{r}{\! p} & U_i \arrow{u}
\end{tikzcd}
\end{equation}
commutes, where $p$ is the projection map. From the Cartesian diagram
\[
\begin{tikzcd}
U_i \times F \arrow{d} \arrow{r}{p} & U_i \arrow{d} \\
F \arrow{r} & \Spec \KK
\end{tikzcd}
\]
and the fact that $R^i \pi_!$ commutes with base change, \cite[VI.3.2(e)]{Milne}, we see that $R^j p_! (\ZZ/\ell)$ is the constant sheaf $H^j_{c,\et}(F,\ZZ/\ell)$ on $U_i \times F$. Applying the same fact to the Cartesian square~\eqref{equation:open:cover}, we obtain that
\[
R^j \pi_! (\ZZ/\ell)|_{U_i\times F}\ \cong\ R^j p_!(\ZZ/\ell).
\]
This shows the claim.

\subsection{Additional lemmas}

We record the main consequence of the Leray--Serre spectral sequence that will be used in the sequel.

\begin{lemma}
\label{lemma:LTFB}
Let $F,E,B$ be varieties over a field $\KK$. Assume that Setting \AN or \ET holds. Furthermore, we make the following assumptions:

\begin{enumerate}[\quad\rm(1)]
\item $F\to E\to B$ is a locally trivial fiber bundle;
\item the critical cohomology group of the fiber $F$ has rank one;
\item one of the following holds:
\begin{enumerate}[\rm(a)]
\item the base $B$ is simply connected with critical cohomology group of rank one, or
\item $B$ is smooth of dimension $b$ as an algebraic variety, is covered by $k$ affines where $\cptdim B=b-k+1$, and has critical cohomology group of rank one;
\end{enumerate}
\item in Setting \AN, the monodromy action of $\pi_1(B)$ on the critical cohomology group of $F$ is trivial (which is automatic when $B$ is simply connected).
\end{enumerate}
Then $\cptdim E = \cptdim F + \cptdim B$, and the critical cohomology of $E$ has rank one.
\end{lemma}

\begin{proof}
In either setting, the lemma is a consequence of the Leray--Serre spectral sequence; first consider Setting~\AN. If $B$ is simply connected, the spectral sequence takes the form
\[
H^i_{c,\sing}(B,H^j_{c,\sing}(F,\QQ)) \Longrightarrow H^{i+j}_{c,\sing}(E,\QQ),
\]
where $H^j_{c,\sing}(\QQ)$ denotes the constant system of $\QQ$-vectorspaces on $B$, with corresponding stalks. Thus, all terms on the $E_2$ page with $j<\cptdim F$ or $i<\cptdim B$ vanish, and the term with $i=\cptdim B$ and $j=\cptdim F$ is one copy of $\QQ$. The conclusion follows.

Next, suppose $B$ is smooth of dimension $b$ as an algebraic variety, is covered by~$k$ affines where $\cptdim B=b-k+1$, and that the monodromy action of $\pi_1(B)$ on the critical cohomology group of $F$ is trivial. By affine vanishing, we have $H^i_c(B, \calL)=0$ for any local system of $\QQ$-vectorspaces and any $i<b-k+1=\cptdim B$, so again all terms on the $E_2$ page with $j<\cptdim F$ or $i<\cptdim B$ vanish. The hypothesis on the monodromy action implies that the term with $i=\cptdim B$ and $j=\cptdim F$ is one copy of $\QQ$, and the conclusion follows.

In Setting~\ET, the argument is similar, using the analogous Leray--Serre spectral sequence and affine vanishing; the only difference is that the automorphism group of $\ZZ/2\ZZ$ is trivial, hence the monodromy action on $\ZZ/2\ZZ$ is necessarily trivial.
\end{proof}

\begin{lemma}[Filtrations and cohomology]
\label{lemma:les:induction}
We assume Setting \AN or \ET holds. Suppose $V_0\cup V_1\cup\cdots\cup V_t$ is a partition of the topological space $Y$ such that
\begin{enumerate}[\quad\rm(1)]
\item $V_i$ is open in $V_0\cup V_1\cup\cdots\cup V_i$ \ for $i=1,\dots,t$;
\item $\cptdim V_{i+1} - 1 > \cptdim V_i$ \ for $i=1,\dots,t-1$;
\item $\cptdim Y > \cptdim V_t + 1$.
\end{enumerate}
Then $\cptdim V_0 = \cptdim V_1 - 1$, and the critical cohomology groups of $V_0$ and $V_1$ are isomorphic.
\end{lemma}

\begin{proof}
Set $V_{\le i} \colonequals V_0 \cup V_1 \cup \cdots\cup V_i$. Set $d_i \colonequals \cptdim V_i$ and $d \colonequals \cptdim Y$. In particular, note that $V_{\le t} = Y$, and that $V_i$ is open in $V_{\le i}$ with closed complement $V_{\le{i-1}}$ for each $i\ge 1$.

We show by descending induction on $i=t$ that
\[
\cptdim V_{\le i-1} + 1\ =\ \cptdim V_{i},
\]
and that the corresponding critical cohomology groups are isomorphic. For $i=t$ we have
\[
\cptdim V_t + 1 \ <\ \cptdim Y
\]
by hypothesis, whence from the exact sequence for the triple $V_t \subseteq Y\supseteq V_{\le t-1}$ one sees that
\[
H^{d_t-1}_c(V_{\le t-1})\ \cong\ H^{d_t}_c(V_{t})
\]
and all lower cohomology groups of $V_{\le t-1}$ vanish. For $i=t-1,\dots,1$, the already established equality $\cptdim V_{\le i}=d_{i+1}-1$, along with the hypothesis, implies that
\[
\cptdim V_{\le i}\ =\ d_{i+1}-1\ >\ d_i.
\]
Using the exact sequence for the triple ${V_i \subseteq V_{\le i}}\supseteq V_{\le i-1}$, it then follows that
\[
H_c^{d_i-1}(V_{\le i-1})\ \cong\ H_c^{d_i}(V_i)
\]
and the lower cohomology groups vanish, completing the induction.

The assertion of the lemma is the case $i=1$.
\end{proof}

\section{Topology of Pfaffian nullcones}
\label{section:topology:pfaffian}

The purpose of this section is to prove:

\begin{theorem}
\label{theorem:cohomology:open:pfaffian}
Let $Y$ be a $2t\times n$ matrix of indeterminates over a field $\KK$, where $2t\le n$. Let~$\Omega$ denote the $2t\times 2t$ alternating matrix as in~\eqref{equation:omega}. Consider the algebraic set
\[
X^0_{2t\times n}\colonequals \Var(Y^\tr \Omega Y).
\]
\begin{enumerate}[\quad\rm(1)]
\item In the case $\KK$ equals $\CC$, we have
\[
H^i_\sing(\CC^{2t\times n}\smallsetminus X^0_{2t\times n},\, \QQ)\ =\
\begin{cases}
\QQ &\text{if }\ i = 4tn - \binom{2t+1}{2} - 1,\\
0 &\text{if }\ i > 4tn - \binom{2t+1}{2} - 1.
\end{cases}
\]
\item For $\KK$ an algebraically closed field of characteristic other than two, we have
\[
H^i_\et(\KK^{2t\times n}\smallsetminus X^0_{2t\times n},\, \ZZ/2)\ =\
\begin{cases}
\ZZ/2 &\text{if }\ i = 4tn - \binom{2t+1}{2} - 1,\\
0 &\text{if }\ i > 4tn - \binom{2t+1}{2} - 1.
\end{cases}
\]
\end{enumerate}
\end{theorem}

We study the cohomology of some auxiliary spaces; for integers $k\le t$, set:
\begin{equation}
\label{equation:pfaffian:pairs}
\begin{array}{r@{}l}
\Sp(2t,2k) &\ \colonequals\ \left\{M\in\KK^{2t \times 2k} \mid M^\tr\Omega_{2t} M = \Omega_{2k}\right\},\\
\\
\Alt(2k) &\ \colonequals\ \{M \in \KK^{2k\times 2k} \mid M \text{ is alternating and invertible}\},\\
\\
\Alt^{2k}_{n\times n} &\ \colonequals\ \{M \in \KK^{n\times n} \mid M \text{ is alternating and }\rank M = 2k\}.
\end{array}
\end{equation}

Note that $\Sp(2t,2t)$ is precisely the symplectic group $\Sp_{2t}$.

\begin{lemma}
\label{lemma:sp}
Consider positive integers $k\le t$. Then $\Sp(2t,2k)$ is a smooth affine variety and, in either Setting \AN or \ET,
\[
\cptdim\Sp(2t,2k)\ =\ \dim\Sp(2t,2k)\ =\ 4tk-\binom{2k}{2},
\]
with critical cohomology group of rank one.

Furthermore, in Setting \AN, the space $\Sp(2t,2k)$ is simply connected.
\end{lemma}

\begin{proof}
It is clear from the construction that $\Sp(2t,2k)$ is affine. Since $\Sp_{2t}$ acts transitively on $\Sp(2t,2k)$ by left multiplication, $\Sp(2t,2k)$ is smooth as well.

For the rest, we proceed by induction on $k$. For the base case $k=1$, we induce on $t\ge 1$ using the locally trivial fiber bundle~\eqref{eqn:sp:1-app}
\[
\CD
\KK^{2t-1} @>>> \Sp(2t,2) @>>> \KK^{2t} \smallsetminus \{0\}.
\endCD
\]
Note that~$\KK^{2t}\smallsetminus \{0\}$ is smooth of dimension $2t$, covered by $2t$ affines, and has compact dimension one with critical cohomology group of rank one. Moreover, in Setting \AN, the space~$\CC^{2t} \smallsetminus \{0\}$ is homotopy equivalent to the real sphere $\SS^{4t-1}$ and therefore simply connected since $t\ge 1$. Thus, Lemma~\ref{lemma:LTFB} applies, and we have
\[
\dim\Sp(2t,2)\ =\ \dim\KK^{2t-1} + \dim \KK^{2t} \smallsetminus \{0\}\ =\ 2t-1 + 2t\ =\ 4t-1,
\]
and
\[
\cptdim\Sp(2t,2)\ =\ \cptdim\KK^{2t-1} + \cptdim \KK^{2t} \smallsetminus \{0\}\ =\ 2(2t-1) + 1\ =\ 4t-1,
\]
and $\Sp(2t,2)$ has critical cohomology group of rank one. Furthermore, in Setting \AN, the homotopy exact sequence
\[
\CD
@>>> \pi_1(\CC^{2t-1}) @>>> \pi_1(\Sp(2t,2)) @>>> \pi_1(\CC^{2t} \smallsetminus \{0\}) @>>>
\endCD
\]
shows that $\Sp(2t,2)$ is simply connected, completing the case $k=1$.

Next, consider the locally trivial fiber bundle~\eqref{eqn:sp:2-app}
\[
\CD
\Sp(2t-2,2k-2) @>>> \Sp(2t,2k) @>>> \Sp(2t,2)
\endCD
\]
given by projection to the first column pair. By the case established above and the inductive hypothesis on $k$, the hypotheses of Lemma~\ref{lemma:LTFB} hold in both settings, so
\begin{multline*}
\dim \Sp(2t,2k)\ =\ \dim \Sp(2t-2,2k-2) + \dim \Sp(2t,2) \\
=\ (2t-2)(2k-2) - \binom{2k-2}{2} + 4t - 1\ =\ 4tk-\binom{2k}{2},
\end{multline*}
and likewise for compact dimension; moreover, $\Sp(2t,2k)$ has critical cohomology group of rank one. The homotopy exact sequence shows that $\Sp(2t,2k)$ is simply connected along the same lines as above.
\end{proof}

\begin{lemma}
\label{lemma:alt}
In either Setting \AN or \ET, the variety $\Alt(2k)$ is smooth affine, and
\[
\cptdim \Alt(2k)\ =\ \dim \Alt(2k)\ =\ \binom{2k}{2},
\]
with critical cohomology group of rank one.
\end{lemma}

\begin{proof}
First note that $\Alt(2k)$ is the complement of a hypersurface in the $\binom{2k}{2}$ dimensional affine space of alternating matrices, and thus is smooth and affine of the claimed dimension.

For the assertion on cohomology, we proceed by induction on $k$. In the case $k=1$, note that $\Alt(2)\cong \KK^\times$, so the assertion holds. For the inductive step, assuming the hypothesis for $\Alt(2k-2)$, first note that by the K\"unneth formula, we have
\begin{multline*}
\cptdim\big(\Alt(2k-2)\times\KK^{2k-2}\big)\ =\ \cptdim \Alt(2k-2)+ \cptdim \KK^{2k-2}\\
=\ \binom{2k-2}{2} + 4k-4\ =\ \binom{2k}{2}-1,
\end{multline*}
with critical cohomology group of rank one. By \cite[Lemma~1.3; Proof of Proposition~4.2]{Barile}, there is a locally trivial fiber bundle
\[
\CD
\Alt(2k-2) \times \KK^{2k-2} @>>> \Alt(2k) @>\pi>> \KK^{2k-1} \smallsetminus \{0\}.
\endCD
\]
But $\KK^{2k-1} \smallsetminus \{0\}$ is smooth and covered by $2k-1$ affines; use Lemma~\ref{lemma:LTFB}.
\end{proof}

\begin{lemma}
\label{lemma:alt:n}
Suppose $0<2k<n$. Then in either Setting \AN or \ET, the variety $\Alt^{2k}_{n\times n}$ is smooth, with compact dimension $\binom{2k}{2}$. In Setting \AN, $\Alt_{n\times n}^{2k}$ is simply connected.
\end{lemma}

\begin{proof}
By \cite[Proof of Theorem~4.1]{Barile}, we have a locally trivial fiber bundle
\begin{equation}
\label{equation:barile}
\CD
\Alt(2k) @>>> \Alt^{2k}_{n\times n} @>>> \Gr(n-2k,n),
\endCD
\end{equation}
where $\Gr(n-2k,n)$ is the Grassmannian parameterizing $(n-2k)$-dimensional subspaces of $\KK^n$. The base is simply connected, of compact dimension zero; the assertion regarding compact dimension follows from Lemma~\ref{lemma:alt} and Lemma~\ref{lemma:LTFB}.

We next examine the fundamental group of $\Alt(2k)$ in Setting \AN. When $k=1$, the space~$\Alt(2)$ is homeomorphic to $\CC^\times$, with fundamental group generated by the loop
\begin{equation}
\label{equation:loop:lambda}
\Lambda \colonequals \begin{bmatrix} 0 & \lambda\\ -\lambda & 0 \end{bmatrix}, \quad\text{ where $\lambda$ varies in $\SS^1$}.
\end{equation}
For $k>1$, as in the proof of Lemma~\ref{lemma:alt}, there is a locally trivial fiber bundle
\[
\CD
\Alt(2k-2) \times \CC^{2k-2} @>>> \Alt(2k) @>\pi>> \CC^{2k-1} \smallsetminus \{0\}.
\endCD
\]
Since $\pi_{2}(\CC^{2k-1} \smallsetminus \{0\})=\pi_{1}(\CC^{2k-1} \smallsetminus \{0\})$ is trivial, the homotopy exact sequence yields that the inclusion map
\[
\Alt(2k-2) \to \Alt(2k),\quad\text{ where }\ M \mapsto \begin{bmatrix} M & 0 \\ 0 & \Omega_2 \end{bmatrix},
\]
induces an isomorphism of fundamental groups
\[
\pi_1(\Alt(2k-2))\cong \pi_1( \Alt(2k)).
\]
In particular, the fundamental group of $\Alt(2k)$ is generated by the loop
\begin{equation}
\label{equation:generating:loop}
\begin{bmatrix}
\Lambda & 0 \\
0 & \Omega_{2k-2}
\end{bmatrix}.
\end{equation}
Similarly, since Grassmannians are simply connected, the locally trivial fiber bundle~\eqref{equation:barile} and the corresponding homotopy exact sequence yield a surjection
\[
\CD
\pi_1( \Alt(2k)) @>>> \pi_1(\Alt^{2k}_{n\times n}).
\endCD
\]
To show that $\Alt^{2k}_{n\times n}$ is simply connected, it suffices to show that the map above is zero, i.e., that the image of~\eqref{equation:generating:loop} in $\Alt^{2k}_{n\times n}$, namely, the loop given by $n\times n$ matrices
\[
L \colonequals
\begin{bmatrix} \Lambda & 0 & 0 \\
0 & \Omega_{2k-2} & 0 \\
0 & 0 & 0
\end{bmatrix},
\]
with $\Lambda$ as in~\eqref{equation:loop:lambda}, can be contracted in $\Alt^{2k}_{n\times n}$. Let $E$ be the $2\times (n-2k)$ matrix with $1$ as the top left entry, and zeros elsewhere. Within $\Alt^{2k}_{n\times n}$, one can continuously deform $L$ to the loop
\[
\begin{bmatrix} \Lambda & 0 & E \\
0 & \Omega_{2k-2} & 0 \\
-E^\tr & 0 & 0
\end{bmatrix},
\]
but this, in turn, deforms to the constant loop
\[
\begin{bmatrix} 0 & 0 & E \\
0 & \Omega_{2k-2} & 0 \\
-E^\tr & 0 & 0
\end{bmatrix}.
\]
It follows that $L$ is indeed contractible in $\Alt^{2k}_{n\times n}$.
\end{proof}

\begin{lemma}
Let $t$ be a positive integer. In either Setting \AN or \ET, one has
\[
\cptdim \GL_t\ =\ t^2,
\]
with critical cohomology group of rank one.
\end{lemma}

\begin{proof}
By \cite[Lemma~4 and 3.1]{Bruns-Schwanzl}, $H^i(\GL)$ vanishes for $i>t^2$ and has rank one for $i=t^2$. As $\GL_t$ is smooth, the claim for compact cohomology follows using Poincar\'e duality.
\end{proof}

The following auxiliary spaces will be used in our cohomology calculations:
\begin{equation}
\label{equation:pfaffian:aux}
\begin{array}{r@{}l}
X_{2t\times n}^{2k} &\ \colonequals\ \left\{M \in \KK^{2t\times n} \mid M^\tr \Omega_{2t} M \ \text{has rank } 2k\right\},\\
\\
G_{2t\times n}^{2k} &\ \colonequals\ \left\{M \in \KK^{2t\times n} \mid M^\tr \Omega_{2t} M = \begin{bmatrix} N & 0 \\ 0 & 0 \end{bmatrix} \text{ for } N\in \Alt(2k)\right\},\\
\\
F_{2t\times n}^{2k} &\ \colonequals\ \left\{M \in \KK^{2t\times n} \mid M^\tr \Omega_{2t} M = \begin{bmatrix} \Omega_{2k} & 0 \\ 0 & 0 \end{bmatrix}\right\}.
\end{array}
\end{equation}

\begin{theorem}
\label{theorem:compact:alternating}
Let $\KK$ be a field. Let $n,t,k$ be integers with $0\le 2k \le 2t \le n$. Then for $X^{2k}_{2t\times n}$ as above, the following hold:
\begin{enumerate}[\quad\rm(1)]
\item When $\KK$ equals $\CC$, we have
\[
H^i_{c,\sing}(X^{2k}_{2t\times n},\,\QQ)\ =\
\begin{cases}
\QQ & \text{if }\ i = \binom{2t+1}{2} + \binom{2k}{2},\\
0 & \text{if }\ i < \binom{2t+1}{2} + \binom{2k}{2}.
\end{cases}
\]
\item For $\KK$ an algebraically closed field of characteristic other than two, we have
\[
H^i_{c,\et}(X^{2k}_{2t\times n},\,\ZZ/2)\ =\
\begin{cases}
\ZZ/2 & \text{if }\ i = \binom{2t+1}{2} + \binom{2k}{2},\\
0 & \text{if }\ i < \binom{2t+1}{2} + \binom{2k}{2}.
\end{cases}
\]
\end{enumerate}
\end{theorem}

\begin{proof}
We first consider the case $k=t$ in both settings. We claim that
\[
X^{2t}_{2t\times n}\ =\ \{M \in \KK^{2t\times n} \mid \rank M = 2t\}.
\]
Indeed, if $\rank M < 2t$, then $\rank(M^\tr \Omega M) < 2t$. Conversely, if $\rank M = 2t$, then multiplication by $M$ is surjective, and multiplication by $M^\tr$ is injective, so $\rank (M^\tr \Omega M) = 2t$. The cohomology calculations now follow from \cite[Lemmas~2 and 2$^\prime$]{Bruns-Schwanzl}.

Next consider the case where $k=0$ and $t=1$. Note that $X^0_{2\times n}$ is closed in $\KK^{2\times n}$ with open complement $X^2_{2\times n}$, which was handled in the previous case. From the long exact sequence for an open subspace, we obtain
\[
\cptdim X^0_{2\times n}\ =\ \cptdim X^2_{2\times n}-1\ =\ 3,
\]
with the critical cohomology group of $X^0_{2\times n}$ having rank one. 
For the remaining cases, we proceed by induction on~$t$: fix~$t>1$, and assume that the claim holds for smaller values of~$t$. We have established the result for $k=t$; fix $k$ with $0<k<t$ and consider the locally trivial fiber bundle~\eqref{equation:ltfbs:alt:3}
\[
\CD
X^0_{(2t-2k) \times (n-2k)} @>>> F^{2k}_{2t\times n} @>>> \Sp(2t,2k).
\endCD
\]
By the inductive hypothesis on $t$ and Lemma~\ref{lemma:sp}, the hypotheses of Lemma~\ref{lemma:LTFB} apply, so
\begin{multline*}
\cptdim F^{2k}_{2t\times n}\ =\ \cptdim X^0_{(2t-2k) \times (n-2k)} + \cptdim \Sp(2t,2k) \\
=\ \binom{2t-2k+1}{2} + 4tk - \binom{2k}{2}\ =\ \binom{2t+1}{2},
\end{multline*}
with critical cohomology group of rank one.

At this stage, we proceed slightly differently in the two settings. In Setting \AN, we consider the locally trivial fiber bundle~\eqref{equation:ltfbs:alt:4}
\[
\CD
F^{2k}_{2t \times n} @>>> X^{2k}_{2t \times n} @>>> \Alt^{2k}_{n\times n}.
\endCD
\]
Since the base $\Alt^{2k}_{n\times n}$ is simply connected, we can apply Lemma~\ref{lemma:LTFB} and Lemma~\ref{lemma:alt:n} to conclude that
\[
\cptdim X^{2k}_{2t \times n}\ =\ \cptdim F^{2k}_{2t \times n} + \cptdim \Alt^{2k}_{n\times n}\ =\
\binom{2t+1}{2} + \binom{2k}{2},
\]
with critical cohomology group of rank one. This completes the case $t>1$ and $0<k<t$ in Setting \AN. In Setting \ET, we consider the locally trivial fiber bundle~\eqref{equation:ltfbs:alt:2}
\[
\CD
F^{2k}_{2t\times n} @>>> G^{2k}_{2t\times n} @>>> \Alt(2k).
\endCD
\]
By Lemma~\ref{lemma:alt}, the hypotheses of Lemma~\ref{lemma:LTFB} apply, and so
\[
\cptdim G^{2k}_{2t\times n}\ =\ \cptdim F^{2k}_{2t\times n} + \cptdim \Alt(2k)\ =\
\binom{2t+1}{2} + \binom{2k}{2},
\]
with critical cohomology group of rank one.

Next, consider the locally trivial fiber bundle~\eqref{equation:ltfbs:alt:1}
\[
\CD
G^{2k}_{2t\times n} @>>> X^{2k}_{2t\times n} @>>> \Gr(n-2k,n).
\endCD
\]
Since $\Gr(n-2k,n)$ is simply connected with compact dimension zero and critical cohomology group of rank one, we apply Lemma~\ref{lemma:LTFB} to obtain
\[
\cptdim X^{2k}_{2t\times n}\ =\ \cptdim G^{2k}_{2t\times n}\ =\ \binom{2t+1}{2} + \binom{2k}{2},
\]
with critical cohomology group of rank one, completing the case $t>1$ and $0<k<t$.

Finally, we deal with the case $k=0$ in both settings. For this, we apply Lemma~\ref{lemma:les:induction} to the partition
\[
\KK^{2t\times n}\ =\ X^0_{2t\times n} \cup X^2_{2t\times n} \cup \cdots \cup X^{2t}_{2t\times n}
\]
to conclude that
\[
\cptdim X^0_{2t\times n}\ =\ \cptdim X^2_{2t\times n} - 1\ =\ \binom{2t+1}{2} + \binom{2}{2} - 1\ =\
\binom{2t+1}{2},
\]
with critical cohomology group of rank one.
\end{proof}

\begin{proof}[Proof of Theorem~\ref{theorem:cohomology:open:pfaffian}]
Using the long exact sequence for a subspace, and the previous theorem, in the case $\KK$ equals $\CC$ we obtain
\[
H^{\binom{2t+1}{2}+1}_{c,\sing}(\CC^{2t\times n} \smallsetminus X^0_{2t\times n},\, \QQ)\ =\ \QQ,
\]
and the lower cohomology groups vanish. Since $\CC^{2t\times n} \smallsetminus X^0_{2t\times n}$ is a complex manifold (with boundary) of real dimension $4tn$, Poincar\'e duality gives
\[
H^{4tn - \binom{2t+1}{2}-1}_\sing(\CC^{2t\times n} \smallsetminus X^0_{2t\times n},\, \QQ)\ =\ \QQ,
\]
and the higher cohomology groups vanish.

Similarly, over an algebraically closed field of characteristic other than two, one has
\[
H^{\binom{2t+1}{2}+1}_{c,\et}(\KK^{2t\times n} \smallsetminus X^0_{2t\times n},\, \ZZ/2)\ =\ \ZZ/2,
\]
and the lower cohomology groups vanish; Poincar\'e duality gives the desired result.
\end{proof}

\section{Topology of generic determinantal nullcones}
\label{section:topology:determinantal}

The main goal of this section is to prove:

\begin{theorem}
\label{theorem:cohomology:open:generic}
Let $Y$ and $Z$ be matrices of indeterminates of size $m\times t$ and $t\times n$ respectively, over a field $\KK$, where $t\le\min\{m,n\}$. Consider the algebraic set
\[
X^0_{m,t,n} \colonequals \Var(YZ).
\]
\begin{enumerate}[\quad\rm(1)]
\item When $\KK$ equals $\CC$, we have
\[
H^i_\sing\big((\CC^{m\times t} \times \CC^{t\times n})\smallsetminus X^0_{m,t,n},\, \QQ\big)\ =\
\begin{cases}
\QQ &\text{if }\ i = 2mt + 2nt - t^2 - 1,\\
0 &\text{if }\ i > 2mt + 2nt - t^2 - 1.
\end{cases}
\]
\item For $\KK$ an algebraically closed field of characteristic other than two, we have
\[
H^i_\et\big((\KK^{m\times t} \times \KK^{t\times n})\smallsetminus X^0_{m,t,n},\, \ZZ/2\big)\ =\
\begin{cases}
\ZZ/2 &\text{if }\ i = 2mt + 2nt - t^2 - 1,\\
0 &\text{if }\ i > 2mt + 2nt - t^2 - 1.
\end{cases}
\]
\end{enumerate}
\end{theorem}

For integers $k\le t$, we examine the following auxiliary spaces:
\begin{equation}
\label{equation:determinantal:pairs}
\begin{array}{r@{}l}
\GL(t,k) &\ \colonequals\ \{M\in \KK^{t\times k} \mid \rank M=k\},\\
\\
\pairs(t,k) &\ \colonequals\ \{(A,B)\in \KK^{k\times t}\times \KK^{t\times k} \mid AB=\one_k\}.
\end{array}
\end{equation}

\begin{lemma}
\label{lemma:gl}
Let $k\le t$ be positive integers. In either Setting \AN or \ET, the variety $\GL(t,k)$ is smooth, with
\[
\dim\GL(t,k)\ =\ tk \quad \text{ and } \quad \cptdim\GL(t,k)\ =\ k^2,
\]
and critical cohomology group of rank one. Moreover, $\GL(t,k)$ is covered by $tk-k^2+1$ affine open sets. If $t>k$, then $\GL(t,k)$ is simply connected in Setting \AN.
\end{lemma}

\begin{proof}
The smoothness and dimension are immediate from the fact that $\GL(t,k)$ is an open subset of $\KK^{t\times k}$. The claim on cohomology follows from \cite[Lemma~2 and Lemma~2$^\prime$]{Bruns-Schwanzl} and Poincar\'e duality. The claim regarding the affine cover is \cite[Theorem~1(a)]{Bruns-Schwanzl}.

For $t>k$, there is a locally trivial fiber bundle~\eqref{equation:fib:gl}
\[
\CD
\KK^t\smallsetminus \KK^{k} @>>> \GL(t,k+1) @>>> \GL(t,k).
\endCD
\]
In Setting \AN, the fiber is the product of $\CC^k$ with $\CC^{t-k}\smallsetminus\{0\}$, and thus simply connected for $t-k\ge 2$. Fix $t\ge 2$, in which case $\GL(t,1)$ is simply connected; using the homotopy sequence, induction on $k$ shows that $\GL(t,k)$ is simply connected for $t>k$.
\end{proof}

\begin{lemma}
\label{lemma:pairs}
Let $k\le t$ be positive integers. In either Setting \AN or~\ET, the variety $\pairs(t,k)$ is smooth, and
\[
\cptdim \pairs(t,k)\ =\ \dim \pairs(t,k)\ =\ 2tk-k^2,
\]
with critical cohomology group of rank one. Furthermore, in Setting \AN, the space~$\pairs(t,k)$ is simply connected.
\end{lemma}

\begin{proof}
The space $\pairs(t,k)$ is affine by definition, and smoothness follows from the transitive $\GL_t$-action where $M$ maps $(A,B)$ to $(AM^{-1}, MB)$.

If $t=k$, then $\pairs(t,k)$ identifies with $\GL_k=\GL(k,k)$, and we are done by Lemma~\ref{lemma:gl}. For the rest of the proof assume that $t>k$, and proceed by induction on $k$.

For $k=1$, we have
\[
\pairs(t,1)\ =\ \Var(1-\sum_{i=1}^t y_i z_i)\ \subseteq\ \KK^t\times \KK^t.
\]
In Setting \AN, the space $\pairs(t,1)$ may be transformed to $\Var(1-\sum_1^{2t}x_i^2)$ by a linear change of coordinates. Suppose vectors $a,b\in\RR^{2t}$ are the real and imaginary part of a point in
\[
\Var(1-\sum_1^{2t} x_i^2)\ \subseteq\ \CC^{2t}.
\]
Setting $\left\Vert a\right\Vert \colonequals\sqrt{\sum a_i^2}$, one has $1+\left\Vert b\right\Vert^2 = \left\Vert a\right\Vert^2$, and $a$ is perpendicular to $b$. Setting $\iota\colonequals\sqrt{-1}$ as before, there is a diffeomorphism to the tangent bundle of the real sphere $\SS^{2t-1}$, given by
\begin{equation}
\label{equation:tangent:sphere}
\Var(1-\sum x_i^2) \to T\SS^{2t-1},\quad\text{ where }\ a+\iota b \mapsto (\frac{a}{\left\Vert a\right\Vert},b).
\end{equation}
Moreover, one has a locally trivial fiber bundle
\[
\CD
\RR^{2t-1} @>>> T\SS^{2t-1} @>>> \SS^{2t-1},
\endCD
\]
where $\SS^{2t-1}$ is simply connected, of compact dimension zero, with critical cohomology group of rank one; $\RR^{2t-1}$ is simply connected, of compact dimension $2t-1$, with critical cohomology group of rank one. Applying this to the diffeomorphism between $\pairs(t,1)$ and the tangent bundle of $\SS^{2t-1}$, one concludes via the Leray--Serre spectral sequence that $\cptdim\pairs(t,1)=2t-1$, with critical cohomology group of rank one. The corresponding conclusion holds in Setting \ET by \cite[Table~3.7]{SGA-ExposeXII}; confer, as well, \cite[V\'erification~3.8]{SGA-ExposeXII} for Setting \AN.

For $1<t<k$, consider the Zariski locally trivial fiber bundle~\eqref{equation:ltfb:pairs}
\[
\CD
\pairs(t-k+1,1) @>>> \pairs(t,k) @>>> \pairs(t,k-1).
\endCD
\]
Lemma~\ref{lemma:LTFB} applies by the induction hypothesis on $k$ in both settings, so
\[
\cptdim\pairs(t,k)\ =\ \cptdim\pairs(t,k-1) + \cptdim\pairs(t-k+1,1)\ =\ 2tk-k^2.
\]
The simply connectedness in Setting~\AN follows from the homotopy exact sequence.
\end{proof}

We define some auxiliary spaces that will be used in our main cohomology calculations:
\begin{equation}
\label{equation:determinantal:aux}
\begin{array}{r@{}l}
X^k_{m,t,n} &\ \colonequals\ \left\{(A,B)\in \KK^{m\times t} \times \KK^{t\times n} \mid AB \text{ has rank } k\right\},\\
\\
G^k_{m,t,n} &\ \colonequals\ \left\{(A,B)\in \KK^{m\times t} \times \KK^{t\times n} \mid \ker(AB)
= \image \begin{bmatrix} 0 \\ \one_{n-k} \end{bmatrix}\right\},\\
\\
F^k_{m,t,n} &\ \colonequals\ \left\{(A,B)\in \KK^{m\times t} \times \KK^{t\times n} \mid AB
= \begin{bmatrix} \one_k & 0 \\ 0 & 0 \end{bmatrix}\right\}.
\end{array}
\end{equation}

\begin{theorem}
\label{theorem:compact:generic}
Let $0\le k\le t\le m,n$ be integers, and $X^k_{m,t,n}$ as above.
\begin{enumerate}[\quad\rm(1)]
\item When $\KK$ equals $\CC$, we have
\[
H^i_{c,\sing}(X^k_{m,t,n},\, \QQ)\ =\
\begin{cases}
\QQ &\text{if }\ i=t^2 + k^2,\\
0 &\text{if }\ i<t^2 + k^2.
\end{cases}
\]
\item For $\KK$ an algebraically closed field of characteristic other than two, we have
\[
H^i_{c,\et}(X^k_{m,t,n},\, \ZZ/2)\ =\
\begin{cases}
\ZZ/2 &\text{if }\ i=t^2+k^2,\\
0 &\text{if }\ i<t^2+k^2.
\end{cases}
\]
\end{enumerate}
\end{theorem}

\begin{proof}
First consider the case $t=k$ in both Settings \AN and \ET. Then $(A,B)\in X_{m,t,n}^t$ if and only if $A,B$ both have rank $t$, so
\[
X^t_{m,t,n}\ \cong\ \GL(m,t) \times \GL(n,t),
\]
and thus $\cptdim X^t_{m,t,n} = t^2+t^2$, with critical cohomology group of rank one.

Now consider the case where $k=0$ and $t=1$. The space $X^0_{m,1,n}$ is the union of $\KK^m$ and~$\KK^n$ intersecting at a point. The Mayer--Vietoris sequence gives $\cptdim X^0_{m,1,n} = 1$, with critical cohomology group of rank one.

We now proceed by induction on $t$: fix $t>1$ and assume the claim holds for all smaller values of $t$. Fix $k$ with $0<k<t$ and consider the locally trivial fiber bundle~\eqref{equation:ltfbs:gen:3}
\[
\CD
X^0_{m-k,\,t-k,\,n-k} @>>> F^k_{m,t,n} @>>> \pairs(t,k).
\endCD
\]
By the inductive hypothesis and Lemma~\ref{lemma:pairs}, the hypotheses of Lemma~\ref{lemma:LTFB} are in force, and we deduce that
\[
\cptdim F^k_{m,t,n}\ =\ \cptdim X^0_{m-k,\,t-k,\,n-k} + \cptdim \pairs(t,k)\ =\ (t-k)^2 + 2tk-k^2\ =\ t^2.
\]
Next, consider the locally trivial fiber bundle~\eqref{equation:ltfbs:gen:2}
\[
\CD
F^k_{m,t,n} @>>> G^k_{m,t,n} @>>> \GL(m,k).
\endCD
\]
By Lemma~\ref{lemma:gl}, we can apply Lemma~\ref{lemma:LTFB} and deduce that
\[
\cptdim G^k_{m,t,n}\ =\ \cptdim F^0_{m,t,n} + \cptdim \GL(m,k)\ =\ t^2+k^2.
\]
Then, we have the locally trivial fiber bundle~\eqref{equation:ltfbs:gen:1}
\[
\CD
G^k_{m,t,n} @>>> X^k_{m,t,n} @>>> \Gr(n-k,n),
\endCD
\]
and by Lemma~\ref{lemma:LTFB} we deduce that
\[
\cptdim X^k_{m,t,n}\ =\ \cptdim G^k_{m,t,n} + \cptdim \Gr(n-k,n)\ =\ t^2+k^2.
\]
This completes the inductive step in the case $k>0$. Finally, we apply Lemma~\ref{lemma:les:induction} to complete the inductive step for $k=0$ as in the proof of Theorem~\ref{theorem:compact:alternating}.
\end{proof}

\begin{proof}[Proof of Theorem~\ref{theorem:cohomology:open:generic}]
This follows from Theorem~\ref{theorem:compact:generic}, along the same lines as the proof of Theorem~\ref{theorem:cohomology:open:pfaffian}.
\end{proof}

\section{Topology of symmetric determinantal nullcones}
\label{section:topology:symmetric}

The purpose of this section is to prove:

\begin{theorem}
\label{theorem:cohomology:open:symmetric}
Consider a $t\times n$ matrix of indeterminates $Y$ over a field $\KK$, where $t\le n$. Consider the algebraic set
\[
X^0_{t\times n} \colonequals \Var( Y^\tr Y).
\]
\begin{enumerate}[\quad\rm(1)]
\item When $\KK$ equals $\CC$, we have
\[
H^i_\sing(\CC^{t\times n}\smallsetminus X^0_{t \times n},\, \QQ)\ =\
\begin{cases}
\QQ &\text{if }\ i = 2tn - \binom{t}{2} - 1,\\
0 &\text{if }\ i > 2tn - \binom{t}{2} - 1.
\end{cases}
\]
\item For $\KK$ an algebraically closed field of characteristic other than two, we have
\[
H^i_\et(\KK^{t\times n}\smallsetminus X^0_{t \times n},\, \ZZ/2)\ =\
\begin{cases}
\ZZ/2 &\text{if }\ i = 2tn - \binom{t}{2} - 1,\\
0 &\text{if }\ i > 2tn - \binom{t}{2} - 1.
\end{cases}
\]
\end{enumerate}
\end{theorem}

We first examine some auxiliary spaces. For positive integers $k\le t$, define
\begin{equation}
\label{equation:orthogonal:pairs}
\begin{array}{r@{}l}
\Sym(k) &\ \colonequals\ \{M\in\KK^{k \times k} \mid M \text{ is symmetric and invertible}\},\\
\\
\Ort(t,k) &\ \colonequals\ \{M\in\KK^{t \times k} \mid M^\tr M = \one_k\}.
\end{array}
\end{equation}

\begin{lemma}
\label{lemma:ort}
Let $t\ge 2$ be an integer.
\begin{enumerate}[\quad\rm(1)]
\item In either Setting \AN or \ET, the variety $\Ort(t,1)$ is smooth affine and
\[
\cptdim\Ort(t,1)\ =\ \dim\Ort(t,1)\ =\ t-1,
\]
with critical cohomology group of rank one.
\item In Setting~\AN, the space $\Ort(t,1)$ is simply connected if $t\ge 3$.
\item In Setting~\AN, the negation map $v\colon \Ort(t,1) \to \Ort(t,1)$ given by $v(M)=-M$ induces the identity map on $H^{t-1}_{c,\sing}(\Ort(t,1),\QQ)$.
\end{enumerate}
\end{lemma}

\begin{proof}
We first consider Setting~\AN. Note that $\Ort(t,1)=\Var(1-\sum_1^t x_i^2)$ is diffeomorphic to the tangent bundle of the real sphere $\SS^{t-1}$ with the diffeomorphism, as in~\eqref{equation:tangent:sphere}, being
\[
\Ort(t,1) \to T\SS^{t-1}, \quad\text{ where }\quad
a+\iota b \mapsto (\frac{a}{\left\Vert a\right\Vert},b),
\]
for $a,b\in\RR^t$. The locally trivial fiber bundle
\[
\CD
\RR^{t-1} @>>> T\SS^{t-1} @>>> \SS^{t-1}
\endCD
\]
readily yields (1) and (2).

We now consider (3). Under the diffeomorphism, the negation map $\nu$ on $\Ort(t,1)$ corresponds to the map $\bar\nu$ on~$T\SS^{t-1}$ with $(a,b)\mapsto (-a,-b)$. Consider first the case $t=2$. For each positive integer~$r$, let $\bar\nu_r\colon \RR^2\to\RR^2$ denote the map that rotates a vector counterclockwise by~$\pi/r$. Then $\bar\nu=(\bar\nu_r)^r$, so the isomorphism on $H^{t-1}_{c,\sing}(\Ort(t,1),\ZZ)$ induced by $\bar\nu$ is the $r$-th power of the isomorphism induced by $\bar\nu_r$. The discreteness of $\ZZ$ forces $\bar\nu$ to be the identity on~$H^{t-1}_{c,\sing}(\Ort(t,1),\ZZ)$, and hence also on $H^{t-1}_{c,\sing}(\Ort(t,1),\QQ)$.

We proceed by induction on $t$. Assume $t\ge 3$, and choose $0<\eps\ll 1$. Set
\[
U_1\colonequals \{(a_1,\dots,a_t)\in\SS^{t-1}\mid -\eps<a_1\}
\quad\text{ and }\quad
U_2\colonequals \{(a_1,\dots,a_t)\in\SS^{t-1}\mid a_1<\eps\}.
\]
Set $U\colonequals U_1\cap U_2$, which is diffeomorphic to the cylinder $(-\eps,\eps)\times \SS^{t-2}$. For a vector $a\colonequals(a_1,\dots,a_t)$ in $\RR^t$, set $a'\colonequals(a_2,\dots,a_t)$. With this notation, there is a corresponding diffeomorphism of tangent bundles given by
\[
TU \to T(-\eps,\eps)\times T\SS^{t-2}, \quad\text{ where }\quad
(a,b) \mapsto (a_1,b_1)\times \Big(\frac{a'}{\left\Vert a'\right\Vert},\ b'-\frac{b'\cdot a'}{\left\Vert a'\right\Vert}a'\Big).
\]
Consider the Mayer--Vietoris sequence
\[
\CD
@>>> H^{t-1}_{c,\sing}(TU_1,\QQ) \oplus H^{t-1}_{c,\sing}(TU_2,\QQ) @>>> H^{t-1}_{c,\sing}(T\SS^{t-1},\QQ) @>\delta>> H^t_{c,\sing}(TU,\QQ) \\
@>>> H^t_{c,\sing}(TU_1,\QQ) \oplus H^t_{c,\sing}(TU_2,\QQ) @>>>
\endCD
\]
and note that $TU_1,TU_2$ are diffeomorphic to $\RR^{2t-2}$. The assumption $t\ge 3$ gives $2t-2>t$, so the outer groups are zero and $\delta$ is an isomorphism. The negation map $\nu$ on $\Ort(t,1)$ induces the negation map on $T\SS^{t-1}$, that restricts to $TU$, and corresponds with negation on each of $T(-\eps,\eps)$ and $T\SS^{t-2}$. But the negation map on~$T(-\eps,\eps)$ induces the trivial map on~$H^2_{c,\sing}(T(-\eps,\eps),\QQ)$, while the negation map on~$\Ort(t-1,1)$, equivalently~$T\SS^{t-2}$, induces the identity map on $H^{t-2}_{c,\sing}(\Ort(t,1),\QQ)$ by the inductive hypothesis.

Statement (1) in Setting~\ET follows from \cite[Table 3.7]{SGA-ExposeXII}.
\end{proof}

\begin{lemma}
\label{lemma:otk}
Let $k < t$ be positive integers.
\begin{enumerate}[\quad\rm(1)]
\item In either Setting \AN or \ET, the variety $\Ort(t,k)$ is smooth affine, and
\[
\cptdim \Ort(t,k)\ =\ \dim\Ort(t,k)\ =\ tk-\binom{k+1}{2},
\]
with critical cohomology group of rank one.
\item In Setting~\AN, the space $\Ort(t,k)$ is simply connected whenever $t-k>1$.
\item In Setting~\AN, the map $v\colon \Ort(t,k) \to \Ort(t,k)$ given by
\[
[w_1,w_2,\dots,w_k]\ \mapsto\ [-w_1,w_2,\dots,w_k]
\]
induces the identity map on $H^{tk-\binom{k+1}{2}}_{c,\sing}(\Ort(t,k),\QQ)$.
\end{enumerate}
\end{lemma}

\begin{proof}
The case $k=1$ is Lemma~\ref{lemma:ort}. For the general case, we proceed by induction on $k$, using the locally trivial fiber bundle
\begin{equation}
\label{equation:fiber:ort}
\CD
\Ort(t-1,k-1) @>>> \Ort(t,k) @>>> \Ort(t,1)
\endCD
\end{equation}
that arises from mapping an element of $\Ort(t,k)$ to its first column. Since $t>k>1$, the base $\Ort(t,1)$ is simply connected in Setting~\AN by Lemma~\ref{lemma:ort}, so the hypotheses of Lemma~\ref{lemma:LTFB} apply in both Settings~\AN and~\ET and (1) follows. Similarly, (2) follows inductively using the homotopy exact sequence.

For (3), note that the negation map $\nu$ on $\Ort(t,k)$ restricts to the negation map on $\Ort(t,1)$ under~\eqref{equation:fiber:ort}, and to the identity map on the fiber $\Ort(t-1,k-1)$. In particular, the restriction of $v$ induces the identity map on the critical cohomology group of $\Ort(t-1,k-1)$, while it also does so on the critical cohomology group of the base $\Ort(t,1)$ by Lemma~\ref{lemma:ort}. The assertion follows from the naturality of the Leray--Serre spectral sequence of~\eqref{equation:fiber:ort}.
\end{proof}

Our next goal is to compute the cohomology of the spaces $\Sym(k)$. We start with some preliminaries from linear algebra.

Recall that a square complex matrix $U$ is \emph{unitary} if the transpose of the conjugate is the inverse, i.e., $\bar U^\tr=U^{-1}$; a matrix $P$ is \emph{Hermitian} if its conjugate equals the transpose, i.e., $\bar P=P^\tr$. We will abbreviate $\overline{(-)^\tr}$ by $(-)^\star$ as is common, so $U$ is unitary if $U^\star=U^{-1}$, while $P$ is Hermitian if $P^\star=P$. A square matrix $B$ is \emph{normal} if $B^\star B=BB^\star$.

The Schur Decomposition Theorem states that a square complex matrix $A$ may be written as $UTU^{-1}$ for a unitary $U$ and upper triangular $T$. If $A$ is normal, then $T$ must be normal and thus diagonal: normal matrices are unitarily diagonalizable.

The Polar Decomposition Theorem (PDT) \cite[Section~2.5]{Hall2015} states that any $k\times k$ complex matrix $A$ can be written as
\[
P'U'\ =\ A\ =\ UP
\]
where $U,U'$ are unitary $k\times k$ matrices, and $P,P'$ are Hermitian positive semi-definite $k\times k$ matrices. Moreover, if $A$ is invertible, then $P,P'$ can be chosen to be positive-definite Hermitian, and the factorizations are then unique. We record a few observations:

(1) Whenever $P'U'=A=UP$ with invertible matrices as in the PDT, then $U'=U$. Indeed, $UP=(UPU^{-1})U$, and $UPU^{-1}$ is Hermitian and positive-definite whenever $P$ is:
\[
{(UPU^{-1})}^\star\ =\ UPU^{-1}
\]
shows that $UPU^{-1}$ is Hermitian, while for a nonzero vector $x\in\CC^k$ we have
\[
x^\star(UPU^{-1})x\ =\ (x^\star U)P(U^\star x)\ =\ {(U^\star x)}^\star P(U^\star x)\ > 0.
\]
The claim now follows from the uniqueness of the polar decomposition.

(2) If $P'U = A = UP$ is a PDT factorization in which $A$ is invertible and symmetric, then~$U$ is symmetric: use $P^\tr U^\tr=A$ along with the uniqueness. In light of (1), in this case we also have $UPU^{-1}=P^\tr$.

(3) Conversely, if $P$ is Hermitian positive-definite with $UPU^{-1}=P^\tr$ for $U$ unitary and symmetric, then $UP$ is symmetric and invertible:
\[
(UP)^\tr=P^\tr U^\tr=UPU^{-1}U=UP.
\]

(4) If $B$ is normal, then it has all eigenvalues on the unit circle $\SS^1$ precisely when it is unitary. Indeed, if we diagonalize $B=UDU^{-1}$ with a unitary matrix then
\[
B^\star B\ =\ {(UDU^{-1})}^\star (UDU^{-1})\ =\ UD^\star U^\star UDU^{-1}\ =\ UD^\star DU^{-1}
\]
is the identity precisely if the diagonal elements $\lambda_i$ of $D$ satisfy $\bar{\lambda_i}=1/\lambda_i$, which characterizes points on $\SS^1$.

Set
\[
\US(k)\colonequals \{U\in\Sym(k)\mid U\text{ is unitary}\}.
\]
By Lemma~\ref{lemma:sym:unitary:sqrt}, any unitary symmetric matrix $U$ has a Euclidean open neighborhood~$W$ where the squaring map has a section $\psi$. Let $\pi\colon\Sym(k)\to \US(k)$ be the map that associates to a symmetric matrix $A$, the unitary symmetric matrix $U$ in the Polar Decomposition Theorem $A=UP$, with $P$ positive-definite Hermitian. Then, by (2) above, for each~$U\in W$, the set $\pi^{-1}(U)$ consists of matrices $PU$ with $P$ positive-definite Hermitian such that $P^\tr =UPU^{-1}$. Denote this set by $\calP_U$.

We claim that $\calP_U$ is diffeomorphic to the space of positive-definite real symmetric matrices, and that the diffeomorphism is smooth in $U$. Indeed, write $V$ for the unitary symmetric matrix $\psi(U)$ and consider the matrix $P'\colonequals VPV^{-1}$, which is Hermitian positive-definite by (1). As $V$ is symmetric and unitary, $\bar V=V^{-1}$, and it follows that
\[
\bar{P'}\ =\ \bar{V}\, P^\tr\,{\bar{V}}^{-1}\ =\ \bar{V}(UPU^{-1}){\bar{V}}^{-1}\ =\
\bar{V}V^2PV^{-2}{(\bar{V})}^{-1}\ =\ VPV^{-1}\ =\ P',
\]
so $P'$ is real. Conversely, if $P'$ is positive-definite real symmetric, then $P=V^{-1}P'V$ is positive-definite symmetric by (1), and satisfies $\bar{P}=UPU^{-1}$. This proves the claim.

It follows that there is a locally trivial fiber bundle
\begin{equation}
\label{equation:ltfb:u}
\CD
\PosSymR(k) @>>> \Sym(k) @>>> \US(k),
\endCD
\end{equation}
with $\PosSymR(k)$ being the space of positive-definite real symmetric $k\times k$ matrices. The fiber is defined in the vectorspace of $k\times k$ symmetric real matrices by the open condition of having positive principal minors. It is clearly nonempty, and is convex from the characterization of positive-definite symmetric matrices $M$ as those that satisfy $x^\tr Mx>0$ for nonzero $x\in\RR^k$. Hence $\PosSymR(k)$ is a contractible $\binom{k+1}{2}$-dimensional real manifold.

\begin{lemma}
\label{lemma:sym}
Let $\KK$ be an algebraically closed field of characteristic other than two, and $k$ a positive integer. In either Setting \AN or \ET, the variety $\Sym(k)$ is smooth affine and
\[
\cptdim\Sym(k)\ =\ \dim\Sym(k)\ =\ \binom{k+1}{2},
\]
with critical cohomology group of rank one. Furthermore, in Setting~\AN, the fundamental group of $\Sym(k)$ is free of rank one, generated by the loop
\[
\begin{bmatrix} \lambda & 0\\ 0 & \one_{k-1} \end{bmatrix}
\quad\text{ where $\lambda$ varies in $\SS^1$}.
\]
\end{lemma}

\begin{proof}
In Setting~\ET, this follows from \cite[Proposition~3.7]{Barile} and Poincar\'e duality.

In Setting~\AN, we use the locally trivial fiber bundle~\eqref{equation:ltfb:u}. Since $\US(k)$ is a connected compact manifold, it has compact dimension zero, with critical cohomology group of rank one; from the earlier discussion, $\PosSymR(k)$ has compact dimension $\binom{k+1}{2}$, with critical cohomology group of rank one. The inclusion of $\US(k)$ in $\Sym(k)$ is a section of the projection in~\eqref{equation:ltfb:u}, so the monodromy action on the fiber is trivial. Thus, the claim on the cohomology follows from Lemma~\ref{lemma:LTFB}.

In order to compute the fundamental group in Setting \AN, let
\[
\Lambda\colon [0,1]\ \to\ \Sym(k)
\]
be a loop in $\Sym(k)$. Let $\Lambda_{i,j}(t)$ denote the entries of $\Lambda(t)$. Since $\CC$ has real dimension $2$, a generic change of coordinates ensures that $\Lambda_{1,1}(t)\neq 0$ for each $t$.

Conjugating $\lambda(t)$ by suitable matrices of the form
\[
\begin{bmatrix}1 & sv\\ sv^\tr & \one_{k-1}\end{bmatrix}
\quad\text{ where $0\le s\le 1$ and $v\in\CC^{k-1}$}
\]
gives a homotopy between $\Lambda$ and a loop $\Lambda'$ in which $\Lambda'_{i,1}(t)=0$ for all $i>1$. After scaling, we may also assume that $\Lambda_{1,1}(t)\in\SS^1$ for each $t$. Proceeding in this manner, $\Lambda$ is homotopic to a loop with diagonal matrices
\[
\begin{bmatrix}
\lambda^{\ell_1} \\
& \lambda^{\ell_2} \\
&&\ddots \\
&&& \lambda^{\ell_k}
\end{bmatrix},
\quad\text{ where $\lambda$ varies in $\SS^1$}.
\]
and $\ell_1,\dots\ell_k\in\ZZ$. To verify that the fundamental group of $\Sym(k)$ is cyclic, it suffices to show that for $k=2$ the loops
\[
\begin{bmatrix}\lambda & 0\\ 0 & 1 \end{bmatrix}
\text{ and }
\begin{bmatrix} 1 & 0\\ 0 & \lambda\end{bmatrix},
\quad\text{ where $\lambda$ varies in $\SS^1$},
\]
are homotopic. Consider the matrices
\[
\begin{bmatrix}
s+(1-s)\lambda & s(1-s)\lambda \\
s(1-s)\lambda & (1-s)+s\lambda
\end{bmatrix}
\]
for $\lambda\in\SS^1$ and $0\le s\le 1$. Taking $s$ to be $0$ and $1$, we obtain the loops in the preceding display, so it only remains to verify that these matrices are invertible, i.e., that the determinant
\[
\lambda^2 \Big(s(1-s)-s^2(1-s)^2\Big) + \lambda \Big((1-s)^2+s^2\Big) + s(1-s)
\]
is nonzero for all $0<s<1$ and $\lambda\in\SS^1$. Indeed, if the above is zero, one obtains
\[
\frac{1}{\lambda}\ =\
-\frac{(1-s)^2+s^2}{2s(1-s)} \pm \sqrt{\left(\frac{(1-s)^2+s^2}{2s(1-s)}\right)^2 - 1 +s(1-s)}.
\]
For $0<s<1$, it is readily seen that
\[
(1-s)^2+s^2\ \ge\ 2s(1-s),
\]
so $1/\lambda$ is real, and also negative, implying that $\lambda=-1$. But, in that case, the determinant~is
\[
-(2s-1)^2-s^2(1-s)^2,
\]
which is negative. It follows that the fundamental group of $\Sym(k)$ is cyclic, generated by the loop
\[
\begin{bmatrix} \lambda & 0\\ 0 & \one_{k-1} \end{bmatrix},
\quad\text{ where $\lambda$ varies in $\SS^1$}.
\]
This loop has infinite order in the fundamental group of $\GL_k(\CC)$, and hence in the fundamental group of $\Sym(k)$.
\end{proof}

The following auxiliary spaces will be used in the cohomology calculations:
\begin{equation}
\label{equation:symmetric:aux}
\begin{array}{r@{}l}
X_{t\times n}^k &\ \colonequals\ \{M\in \KK_{t\times n} \mid M^\tr M \ \text{has rank }\ k\},\\
\\
G_{t\times n}^k &\ \colonequals\ \left\{M\in \KK_{t\times n} \mid M^\tr M = \begin{bmatrix} N & 0 \\ 0 & 0\end{bmatrix} \ \text{with} \ N\in \Sym(k)\right\},\\
\\
F_{t\times n}^k &\ \colonequals\ \left\{M\in \KK_{t\times n} \mid M^\tr M = \begin{bmatrix} \one_k & 0 \\ 0 & 0\end{bmatrix} \right\}.
\end{array}
\end{equation}

With this notation, we prove:

\begin{theorem}
\label{theorem:compact:symmetric}
Let $\KK$ be a field and $0\le k\le t\le n$ be integers.
\begin{enumerate}[\quad\rm(1)]
\item When $\KK$ equals $\CC$, we have
\[
H^i_{c,\sing}(X^k_{t\times n},\, \QQ)\ =\
\begin{cases}
\QQ &\text{if }\ i = \binom{t}{2} + \binom{k+1}{2},\\
0 &\text{if }\ i < \binom{t}{2} + \binom{k+1}{2}.
\end{cases}
\]
\item For $\KK$ an algebraically closed field of characteristic other than two, we have
\[
H^i_{c,\et}(X^k_{t\times n},\, \ZZ/2)\ =\
\begin{cases}
\ZZ/2 &\text{if }\ i = \binom{t}{2} + \binom{k+1}{2},\\
0 &\text{if }\ i > \binom{t}{2} + \binom{k+1}{2}.
\end{cases}
\]
\end{enumerate}
\end{theorem}

\begin{proof}
We start with the case $t=k$. Then $X^t_{t\times n}$ is simply the set of $t\times n$ matrices of maximal rank, and the result follows from \cite[Lemmas~2 and~2$^\prime$]{Bruns-Schwanzl}.

For the case $k=0$ and $t=1$, note that $X^0_{1\times n}$ is simply the zero matrix, so the result holds. We proceed by induction on $t$: fix $t>1$ and assume that the claim holds for smaller values of~$t$. We have already established the result for $t=k$, so fix $k$ with $0<k<t$.

Consider the locally trivial fiber bundle~\eqref{equation:ltfbs:sym:3}
\[
\CD
X^0_{(t-k) \times (n-k)} @>>> F^{k}_{t\times n} @>>> \Ort(t,k).
\endCD
\]
If $k=t-1$, then $X^0_{(t-k) \times (n-k)}=\{0\}$ so $F^k_{t\times n} \cong \Ort(t,k)$. By Lemma~\ref{lemma:otk},
\[
\cptdim F^{t-1}_{t\times n}\ =\ \cptdim\Ort(t,t-1)\ =\ \binom{t}{2},
\]
with critical cohomology group of rank one. For $k<t-1$, by Lemma~\ref{lemma:otk}, the hypotheses of Lemma~\ref{lemma:LTFB} apply, and we deduce that
\[
\cptdim F^k_{t\times n}\ =\ \cptdim X^0_{(t-k) \times (n-k)} + \cptdim\Ort(t,k)\ =\ \binom{t-k}{2} + tk-\binom{k+1}{2}\ =\ \binom{t}{2},
\]
with critical cohomology group of rank one.

We now consider the locally trivial fiber bundle~\eqref{equation:ltfbs:sym:2}
\[
\CD
F^k_{t\times n} @>>> G^k_{t\times n} @>>> \Sym(k).
\endCD
\]
Lemma~\ref{lemma:sym} shows that the hypotheses of Lemma~\ref{lemma:LTFB} hold in Setting~\ET; in Setting~\AN, we must also verify that the monodromy action on $H_{c,\sing}^{\binom{t}{2}}(F^k_{t\times n},\QQ)$ is trivial. Consider the generator for $\pi_1(\Sym(k))$ from Lemma~\ref{lemma:sym}, namely
\[
\Lambda\colonequals\begin{bmatrix} \lambda & 0\\ 0 & \one_{k-1} \end{bmatrix}
\quad\text{ where $\lambda$ varies in $\SS^1$}.
\]
Under the map $G^k_{t\times n} \to \Sym(k)$, this has a lift
\[
\tilde{\Lambda}\colonequals
\begin{bmatrix} \lambda & 0 & 0\\ 0 & \one_{k-1} & 0\\ 0 & 0 & 0\end{bmatrix}
\ \mapsto\
\begin{bmatrix} \lambda^2 & 0\\ 0 & \one_{k-1} \end{bmatrix}
\]
in $G^k_{t\times n}$, with $\lambda$ now varying over the upper half of $\SS^1$. Thus, the monodromy action on the fiber $F^k_{t\times n}$ takes the form
\[
[w_1,w_2,\dots,w_k]\ \mapsto\ [-w_1,w_2,\dots,w_k].
\]
This map is compatible with the projection $F^k_{t\times n} \to \Ort(t,k)$ in~\eqref{equation:ltfbs:sym:3} and restricts to the identity map $X^0_{(t-k)\times(n-k)}$. By Lemma~\ref{lemma:otk}, the induced map on the critical cohomology group of $\Ort(t,k)$ is the identity map. It follows from the naturality of the Leray--Serre spectral sequence that the induced map on the critical cohomology group of $F^k_{t\times n}$ is the identity as well. By Lemma~\ref{lemma:LTFB}, we deduce that
\[
\cptdim G^k_{t\times n}\ =\ \cptdim F^k_{t\times n} + \cptdim\Sym(k)\ =\ \binom{t}{2} + \binom{k+1}{2},
\]
with critical cohomology group of rank one.

We now consider the locally trivial fiber bundle
\[
\CD
G^k_{t\times n} @>>> X^k_{t\times n} @>>> \Gr(n-k,n)
\endCD
\]
from Lemma~\ref{equation:ltfbs:sym:1}. Since $\Gr(n-k,n)$ is simply connected of compact dimension zero, we obtain that
\[
\cptdim X^k_{t\times n}\ =\ \cptdim G^k_{t\times n}\ =\ \binom{t}{2} + \binom{k+1}{2}
\]
with critical cohomology group of rank one.

The case $k=0$ follows by the previous cases using Lemma~\ref{lemma:les:induction} as in the proof of Theorem~\ref{theorem:compact:alternating}. This completes the induction on $t$, and the proof.
\end{proof}

\begin{proof}[Proof of Theorem~\ref{theorem:cohomology:open:symmetric}]
This follows from Theorem~\ref{theorem:compact:symmetric} along the same lines as the proof of Theorem~\ref{theorem:cohomology:open:pfaffian}.
\end{proof}

\appendix\section{Some locally trivial fiber bundles}
\label{appendix}

We justify here the locally trivial fiber bundles used in the previous sections; the main results are Lemmas~\ref{lemma:ltfbs:alt},~\ref{lemma:ltfbs:gen}, and~\ref{lemma:ltfb:sym}, addressing the Pfaffian, generic determinantal, and symmetric determinantal cases, respectively. In order to establish the local triviality of these fiber bundles, we collect a number of lemmas from linear algebra.

\subsection{The Pfaffian case}
Let $\KK$ be an algebraically closed field. The matrix $\Omega_{2t}$ from~\eqref{equation:omega} defines a symplectic bilinear form on $\KK^{2t}$ via
\[
\langle a, b \rangle \colonequals a^\tr \Omega_{2t} b.
\]
Note that $\langle a,a\rangle$ vanishes. Set
\[
a^\perp\colonequals\{b\in\KK^{2t}\mid\langle a,b\rangle=0\}.
\]

\begin{lemma}[Alternating Gram--Schmidt]
\label{lemma:section:sp}
For integers $0<k<t$, let
\[
\rho\colon \Sp(2t,2t) \to \Sp(2t,2k)
\]
be the map sending a matrix to its first $2k$ columns; see~\eqref{equation:pfaffian:pairs} for definitions. Then $\rho$ is surjective, and there exists a Zariski open cover of $\Sp(2t,2k)$ such that, for each open set~$U$ in the cover, the restriction $\rho^{-1}(U)\to U$ admits a section.
\end{lemma}

\begin{proof}
It suffices to show that for $k$ as above, the map $\Sp(2t,2k+2)\to \Sp(2t,2k)$ is surjective and admits sections on an open cover. Let $a_1,\dots, a_{2k}\in \KK^{2t}$ be such that
\[
A=\begin{bmatrix} a_1 & \cdots & a_{2k} \end{bmatrix}\ \in\ \Sp(2t,2k).
\]
Fix some nonzero $a_{2k+1}\in\KK^{2t}$ such that $\langle a_i, a_{2k+1} \rangle = 0$, and some $a'_{2k+2}\in\KK^{2t}$ such that $\langle a_{2k+1}, a'_{2k+2}\rangle\neq 0$; such vectors exist by a dimension count and the nondegeneracy of the bilinear form. After rescaling, one may assume that $\langle a_{2k+1},a'_{2k+2}\rangle = 1$. Setting
\[
a_{2k+2}\colonequals a'_{2k+2} - \sum_{i=1}^k \langle a'_{2k+2}, a_{2i} \rangle a_{2i-1} + \sum_{i=1}^k \langle a'_{2k+2}, a_{2i-1} \rangle a_{2i},
\]
one obtains a matrix
\[
B=\begin{bmatrix} a_1 & \cdots & a_{2k+2} \end{bmatrix}\ \in\ \Sp(2t,2k+2)
\]
that maps to $A$. 

To obtain a section on a neighborhood of $A$, let $x_1,\dots,x_{2k+2}$ be vectors of indeterminates 
denoting coordinates of $\Sp(2t,2k+2)$, with the first $2k$ vectors serving as coordinates for~$\Sp(2t,2k)$. Setting
\[
y_1 \colonequals a_{2k+1} - \sum_{i=1}^k \langle a_{2k+1}, x_{2i} \rangle x_{2i-1} + \sum_{i=1}^k \langle a_{2k+1}, x_{2i-1} \rangle x_{2i}
\]
\[
y_2 \colonequals a_{2k+2} - \sum_{i=1}^k \langle a_{2k+2}, x_{2i} \rangle x_{2i-1} + \sum_{i=1}^k \langle a_{2k+2} , x_{2i-1} \rangle x_{2i}
\]
one has $\langle y_1, x_j \rangle = \langle y_2, x_j \rangle = 0$ for $1\le j \le 2k$, and $f\colonequals \langle y_1, y_2 \rangle$ is a polynomial function in the coordinates of $\Sp(2t,2k)$. Since $f(A)=1$, the function $1/f$ is regular on an open neighborhood $U$ of $A$, and $x_{2k+1} \mapsto y_1$, $x_{2k+2} \mapsto (1/f) y_2$ determines a section on $U$.
\end{proof}

\begin{lemma}
For each positive integer $t$, there is a Zariski locally trivial fiber bundle
\begin{equation}
\label{eqn:sp:1-app}
\CD
\KK^{2t-1} @>>> \Sp(2t,2) @>\pi>> \KK^{2t} \smallsetminus \{0\},
\endCD
\end{equation}
where $\pi$ maps an element of $\Sp(2t,2)$ to its first column.
\end{lemma}

\begin{proof}
Let $e_1,\dots,e_{2t}$ denote the standard basis of $\KK^{2t}$, and let $u,v$ be the column vectors of a matrix in $\Sp(2t,2)$. For $1\le i\le t$, set $U_i$ and $U'_i$ to be the open subsets of $\KK^{2t} \smallsetminus \{0\}$ where $\langle u,e_{2i-1}\rangle\neq 0$ and $\langle u,e_{2i}\rangle\neq 0$, respectively. Identifying $\KK^{2t-1}$ with $e_{2i-1}^{\perp}$, one has an isomorphism
\[
\begin{array}{rcl}
\pi^{-1}(U_i) & \cong & U_i\times \KK^{2t-1}\\
{[u\ \ v]} & \mapsto & (u,\ v+\langle v,e_{2i-1}\rangle e_{2i}) \\
{\left[u\ \ v+\frac{1-\langle u,v\rangle}{\langle u,e_{2i}\rangle} e_{2i}\right]} & \mapsfrom & (u,\ v)
\end{array}
\]
and a similar isomorphism involving $U'_i$. The assertion follows.
\end{proof}

\begin{lemma}
For integers $1<k\le t$, there is a Zariski locally trivial fiber bundle
\begin{equation}
\label{eqn:sp:2-app}
\CD
\Sp(2t-2,2k-2) @>>> \Sp(2t,2k) @>\pi>> \Sp(2t,2),
\endCD
\end{equation}
where $\pi$ maps an element of $\Sp(2t,2k)$ to its first two columns.
\end{lemma}

\begin{proof}
Note that $\Sp_{2t}$ acts transitively on $\Sp(2t,2)$. By Lemma~\ref{lemma:section:sp}, $\Sp(2t,2)$ is covered by Zariski open sets $U$ on which $\Sp_{2t}\to\Sp(2t,2)$ admits a section; it suffices to show that $\pi^{-1}(U)$ is isomorphic to $U\times \Sp(2t-2,2k-2)$, compatibly with $\pi$.

Let $\alpha\colon U\to\Sp_{2t}$ be a section. For $M\in\pi^{-1}(U)$, set
\[
\beta(M)\colonequals\alpha(\pi(M))^{-1} M.
\]
Then $\beta(M) \in \Sp(2t,2k)$, and its first two columns coincide with those of $\one_{2k}$. Hence every other column of $\beta(M)$ has zeros in rows one and two. In particular, deleting the first two rows and columns of~$\beta(M)$ yields a matrix $M'\in\Sp(2t-2,2k-2)$. This provides an isomorphism
\[
\begin{array}{rcl}
\pi^{-1}(U) & \cong & U\times \Sp(2t-2,2k-2)\\
M & \mapsto & (\pi(M),\ M')\\
\alpha(A) \begin{bmatrix}\one_{2k} & 0\\ 0 & B \end{bmatrix} & \mapsfrom & (A,\ B).
\end{array}
\]
This shows that~\eqref{eqn:sp:2-app} is a Zariski locally trivial fiber bundle.
\end{proof}

\begin{lemma}[Jozefiak--Pragacz \cite{Jozefiak-Pragacz}]
\label{lemma:alt:block}
Let $U$ be the variety of $n\times n$ alternating matrices $M=(m_{ij})$ over $\KK$ with $m_{12}\neq 0$. Then there exists a morphism $\alpha\colon U\to \GL_n$ such that, for each $M\in U$, the matrix $\alpha(M)^\tr M \alpha(M)$ is alternating, with block form
\[
\begin{bmatrix} \Omega_2 & 0 \\ 0 & M' \end{bmatrix}.
\]
\end{lemma}

\begin{lemma}[Alternating roots]
\label{lemma:alt:root}
Consider the map
\[
\begin{array}{rcl}
\GL_{2k} & \stackrel{\mu}{\to} & \Alt(2k) \\
M & \mapsto & M^\tr \Omega_{2k} M.
\end{array}
\]
Then each element of $\Alt(2t)$ has a Zariski neighborhood on which $\mu$ admits a section.
\end{lemma}

\begin{proof}
For $k=1$, one can globally choose a section
\[
\begin{bmatrix} 0 & a\\ -a & 0\end{bmatrix}\ \mapsto\ \begin{bmatrix} a & 0\\ 0 & 1\end{bmatrix}.
\]
We proceed by induction on $k$. Given $A\in\Alt(2t)$ assume, for notational simplicity, that~$a_{12}\neq 0$. Then there is an open neighborhood $U$ of $A$ consisting of matrices $M$ with the property that $m_{12}\neq 0$. By Lemma~\ref{lemma:alt:block}, there exists $\alpha\colon U\to\GL_{2k}$ such that, for each~$M\in U$, the matrix $\alpha(M)^\tr M \alpha(M)$ has block form
\[
\begin{bmatrix} \Omega_2 & 0 \\ 0 & N \end{bmatrix}
\]
with $N\in\Alt(2t-2)$. Using the inductive hypothesis, replacing $U$ by a possibly smaller neighborhood, the result follows.
\end{proof}

We refer to~\eqref{equation:determinantal:pairs} for the notation in the elementary lemma below, the proof of which we leave for the reader.

\begin{lemma}
\label{lemma:GL:section}
Let $1\le k \le t$. Given $M \in \GL(t,k)$, there exists a Zariski open neighborhood $U$ of $M$ with a section $\alpha\colon U \to \GL_t$ of the projection map $\pi\colon \GL_t \to \GL(t,k)$.
\end{lemma}

We get to the main result of this subsection; see~\eqref{equation:pfaffian:pairs} and~\eqref{equation:pfaffian:aux} for definitions.

\begin{lemma}
\label{lemma:ltfbs:alt}
For integers $0\le 2k < 2t \le n$, each of the following is a Zariski locally trivial fiber bundle:
\begin{equation}
\label{equation:ltfbs:alt:1}
\CD
G^{2k}_{2t\times n} @>>> X^{2k}_{2t\times n} @>\pi>> \Gr(n-2k,n),
\endCD
\end{equation}
where $M$ in $X^{2k}_{2t\times n}$ maps to $\ker(M^\tr \Omega_{2t} M)$;

\begin{equation}
\label{equation:ltfbs:alt:2}
\CD
F^{2k}_{2t\times n} @>>> G^{2k}_{2t\times n} @>\pi>> \Alt(2k),
\endCD
\end{equation}
where an element $M$ of $G^{2k}_{2t\times n}$ maps to the top left $2k\times 2k$ submatrix of $M^\tr\Omega_{2t}M$;

\begin{equation}
\label{equation:ltfbs:alt:3}
\CD
X^0_{(2t-2k) \times (n-2k)} @>>> F^{2k}_{2t\times n} @>\pi>> \Sp(2t,2k),
\endCD
\end{equation}
where an element of $F^{2k}_{2t\times n}$ maps to its first $2k$ columns;

\begin{equation}
\label{equation:ltfbs:alt:4}
\CD
F^{2k}_{2t \times n} @>>> X^{2k}_{2t \times n} @>\pi>> \Alt^{2k}_{n\times n},
\endCD
\end{equation}
where $M\in X^{2k}_{2t \times n}$ maps to $M^\tr \Omega_{2t} M$.
\end{lemma}

\begin{proof}
\eqref{equation:ltfbs:alt:1} Let $\{e_1,\dots,e_n\}$ denote the standard basis for $\KK^n$, and denote by $K$ the subspace of $\KK^n$ spanned by $\{e_{2k+1},\dots,e_n\}$. Note that $G^{2k}_{2t \times n}$ is the set of matrices $M$ in~$X^{2k}_{2t \times n}$ with $\pi(M)=K$.

Consider the map $\alpha\colon \GL_n \to \Gr(n-2k,n)$ with $M\mapsto \image\begin{bmatrix} m_{2k+1} & \cdots & m_n\end{bmatrix}$, where $m_i$ denotes the $i$-th column of $M$. We claim that there is a Zariski open cover of $\Gr(n-2k,n)$ on which the map $\alpha$ admits sections. Indeed, we have that $\alpha$ is the composition of the projection map $\GL_n \to \GL(n,n-2k)$ from Lemma~\ref{lemma:GL:section} (up to permuting columns) and the map $\GL(n,n-2k) \to \Gr(n,n-2k)$ with $M\mapsto \image M$. The former admits local sections by Lemma~\ref{lemma:GL:section}, whereas the latter admits local sections by \cite[Lemma~6]{Bruns-Schwanzl}. This shows the claim.

Given a point in $\Gr(n-2k,n)$, fix an open neighborhood $U$ and a section $\beta\colon U \to \GL_n$ of the map $\alpha$, so $\beta(W)$ is invertible and $\beta(W)(K)=W$ for all $W\in U$. Note that for $N\in G^{2k}_{2t\times n}$ we have
\[
\begin{aligned}
\pi(N \beta(W)^{-1})\ &=\ \ker((\beta(W)^{\tr})^{-1} N^\tr \Omega N \beta(W)^{-1})\ =\ \ker(N^\tr \Omega N \beta(W)^{-1}) \\
\ & \ =\ \beta(W)\ker(N^\tr \Omega N)\ =\ \beta(W)(K)\ =\ W,
\end{aligned}
\]
and along similar lines one checks that $\ker(M \beta(\pi(M)))= K$. 

One then has an isomorphism
\[
\begin{array}{rcl}
\pi^{-1}(U) & \cong & U\times G^{2k}_{2t\times n}\\
M & \mapsto & (\pi(M),\ M \beta(\pi(M)))\\
N (\beta(W))^{-1} & \mapsfrom & (W,\ N).
\end{array}
\]

\eqref{equation:ltfbs:alt:2} By Lemma~\ref{lemma:alt:root}, the projection is surjective, and there is an open cover of $\Alt(2k)$ by sets $U$, with $\psi\colon U\to\GL_{2k}$ such that $\psi(A)^\tr \Omega_{2k} \psi(A) = A$. 
We have an isomorphism
\[
\begin{array}{rcl}
\pi^{-1}(U) & \cong & U \times F^{2k}_{2t\times n}\\
M \begin{bmatrix} \psi(A) & 0 \\ 0 & \one_{n-2k} \end{bmatrix} & \mapsfrom & (A,\ M)\\
M & \mapsto & \left(\pi(M),\ M \begin{bmatrix} \psi(\pi(M))^{-1} & 0 \\ 0 & \one_{n-2k}\end{bmatrix}\right).
\end{array}
\]

\eqref{equation:ltfbs:alt:3} For $\rho\colon \Sp(2t,2t) \to \Sp(2t,2k)$ as in Lemma~\ref{lemma:section:sp}, there is a covering of $\Sp(2t,2k)$ by open sets $U$
such that $\rho^{-1}(U)\to U$ admits a section $\alpha\colon U\to \Sp(2t,2t)$. The set $X^0_{(2t-2k)\times(n-2k)}$ may be identified with
\[
X' \colonequals \left\{\begin{bmatrix}\one_{2k} & 0 \\ 0 & N\end{bmatrix}
\mathrel{\Big|} N \in X^0_{(2t-2k) \times (n-2k)}\right\}. 
\]
We then have an isomorphism
\[
\begin{array}{rcl}
\pi^{-1}(U) & \cong & U \times X'\\
\alpha(A)M & \mapsfrom & (A,\ M) \\
M & \mapsto & (\pi(M),\ \alpha(\pi(M))^{-1}M).
\end{array}
\]

\eqref{equation:ltfbs:alt:4} By \cite[p.~77]{Barile}, there is a Zariski locally trivial fiber bundle
\[
\CD
\Alt(2k) @>>> \Alt^{2k}_{n\times n} @>>> \Gr(n-2k,n)
\endCD
\]
given by mapping a matrix $N\in \Alt^{2k}_{n\times n}$ to its kernel. Take a Zariski open cover $\{V_i\}$ of $\Gr(n-2k,n)$ on which this bundle and the bundle~\eqref{equation:ltfbs:alt:1} both trivialize, and let $U_i$ and $T_i$ be the preimages of $V_i$ in $\Alt^{2k}_{n\times n}$ and $X^{2k}_{2t \times n}$, respectively. We then have a commutative diagram of the form
\[
\begin{tikzcd}
T_i \arrow{r}\arrow{d}{\cong} & U_i \arrow{r}\arrow{d}{\cong} & V_i \\
G^{2k}_{2t\times n}\times V_i \arrow{r} & \Alt({2k}) \times V_i \arrow{ur}
\end{tikzcd}
\]
where the bottom map is the product of the projection in~\eqref{equation:ltfbs:alt:2} with the identity on $V_i$. Since~\eqref{equation:ltfbs:alt:2} is Zariski locally trivial, one can take an open cover of $\Alt(2k)$ on which the map $G^{2k}_{2t\times n} \to \Alt({2k})$ decomposes as a product with fiber $F^{2k}_{2t\times n}$; taking the preimage of this cover in each $U_i$ gives a cover of $\Alt^{2k}_{n\times n}$ on which~\eqref{equation:ltfbs:alt:4} decomposes as a product.\qedhere
\end{proof}

\subsection{The generic determinantal case}

For the notation for the next few lemmas, refer to~\eqref{equation:determinantal:pairs}. Throughout this section, $\KK$ denotes an algebraically closed field.

\begin{lemma}
For integers $1\le k\le t$, there is a Zariski locally trivial fiber bundle
\begin{equation}
\label{equation:ltfb:pairs}
\CD
\pairs(t-k+1,1) @>>> \pairs(t,k) @>{\pi}>> \pairs(t,k-1),
\endCD
\end{equation}
given by projecting $(A,B)$ to the top $k-1$ rows of $A$ and the left $k-1$ columns of $B$.
\end{lemma}

\begin{proof}
Note that an element $(A,B)$ in the fiber $F_0$ over the point
\[
\left(\begin{bmatrix} \one_{k-1} & 0\end{bmatrix},\ \begin{bmatrix} \one_{k-1} \\ 0 \end{bmatrix} \right)\ \in\ P(t,k-1)
\]
is determined by the last row of $A$ and last column of $B$, both of which have first $k-1$ entries zero. This induces an isomorphism $\psi\colon F_0 \to \pairs(t-k+1,1)$.

Now, by Lemma~\ref{lemma:GL:section}, given a matrix $M \in \GL(t,k-1)$, there exists an open neighborhood $U'$ of $M$ with a section $\alpha\colon U' \to \GL_t$. Let $U$ be the preimage of $U'$ under the map $(A,B) \mapsto B$ from $\pairs(t,k-1)\to \GL(t,k-1)$. Note that the collection of such $U$ forms an open cover of $\pairs(t,k-1)$. We then have an isomorphism
\[
\begin{array}{rcll}
\pi^{-1}(U) & \cong & U\times P(t-k+1,1)\\
(A,B) & \mapsto & \big((A',B') \colonequals \pi(A,B),\ \psi(A\alpha(B'), \alpha(B')^{-1} B)\big)\\
(\tilde{E}\alpha(D)^{-1},\ \alpha(D)\tilde{F}) & \mapsfrom & \big((C,D),\ (E,F)\big),
\end{array}
\]
where $(\tilde{E},\tilde{F})\colonequals \psi^{-1}(E,F)$.
\end{proof}

\begin{lemma}
\label{lemma:fib:Gl}
For integers $1\le k<t$, there is a Zariski locally trivial fiber bundle
\begin{equation}
\label{equation:fib:gl}
\CD
\KK^t\smallsetminus \KK^{k} @>>> \GL(t,k+1) @>{\pi}>> \GL(t,k)
\endCD
\end{equation}
that forgets the last column.
\end{lemma}

\begin{proof}
Note first that a point in the fiber $F_0$ over the point $\begin{bmatrix} \one_k & 0\end{bmatrix}$ is determined by the last column, and a column vector corresponds to a point in the fiber if and only if the last $t-k$ entries are not all zero; this gives an isomorphism $\psi\colon F_0 \to \KK^t \smallsetminus \KK^k$. 

By Lemma~\ref{lemma:GL:section}, given a matrix $M \in \GL(t,k)$, there exists an open neighborhood $U$ of~$M$ with a section $\alpha \colon U \to \GL_t$. We then have an isomorphism
\[
\begin{array}{rcl}
\pi^{-1}(U) & \cong & U\times (\KK^t \smallsetminus \KK^k)\\
A & \mapsto & \big(\pi(A),\ \psi(\alpha(\pi(A))^{-1} A)\big) \\
\alpha(B) \psi^{-1}(C) & \mapsfrom & (B,C).
\end{array}\qedhere
\]
\end{proof}

We prove the main result of the subsection; for definitions, see~\eqref{equation:determinantal:pairs} and~\eqref{equation:determinantal:aux}.

\begin{lemma}
\label{lemma:ltfbs:gen}
For integers $0\le k < t \le m,n$, each of the following is a Zariski locally trivial fiber bundle:
\begin{equation}
\label{equation:ltfbs:gen:1}
\CD
G^{k}_{m,t,n} @>>> X^k_{m,t,n} @>\pi>> \Gr(n-k,n),
\endCD
\end{equation}
where $(A,B)$ in $X^k_{m,t,n}$ maps to $\ker(AB)$;

\begin{equation}
\label{equation:ltfbs:gen:2}
\CD
F^{k}_{m,t,n} @>>> G^k_{m,t,n} @>\pi>> \GL(m,k),
\endCD
\end{equation}
where $(A,B)$ in $G^k_{m,t,n}$ maps to the left $k$ columns of $AB$;

\begin{equation}
\label{equation:ltfbs:gen:3}
\CD
X^0_{m-k,\,t-k,\,n-k} @>>> F^k_{m,t,n} @>\pi>> \pairs(t,k).
\endCD
\end{equation}
where $(A,B)$ in $F^k_{m,t,n}$ maps to the pair consisting of the top $k\times t$ submatrix of $A$ and the left $t\times k$ submatrix of $B$.
\end{lemma}

\begin{proof}
\eqref{equation:ltfbs:gen:1} One may follow along similar lines to~\eqref{equation:ltfbs:alt:1}; retaining the notation from there, with the sole change that $d=t$. We then have an isomorphism
\[
\begin{array}{rcl}
\pi^{-1}(U) & \cong & U\times G^{k}_{m,t, n}\\
A & \mapsto & \big(W\colonequals \ker( AB),\ (A,BM_W)\big)\\
(A,\ B'(M_W)^{-1}) & \mapsfrom & \big(W,\ (A,B')\big).
\end{array}
\]

\eqref{equation:ltfbs:gen:2} Given a matrix $M \in \GL(t,k)$, there exists an open neighborhood $U$ of $M$ with a section $\alpha\colon U \to \GL_t$ by Lemma~\ref{lemma:GL:section}. We then have an isomorphism
\[
\begin{array}{rcl}
\pi^{-1}(U) & \cong & U\times F^k_{m,t,n}\\
(A,\ B) & \mapsto & \big(C\colonequals \pi(A,B),\ (C^{-1}A,B)\big)\\
(CA,\ B) & \mapsfrom & \big(C,\ (A,B)\big).
\end{array}
\]

\eqref{equation:ltfbs:gen:3} By Lemma~\ref{lemma:GL:section}, given $M\in \GL(t,k)$, there is an open neighborhood $U'$ with a section $\alpha\colon U'\to \GL_t$. Let $U$ be the preimage of $U'$ under the second projection $\pairs(t,k)\to \GL(t,k)$. Note that the collection of such $U$ forms an open cover of $\pairs(t,k)$. Identify $X^0_{m-k,\,t-k,\,n-k}$ with 
\[
X'\colonequals \left\{\begin{bmatrix} \one_k & 0\\ 0 & N \end{bmatrix} \mathrel{\Big|} N\in X^0_{m-k,t-k,n-k} \right\},
\]
one has an isomorphism
\[
\begin{array}{rcl}
\pi^{-1}(U) & \cong & U\times X' \\
(A,B) & \mapsto & \big((A',B')\colonequals \pi(A,B),\ (A\alpha(B'),\alpha(B')B)\big)\\
(C\alpha(B)^{-1},\ \alpha(B) D) & \mapsfrom & \big((A,B),\ (C,D)\big).
\end{array}\qedhere
\]
\end{proof}

\subsection{The symmetric determinantal case}

Let $\KK$ be an algebraically closed field of characteristic other than two. We consider the standard inner product
\[
\langle a, b \rangle \colonequals a^\tr b,
\]
for $a,b$ in $\KK^t$.

\begin{lemma}[Gram--Schmidt]
\label{lemma:section:o}
For integers $0<k<t$, let
\[
\pi\colon \Ort(t,t) \to \Ort(t,k)
\]
be the map sending a matrix to its first $k$ columns; see~\eqref{equation:orthogonal:pairs} for the definitions. Then $\pi$ is surjective, and:
\begin{enumerate}[\quad\rm(1)]
\item In Setting~\AN, there exists a Euclidean open cover of $\Ort(t,k)$ such that, for each open set $U$ in the cover, the restriction $\pi^{-1}(U)\to U$ admits a section.

\item In Setting~\ET, there exists an \'etale cover of $\Ort(t,k)$ such that for each extension $U\to \Ort(t,k)$ in the cover, the base change of the projection map
\[
\CD
U \times_{\Ort(t,k)} {\Ort(t,t)} @>{U \times \pi}>> U
\endCD
\]
admits a section.
\end{enumerate}
\end{lemma}

\begin{proof}
In either setting, it suffices to show that the map $\Ort(t,k+1)\to \Ort(t,k)$ is surjective for $k$ as above, and admits sections on an open cover. 

Let $a_1,\dots, a_{k}\in \KK^{t}$ be such that
\[
A=\begin{bmatrix} a_1 & \cdots & a_k \end{bmatrix}\ \in\ \Ort(t,k).
\] 
Let $W\colonequals \image A$ in $\KK^t$. Given $v\in\KK^t$, note that
\[
v-\sum_{i=1}^t \langle a_i, v\rangle a_i\ \in\ W^\perp,
\]
so $\KK^t = W + W^\perp$. Thus, given a nonzero vector $w'\in W^\perp$, by the nondegeneracy of the inner product, there exists $w''\in W^\perp$ with $\langle w', w'' \rangle\neq 0$. It then follows from the identity
\[
2 \langle w', w''\rangle\ =\ \langle w'+w'', w'+w''\rangle - \langle w', w'\rangle - \langle w'', w''\rangle
\]
that there exists $w\in W^\perp$ with $\langle w, w\rangle\neq 0$. Setting $a_{k+1}=\lambda w$, for some square root $\lambda$ of $\langle w, w\rangle^{-1}$, one obtains a matrix 
\[
B=\begin{bmatrix} a_1 & \cdots & a_{k+1} \end{bmatrix}\ \in\ \Ort(t,k+1)
\]
mapping to $A$.

To obtain a section on an \'etale neighborhood of $A$, let $x_1,\dots,x_{k+1}$ be vectors of indeterminates denoting coordinates of $\Ort(t,k+1)$, with the first $k$ vectors serving as coordinates for~$\Ort(t,k)$. Setting 
\[
y \colonequals a_{k+1} - \sum_{i=1}^k \langle x_i, a_{k+1} \rangle x_i,
\]
one has $\langle y, x_i\rangle =0$ for $1\le i \le k$, and $f= \langle y, y\rangle$ is a polynomial function on $\Ort(t,k)$. Since~$f(A)=1$, the function $1/f$ is regular on an open neighborhood of $A$. Take~$g$ such that $g^2=1/f$ on an \'etale neighborhood of $A$ in Setting~\ET, or on an analytic neighborhood in Setting~\AN; then $x_{k+1} \mapsto g y$ determines a section on this neighborhood. 
\end{proof}

We require the following result of Micali and Villamayor.
\begin{lemma}{\cite[Lemme~2]{Micali-Villamayor}}
\label{lemma:sym:block}
Let $U$ denote the set of $n\times n$ symmetric matrices $A=[a_{ij}]$ over~$\KK$ with $a_{11}\neq 0$. In Setting~\ET, there is an \'etale morphism $\alpha\colon U\to \GL_n$ such that for each $A\in U$, the matrix $B\colonequals \alpha(A)^\tr A \alpha(A)$ is symmetric, and block decomposes~as
\[
B = \begin{bmatrix} 1 & 0 \\ 0 & A' \end{bmatrix}
\]
with $\mathrm{I}_{t}(A) = \mathrm{I}_{t-1}(A')$.
\end{lemma}

For notation in the following lemma, refer to~\eqref{equation:orthogonal:pairs}.

\begin{lemma}[Symmetric roots]
\label{lemma:sym:root}
Consider the map
\[
\begin{array}{rcl}
\GL_{t} & \stackrel{\mu}{\to} & \Sym(t) \\
M & \mapsto & M^\tr M.
\end{array}
\]
\begin{enumerate}[\quad\rm(1)]
\item In Setting~\AN, there exists a Euclidean open cover of $\Sym(t)$ such that, for each open set $U$ in the cover, the base change of the projection map
\[
\CD
U \times_{\Sym(t)} {\GL_t} @>{U\times \mu}>> U
\endCD
\]
admits a section.

\item In Setting~\ET, there exists an \'etale cover of $\Sym(t)$ such that for each extension $U \to \Sym(t)$ in the cover, the base change of the projection map
\[
\CD
U \times_{\Sym(t)} {\GL_t} @>{U\times \mu}>> U
\endCD
\]
admits a section.
\end{enumerate}
\end{lemma}

\begin{proof}
We start with Setting~\ET, where we proceed by induction on $t$. For $t=1$, note first that $\Sym(1) \cong \GL_1 \cong \KK^\times$, and under these isomorphisms, the map $\mu\colon \GL_1 \to \Sym(1)$ corresponds to the squaring map $\sigma\colon \KK^\times \to \KK^\times$. Since $\mathrm{char}(\KK)\neq 2$, the map $\sigma$ is itself an \'etale surjection. Let $U=\KK^\times \stackrel{\sigma}{\to} \KK^\times$ and consider the base change
\[
\CD
U\times_{\KK^\times} {\KK^\times} @>U\times \sigma>> U.
\endCD
\]
This corresponds to the ring map $\KK[x^{\pm 1}] \to \KK[x^{\pm 1}, y^{\pm 1}]/(x^2 - y^2)$, which splits via passing to the quotient by $(x-y)$.

Let $t>1$. Let $U_{1,1}$ be the Zariski open set in $\Sym(t)$ of matrices $A$ with $a_{1,1}\neq 0$. Lemma~\ref{lemma:sym:block} then reduces the issue to the case of a $(t-1)\times(t-1)$ matrix, where a section exists by the inductive hypothesis.

Given a matrix with $a_{1,1}=0$, let $U_{1,k}$ be the Zariski open set where $a_{1,k}\neq 0$. Let $E$ be the elementary operation that adds row $k$ to row~$1$. Then $EAE^\tr$ is in $U_{1,1}$ (since $2\neq 0$), and thus $E^{-1}\sigma E^{-\tr}$, with $\sigma$ the section found over $U_{1,1}$, gives the required section over the Zariski open set $E^{-1}U_{1,1}E^{-\tr}\cap U_{1,k}$ containing $A$.

In the analytic setting, one uses the analogue of Lemma~\ref{lemma:sym:block}.
\end{proof}

The following proof is a fleshed out version of \cite{Israel}.

\begin{lemma}[Unitary symmetric square roots]
\label{lemma:sym:unitary:sqrt}
The map from the set of unitary symmetric matrices to itself given by $U\mapsto U^2$ has Euclidean local sections.
\end{lemma}

\begin{proof}
We consider, for variables $r=\{r_1,\dots,r_k\}$, the rational function
\[
f_r(z)\colonequals \sum_{i=1}^k \frac{(z-r_i^2+r_i)\cdot\prod_{j\neq i}(z-r_j^2)}{\prod_{j\neq i} (r_i^2-r_j^2)}.
\]
Clearly, $f_r(z)$ has at worst poles at $r_i+r_{i'}$ and at $r_i-r_{i'}$, and our first claim on $f_r(z)$ is that only the former poles will occur. Indeed, the only summands where $r_i-r_{i'}$ is a pole are those of index $i$ and $i'$. However,
\[
\frac{(z-r_i^2+r_i)\cdot\prod\limits_{j\neq i}(z-r_j^2)}{\prod\limits_{j\neq i} (r_i^2-r_j^2)}\ +\
\frac{(z-r_{i'}^2+r_{i'})\cdot\prod\limits_{j\neq i'}(z-r_j^2)}{\prod\limits_{j\neq i'} (r_{i'}^2-r_j^2)}
\]
can---up to the factor $(r_i+r_{i'})$---be interpreted (reading $r_{i'}$ as $r_i+\Delta r_i$) as the difference quotient of $g_{\bar r}(z)$ where
\[
g_{\bar r}(z)=\frac{(z-r_i^2+r_i)\prod\limits_{i'\neq j\neq i}(z-r_j^2)}{\prod\limits_{i'\neq j\neq i}(r_i^2-r_j^2)}
\]
in the variables $z$ and $\bar r=\{r_1,\dots,r_{i'-1},r_{i'+1},\dots,r_k\}$. Since $g_{\bar r}(z)$ is differentiable, the claim follows.

We observe next, that $f_r(z)$ evaluates to $r_i$ at $r_i^2$. Indeed, setting $z=r_i^2$ wipes out all terms except term $i$, which returns $r_i$.

Choose a unitary symmetric $k\times k$ matrix $U_0$ and denote its eigenvalues $\lambda_1,\dots,\lambda_k$. Choose a ray $R$ emanating from the origin in $\CC$ and not containing any $\lambda_i$, and a section $\sqrt{-}$ of the square function on $\CC\smallsetminus R$. Note that $\sqrt{a}+\sqrt{b}=0$ is then impossible on $\CC\smallsetminus R$. For $\mu\in(\CC\smallsetminus R)^k$, let $f_{\sqrt{\mu}}(z)$ denote the function $f_r(z)$ from above, with parameters $\sqrt{\mu_1},\dots,\sqrt{\mu_k}$. Then the rational function
$f_{\sqrt\mu}(z)$ has no poles on $(\CC\smallsetminus R)^k\times\CC$; this follows from the discussion on poles of $f_r(z)$ above, in light of the fact that roots cannot sum to zero on $\CC\smallsetminus R$. In particular, for any fixed choice $\mu$ of the parameters,
$f_{\sqrt\mu(z)}$ is a well-defined polynomial that varies analytically with $\mu$.

We now consider for unitary symmetric $U$ with eigenvalues $\mu\in(\CC\smallsetminus R)^k$ the matrix $f_{\sqrt\mu}(U)$. As $f_{\sqrt\mu}(\mu_i)=\sqrt{\mu_i}$, the function $(f_{\sqrt\mu}(z))^2-z$ is zero at each $\mu_i$. If $U$ is unitary with eigenvalues $\mu$ all in $\CC\smallsetminus R$, then the minimal polynomial of $U$ divides $(f_{\sqrt\mu}(z))^2-z$ and thus $(f_{\sqrt \mu}(U))^2=U$. In particular, eigenvalues of $f_{\sqrt \mu}(U)$ are, like those of $U$, on the unit circle.

Thus, for unitary symmetric $U$, the matrix $f_{\sqrt \mu}(U)$ is symmetric (as $f_{\sqrt\mu}$ is a polynomial and $U$ symmetric), normal (as $U$ is normal, and $f_{\sqrt\mu}$ is a polynomial), unitary (since it is normal and has its eigenvalues on the unit circle). It follows that $U\mapsto f_{\sqrt \mu}(U)$ is an analytic section of the square function on unitary symmetric matrices with eigenvalues different from the intersection of $R$ with the unit circle.
\end{proof}

The following is the main result of the subsection; for definitions, see~\eqref{equation:orthogonal:pairs} and~\eqref{equation:symmetric:aux}.

\begin{lemma}
\label{lemma:ltfb:sym}
For integers $0\le k < t \le n$, each of the following is a Zariski locally trivial fiber bundle:
\begin{equation}
\label{equation:ltfbs:sym:1}
\CD
G^{k}_{t\times n} @>>> X^k_{t\times n} @>\pi>> \Gr(n-k,n),
\endCD
\end{equation}
sending $M\in X^k_{t\times n}$ to the kernel of $M^\tr M$;

\begin{equation}
\label{equation:ltfbs:sym:2}
\CD
F^{k}_{t\times n} @>>> G^{k}_{t\times n} @>\pi>> \Sym(k),
\endCD
\end{equation}
sending $M\in G^k_{t\times n}$ with $M^\tr M=\begin{bmatrix}A&0\\0&0\end{bmatrix}$ to $A\in\Sym(k)$;

\begin{equation}
\label{equation:ltfbs:sym:3}
\CD
X^0_{(t-k) \times (n-k)} @>>> F^{k}_{t\times n} @>\pi>> \Ort(t,k),
\endCD
\end{equation}
sending $M\in F^k_{t\times n}$ to the submatrix consisting of the left $k$ columns.
\end{lemma}

\begin{proof}
\eqref{equation:ltfbs:sym:1} The proof is parallel to that of~\eqref{equation:ltfbs:alt:1}, taking instead $d=k$ and replacing $\Omega_{2t}$ by $\one_t$.

\eqref{equation:ltfbs:sym:2} By Lemma~\ref{lemma:sym:root}, the projection is surjective, and there exists an open cover of~$\Sym(k)$ by sets $U$ with $\psi\colon U\to\GL_k$ such that $\psi(A)^\tr\psi(A)=A$; the sets are Euclidean open in Setting \AN, and \'etale open in Setting~\ET. We then have isomorphisms
\[
\begin{array}{rcl}
\pi^{-1}(U) & \cong & U \times F^{k}_{t\times n}\\
M \begin{bmatrix} \psi(A) & 0 \\ 0 & \one_{n-k} \end{bmatrix} & \mapsfrom & (A,\ M)\\
M & \mapsto & \left(\pi(M),\ M \begin{bmatrix} \psi(\pi(M))^{-1} & 0 \\ 0 & \one_{n-k}\end{bmatrix}\right).
\end{array}
\]

\eqref{equation:ltfbs:sym:3} By Lemma~\ref{lemma:section:o}, there is an open cover of $\Ort(t,k)$ by sets $U$ for which there is a section $\alpha: U \to U \times_{\Ort(t,k)} \Ort(t,t)$ of the projection; the sets $U$ are Euclidean open in Setting~\AN, and \'etale open in Setting~\ET.

The set $X^0_{(t-k)\times (n-k)}$ may be identified with
\[
X' \colonequals \left\{\begin{bmatrix}\one_k & 0 \\ 0 & N\end{bmatrix} \mathrel{\Big|} N \in X^0_{(t-k) \times (n-k)} \right\}.
\]
We then have isomorphisms
\[
\begin{array}{rcl}
\pi^{-1}(U) & \cong & U \times X'\\
\alpha(A)M & \mapsfrom & (A,\ M) \\
M & \mapsto & (\pi(M),\ \alpha(\pi(M))^{-1}M).
\end{array}
\]
This concludes the proof.
\end{proof}

\section*{Acknowledgments}

We are grateful to Ilya Smirnov for helpful comments.


\end{document}